\documentclass[11pt]{amsart}
\usepackage{amsfonts,latexsym,amsthm,amssymb,graphicx}
\usepackage[all]{xy}
\usepackage[usenames]{color}
\usepackage{stackengine}
\usepackage{tikz}
\usetikzlibrary{intersections,calc,arrows.meta}
\usepackage[enableskew]{youngtab}
\usepackage{tableau,ytableau}

\usepackage[hang,small,bf]{caption}
\usepackage[subrefformat=parens]{subcaption}
\allowdisplaybreaks

\usepackage{xypic,amscd,epsfig}
\usepackage{hyperref}
\usepackage[capitalize,nameinlink]{cleveref}

\DeclareFontFamily{OT1}{rsfs}{}
\DeclareFontShape{OT1}{rsfs}{n}{it}{<-> rsfs10}{}
\DeclareMathAlphabet{\mathscr}{OT1}{rsfs}{n}{it}

\CompileMatrices

\newtheorem{theorem}{Theorem}[section]
\newtheorem{lemma}[theorem]{Lemma}
\newtheorem{corol}[theorem]{Corollary}
\newtheorem{prop}[theorem]{Proposition}

\newtheorem{conj}{Conjecture}

{\theoremstyle{definition} 
	\newtheorem{defin}[theorem]{Definition}}
{\theoremstyle{remark} \newtheorem{remark}[theorem]{Remark}
\newtheorem{example}[theorem]{Example}}

\newcommand{\nc}{\newcommand}
\nc{\rnc}{\renewcommand}
\nc{\bb}[1]{{\mathbb #1}}
\nc{\bbA}{\bb{A}}\nc{\bbB}{\bb{B}}\nc{\bbC}{\bb{C}}\nc{\bbD}{\bb{D}}
\nc{\bbE}{\bb{E}}\nc{\bbF}{\bb{F}}\nc{\bbG}{\bb{G}}\nc{\bbH}{\bb{H}}
\nc{\bbI}{\bb{I}}\nc{\bbJ}{\bb{J}}\nc{\bbK}{\bb{K}}\nc{\bbL}{\bb{L}}
\nc{\bbM}{\bb{M}}\nc{\bbN}{\bb{N}}\nc{\bbO}{\bb{O}}\nc{\bbP}{\bb{P}}
\nc{\bbQ}{\bb{Q}}\nc{\bbR}{\bb{R}}\nc{\bbS}{\bb{S}}\nc{\bbT}{\bb{T}}
\nc{\bbU}{\bb{U}}\nc{\bbV}{\bb{V}}\nc{\bbW}{\bb{W}}\nc{\bbX}{\bb{X}}
\nc{\bbY}{\bb{Y}}\nc{\bbZ}{\bb{Z}}
\nc{\mbf}[1]{{\mathbf #1}}
\nc{\bfA}{\mbf{A}}\nc{\bfB}{\mbf{B}}\nc{\bfC}{\mbf{C}}\nc{\bfD}{\mbf{D}}
\nc{\bfE}{\mbf{E}}\nc{\bfF}{\mbf{F}}\nc{\bfG}{\mbf{G}}\nc{\bfH}{\mbf{H}}
\nc{\bfI}{\mbf{I}}\nc{\bfJ}{\mbf{J}}\nc{\bfK}{\mbf{K}}\nc{\bfL}{\mbf{L}}
\nc{\bfM}{\mbf{M}}\nc{\bfN}{\mbf{N}}\nc{\bfO}{\mbf{O}}\nc{\bfP}{\mbf{P}}
\nc{\bfQ}{\mbf{Q}}\nc{\bfR}{\mbf{R}}\nc{\bfS}{\mbf{S}}\nc{\bfT}{\mbf{T}}
\nc{\bfU}{\mbf{U}}\nc{\bfV}{\mbf{V}}\nc{\bfW}{\mbf{W}}\nc{\bfX}{\mbf{X}}
\nc{\bfY}{\mbf{Y}}\nc{\bfZ}{\mbf{Z}}
\nc{\bfa}{\mbf{a}}\nc{\bfb}{\mbf{b}}\nc{\bfc}{\mbf{c}}\nc{\bfd}{\mbf{d}}
\nc{\bfe}{\mbf{e}}\nc{\bff}{\mbf{f}}\nc{\bfg}{\mbf{g}}\nc{\bfh}{\mbf{h}}
\nc{\bfi}{\mbf{i}}\nc{\bfj}{\mbf{j}}\nc{\bfk}{\mbf{k}}\nc{\bfl}{\mbf{l}}
\nc{\bfm}{\mbf{m}}\nc{\bfn}{\mbf{n}}\nc{\bfo}{\mbf{o}}\nc{\bfp}{\mbf{p}}
\nc{\bfq}{\mbf{q}}\nc{\bfr}{\mbf{r}}\nc{\bfs}{\mbf{s}}\nc{\bft}{\mbf{t}}
\nc{\bfu}{\mbf{u}}\nc{\bfv}{\mbf{v}}\nc{\bfw}{\mbf{w}}\nc{\bfx}{\mbf{x}}
\nc{\bfy}{\mbf{y}}\nc{\bfz}{\mbf{z}}

\nc{\mcal}[1]{{\mathcal #1}}
\nc{\calA}{\mcal{A}}\nc{\calB}{\mcal{B}}\nc{\calC}{\mcal{C}}\nc{\calD}{\mcal{D}}
\nc{\calE}{\mcal{E}} \nc{\calF}{\mcal{F}}\nc{\calG}{\mcal{G}}\nc{\calH}{\mcal{H}}
\nc{\calI}{\mcal{I}}\nc{\calJ}{\mcal{J}}\nc{\calK}{\mcal{K}}\nc{\calL}{\mcal{L}}
\nc{\calM}{\mcal{M}}\nc{\calN}{\mcal{N}}\nc{\calO}{\mcal{O}}\nc{\calP}{\mcal{P}}
\nc{\calQ}{\mcal{Q}}\nc{\calR}{\mcal{R}}\nc{\calS}{\mcal{S}}\nc{\calT}{\mcal{T}}
\nc{\calU}{\mcal{U}}\nc{\calV}{\mcal{V}}\nc{\calW}{\mcal{W}}\nc{\calX}{\mcal{X}}
\nc{\calY}{\mcal{Y}}\nc{\calZ}{\mcal{Z}}
\nc{\fB}{\frak{B}}\nc{\fC}{\frak{C}} \nc{\fD}{\frak{D}}
\nc{\fE}{\frak{E}}\nc{\fF}{\frak{F}}\nc{\fG}{\frak{G}}\nc{\fH}{\frak{H}}
\nc{\fI}{\frak{I}}\nc{\fJ}{\frak{J}}\nc{\fK}{\frak{K}}\nc{\fL}{\frak{L}}
\nc{\fM}{\frak{M}}\nc{\fN}{\frak{N}}\nc{\fO}{\frak{O}}\nc{\fP}{\frak{P}}
\nc{\fQ}{\frak{Q}}\nc{\fR}{\frak{R}}\nc{\fS}{\frak{S}}\nc{\fT}{\frak{T}}
\nc{\fU}{\frak{U}}\nc{\fV}{\frak{V}}\nc{\fW}{\frak{W}}\nc{\fX}{\frak{X}}
\nc{\fY}{\frak{Y}}\nc{\fZ}{\frak{Z}}
\nc{\fa}{\frak{a}}\nc{\fb}{\frak{b}}\nc{\fc}{\frak{c}} \nc{\fd}{\frak{d}}
\nc{\fe}{\frak{e}}\nc{\fFf}{\frak{f}}\nc{\fg}{\frak{g}}\nc{\fh}{\frak{h}}
\nc{\fri}{\frak{i}}\nc{\fj}{\frak{j}}\nc{\fk}{\frak{k}}\nc{\fl}{\frak{l}}
\nc{\fm}{\frak{m}}\nc{\fn}{\frak{n}}\nc{\fo}{\frak{o}}\nc{\fp}{\frak{p}}
\nc{\fq}{\frak{q}}\nc{\fr}{\frak{r}}\nc{\fs}{\frak{s}}\nc{\ft}{\frak{t}}
\nc{\fu}{\frak{u}}\nc{\fv}{\frak{v}}\nc{\fw}{\frak{w}}\nc{\fx}{\frak{x}}
\nc{\fy}{\frak{y}}\nc{\fz}{\frak{z}}

\newcommand{\C}{{\mathbb C}}

\newcommand{\bZ}{{\mathbb Z}}
\newcommand{\caD}{{\mathcal D}}

\newcommand{\caF}{{\mathcal F}}

\newcommand{\cO}{{\mathcal O}}

\newcommand{\Fl}{\mathrm{Fl}}
\newcommand{\Gr}{\mathrm{Gr}}

\newcommand{\csm}{c_{\mathrm{SM}}}
\newcommand{\ssm}{s_{\mathrm{M}}}

\newcommand{\one}{1\hskip-3.5pt1}

\DeclareMathOperator{\id}{id}
\DeclareMathOperator{\rk}{rk}

\DeclareMathOperator{\stab}{stab}
\DeclareMathOperator{\Lie}{Lie}

\DeclareMathOperator{\Pic}{Pic}
\DeclareMathOperator{\Frac}{Frac}

\DeclareMathOperator{\End}{End}

\DeclareMathOperator{\Attr}{Attr}

\DeclareMathOperator{\Supp}{Supp}

\DeclareMathOperator{\aff}{aff}
\DeclareMathOperator{\SL}{SL}

\DeclareMathOperator{\Stab}{Stab}
\DeclareMathOperator{\loc}{loc}
\DeclareMathOperator{\Rep}{Rep}

\DeclareMathOperator{\pt}{pt}
\DeclareMathOperator{\fA}{A_\circ}
\DeclareMathOperator{\ext}{ext}
\DeclareMathOperator{\opp}{opp}

\begin{document}
\title[Chevalley formulae for motivic classes and stable envelopes]{Chevalley formulae for the motivic Chern classes of Schubert cells and for the stable envelopes}

\author{Leonardo C.~Mihalcea}
\address{ 225 Stanger St.,
460 McBryde Hall,
Department of Mathematics, 
Virginia Tech University, 
Blacksburg, VA 24061
USA
}
\email{lmihalce@vt.edu}

\author{Hiroshi Naruse}
\address{Graduate School of Education, University of Yamanashi, 
Kofu, 400-8510, Japan}
\email{hnaruse@yamanashi.ac.jp}

\author{Changjian Su}
\address{Yau Mathematical Sciences Center, Tsinghua University, Beijing, China}
\email{changjiansu@mail.tsinghua.edu.cn}

\thanks{L.~C.~Mihalcea was supported in part by the NSF grant DMS-2152294 and a Simons Collaboration Grant. H.~Naruse was supported in part by JSPS KAKENHI Grant Number 16H03921. C. Su is supported by the National Key R\&D Program of China (No. 2024YFA1014700).}

\date{\today}

\begin{abstract} We prove a Chevalley formula to multiply the motivic 
Chern classes of Schubert cells in a generalized flag manifold $G/P$ 
by the class of any line bundle $\mathcal{L}_\lambda$. Our formula 
is given in terms of the $\lambda$-chains of Lenart and Postnikov. 
Its proof relies on a change of basis formula in the affine Hecke algebra 
due to Ram, and on the Hecke algebra action on torus-equivariant 
K-theory of the complete flag manifold $G/B$ via left Demazure--Lusztig operators. 
We revisit some wall-crossing formulae for the stable envelopes in
$T^*(G/B)$. We use our Chevalley formula, and the equivalence 
between motivic Chern classes of Schubert cells and K-theoretic stable envelopes
in $T^*(G/B)$, to give formulae for the change of polarization, and 
for the change of slope for stable envelopes.
We prove several additional applications, including Serre, star, and Dynkin, 
dualities of the Chevalley coefficients, new formulae for the 
Whittaker functions, and for the Hall--Littlewood polynomials.
{We also discuss positivity properties of  
Chevalley coefficients, and properties of the coefficients arising from multiplication by minuscule weights.}
\end{abstract}

\maketitle

\tableofcontents

\section{Introduction}\label{S:intro}
Let $G$ be a complex, semisimple, Lie group and $T \subset B \subset P \subset G$ be a parabolic subgroup containing a Borel subgroup and the (standard) maximal torus.
Let $W$ be the Weyl group determined by $(G,T)$. In the study of cohomology and K-theory rings of (generalized) flag manifolds $G/P$, the Chevalley formula expresses 
the multiplication of a Schubert class by the class of a line bundle, or a Schubert divisor in $G/P$. 
If one works equivariantly, this formula determines completely the multiplication in the equivariant K ring.
In this paper we prove a Chevalley formula for the coefficients $C_{u,\lambda}^w(y) \in K_T(pt)[y]$ 
arising in the multiplication 
\begin{equation}\label{E:Chev-intro} MC_y(X(w)^\circ) \cdot \mathcal{L}_\lambda = \sum C_{u,\lambda}^w(y) MC_y(X(u)^\circ)
\end{equation} 
in the equivariant K-theory ring $K_T(G/B)[y]$; see \eqref{Chev:MC} and \Cref{thm:chemc} below. 
Here $MC_y(X(w)^\circ) \in K_T(G/B)[y]$ is 
the {\em motivic Chern class} of a Schubert cell $X(w)^\circ \subset G/B$, and
$\mathcal{L}_\lambda = G \times^B \C_\lambda$ is the line bundle on $G/B$ associated to the one dimensional $B$-module of weight $\lambda$.

The motivic Chern classes $MC_y(X(w)^\circ) \in K_T(G/B)[y]$ have been defined by Brasselet, Sch{\"u}rmann, 
and Yokura \cite{brasselet.schurmann.yokura:hirzebruch} more generally for elements $[Y \to X]$ in the Grothendieck
group $G_0(var/X)$ of varieties over $X$. They are the unique classes which are functorial with respect to
proper morphisms $f: X_1 \to X_2$, and which satisfy the normalization condition 
$MC[id_X: X \to X] = \lambda_y(T^*_X)$ for $X$ smooth, where
$\lambda_y(T^*_X) = \sum y^i [\wedge^i T^*_X]$ is the Hirzebruch $\lambda_y$ class; see \S \ref{sec:MCclasses} below. They may be thought of
as the K-theoretic generalizations of Chern--Schwartz--MacPherson classes 
defined by MacPherson \cite{macpherson:chern}. 

The motivic Chern classes of Schubert cells generalize well studied classes from Schubert calculus. 
If $y=0$, the
motivic class $MC_y(X(w)^\circ)$ is equal to the class of the ideal sheaf 
$[\cO_{X(w)}(-\partial X(w))]$ 
of the boundary $\partial X(w) = X(w) \setminus X(w)^\circ$, where $X(w) = \overline{X(w)^\circ}$ is the Schubert variety. If $y=-1$, then $MC_y(X(w)^\circ)$ is equal to the class of the unique $T$-fixed 
point in $X(w)^\circ$; see \cite{AMSS:specializations}. The Poincar{\'e} duals of the classes $MC_y(X(w)^\circ)$, the {\em Segre 
motivic classes} \cite{AMSS:motivic,mihalcea2022whittaker}, specialize when $y=0$ to the Grothendieck classes of the structure sheaves of the opposite Schubert varieties.  
Our Chevalley formula \eqref{E:Chev-intro}, and its analogous formula for the Segre motivic classes, 
specialize to known Chevalley formulae for K-theoretic Schubert classes and ideal sheaves from 
\cite{griffeth.ram:affine,lenart.postnikov:affine}.

The formula for the coefficients $C_{u,\lambda}^w(y) \in K_T(pt)[y]$ from \eqref{E:Chev-intro} follows from 
a formula of Ram \cite{R06} 
in the affine Hecke algebra $\mathbb{H}$, calculating transition coefficients between
two bases $\{ T_w X^\lambda \}$ and $\{ X^\lambda T_w \}$ of the affine Hecke algebra:
\begin{equation}\label{E:Hecke-exp} 
T_{w}X^{-\lambda}=\sum_{\mu\in X^*(T),u\in W}(-q)^{\ell(w)-\ell(u)}c_{u,\mu}^{w,\lambda}X^{-\mu}T_{u} \/.
\end{equation}
Here $T_w$ is an element in the standard basis of 
$\mathbb{H}$, $X^{-\lambda}$ is an affine element in $\mathbb{H}$,
and $X^*(T)$ denotes the weight lattice of $T$. Ram's formula is stated 
in terms of a combinatorial model utilizing alcove walks, and it is convenient for our purposes to 
rewrite it utilizing in terms of $\lambda$-chains, a model introduced and studied by Lenart and Postnikov 
\cite{lenart.postnikov:affine,lenart2008combinatorial} in relation to equivariant K 
theory of flag manifolds. We refer to 
\Cref{thm:lambda-chain1} and \Cref{thm:lambda-chain2} 
for the precise statements in the Hecke 
algebra in terms 
of $\lambda$-chains, and to \S \ref{sec:checoeff} for 
the formulae involving motivic Chern classes. We also note that (affine) Hecke algebras have long been 
used to obtain Chevalley formulas in various contexts, for example in \cite{pittie.ram,lenart.postnikov:affine}.
We state next our main result.

Assume $\lambda$ is an integral weight and fix a reduced 
$\lambda$-chain $(\beta_1, \beta_2,\ldots, \beta_l)$. The chain
corresponds to an alcove walk from the fundamental alcove $\fA$ to 
$\fA - \lambda$, with separating hyperplanes $H_{-\beta_j, d_j}$. 
Denote by $s_\beta$ the reflection determined by the root $\beta$. 
We refer the reader to \S \ref{sec:lambda-chains} below 
for full definitions. 
The following
is our main result, cf.~\Cref{thm:lambda}.

{
\begin{theorem}\label{thm:chemc-intro} 
The following Chevalley formula holds in $K_T(G/B)[y]$:
\[\calL_{-\lambda}\otimes MC_y(X(w)^\circ)=\sum_{\mu\in X^*(T),u\in W} C_{u,-\lambda}^w MC_y(X(u)^\circ),\]
where the Chevalley coefficients are given by
\begin{equation}\label{equ:-lambda-intro}
	C_{u,-\lambda}^w= \sum_{J\subset \{ 1,2,\ldots, l \}}
	(-1)^{n(J)} (1+y)^{|J|} (-y)^{\frac{\ell(w)-\ell(u)-|J|}{2}}
	e^{-w\tilde{r}_{J_{>}}(\lambda)},
\end{equation}
and the sum is over subsets $J = \{ j_1 < \ldots < j_t \} \subset \{ 1, 2, \ldots, l \}$
such that $u < u s_{\beta_{j_1}} < u s_{\beta_{j_1}} s_{\beta_{j_2}} < \ldots < u s_{\beta_{j_1}} s_{\beta_{j_2}} \cdot \ldots \cdot s_{\beta_{j_t}} =w$; the Weyl group element $\tilde{r}_{J_>}$ is defined in \eqref{E:defrtilde}.
For the multiplication $\calL_{\lambda}\otimes MC_y(X(w)^\circ)$, the Chevalley 
coefficients are given by
\begin{equation}\label{equ:lambda-intro}
	C_{u,\lambda}^w= \sum_{J\subset \{ 1,2,\ldots, l \}}
	(-1)^{n(J)}
	(-1-y)^{|J|} (-y)^{\frac{\ell(w)-\ell(u)-|J|}{2}}
	e^{-w \hat{r}_{J_{<}}(-\lambda)},
\end{equation}
where the sum is over subsets $J = \{ j_1 < \ldots < j_t \} \subset \{ 1, 2, \ldots, l \}$
such that $u < u s_{\beta_{j_t}} < u s_{\beta_{j_t}} s_{\beta_{j_{t-1}}} < \ldots < u s_{\beta_{j_t}} 
\cdot \ldots \cdot s_{\beta_{j_1}} =w$, and with $\hat{r}_{J_{<}}$ defined in \eqref{E:hatr}.
\end{theorem}
The connection between the Hecke algebra coefficients from \eqref{E:Hecke-exp} and the Chevalley 
coefficients above is given by
 \begin{equation}\label{equ:intro-leftdle}
	C_{u,-\lambda}^w= \sum_{\mu\in X^*(T)} y^{\ell(w)-\ell(u)}e^{-\mu}c_{u,\mu}^{w,\lambda}|_{q=-y}.
\end{equation}
The coefficients $c_{u,\mu}^{w,\lambda}$ are in general Laurent polynomials in $y$, while
$C_{u,-\lambda}^w$ are polynomials in $K_T(pt)[y]$. In fact, the power 
$y^{\ell(w)-\ell(u)}$ from the formula \eqref{equ:intro-leftdle} is absorbed into 
$c_{u,\mu}^{w,\lambda}$ so it becomes polynomial in $y$.}

As mentioned above, our Chevalley formula generalizes to 
the motivic situation the classical Chevalley multiplication in $K_T(G/B)$. It
also generalizes the Chevalley multiplication 
by (equivariant) Chern--Schwartz--MacPherson classes
of Schubert cells from \cite{AMSS:shadows}; 
a short, self-contained, proof of this is added in an Appendix. 

All these specializations are appropriately positive, in the sense of 
\cite{buch:Kpos,brion:Kpos,anderson.griffeth.miller:positivity}.
In \S \ref{sec:positivity} below we investigate some positivity results for the general formula.
Notably, our formula for the multiplication by $\mathcal{L}_\lambda$ 
with $\lambda$ dominant (i.e., when
$\mathcal{L}_{-\lambda}$ is globally generated) 
may be written as a positive combination of products $q^a(q-1)^b$, with $q=-y$; see
\Cref{prop:lambda-pos}. This positivity is similar to the one satisfied by $R$-polynomials in Kazhdan--Lusztig theory. In an earlier ar$\chi$iv version of this paper, we conjectured different positivity
properties for special cases of the Chevalley coefficients, regarded as polynomials in $y$. 
As we explain in \cref{rmk:former-conj}, we since found examples in Lie types $D_6,E_6,A_7$ where the conjectured positivity fails.

We now give a rough idea on the proof of \Cref{thm:chemc-intro}.
The key connection between the Chevalley formula in the Hecke algebra to motivic Chern classes, 
proved in \cite{mihalcea2020left}, and ultimately based on results from \cite{AMSS:motivic}, is that the motivic 
Chern classes are recursively obtained
by certain {\em left} Demazure--Lusztig operators $\mathcal{T}_w^L$ acting on $K_T(G/B)[y]$:
\[ MC_y(X(w)^\circ) =\mathcal{T}_{w}^L[\mathcal{O}_{1.B}] \/. \]
These operators commute with elements in $K_G(G/B)[y]$ (i.e., the Weyl-group
invariants of $K_T(G/B)$), and an argument based on equivariant localization shows
that 
\[ \begin{split} MC_y(X(w)^\circ) \cdot \mathcal{L}_\lambda  = \mathcal{T}_{w}^L[\mathcal{O}_{1.B}] \cdot \mathcal{L}_\lambda
= \mathcal{T}_{w}^L (\mathcal{L}_\lambda \cdot [\mathcal{O}_{1.B}]) 
= \mathcal{T}_{w}^L (e^\lambda \cdot [\mathcal{O}_{1.B}]) \/. \end{split} \]
Therefore, the knowledge of the expansion from \eqref{E:Hecke-exp}
implies the Chevalley formula in the geometric case. This argument may be generalized to any homogeneous 
bundle, see \Cref{rmk:bundle-mult} below. In cohomology (i.e., for the Chern--Schwartz--MacPherson classes), and for 
$G=\mathrm{SL}_n$, this argument is implicitly utilized in the paper \cite{fan.xiong.guo:MN} to obtain 
a Murnaghan--Nakayama formula. 

We briefly survey next other results from this note.
Having established a formula to calculate the Chevalley coefficients, in \S \ref{sec:chev-dualities} we utilize several dualities with geometric origin (the Serre duality, the star duality, and the Dynkin automorphism duality) to obtain several
symmetries of the coefficients $C_{v,\lambda}^w(y)$. See e.g. \Cref{prop:dual}. Combining these dualities shows that the polynomials $C_{v,\lambda}^w(y)$ are palindromic.

A remarkable property of the motivic Chern classes of Schubert cells, proved in \cite{AMSS:motivic,feher2021motivic}, 
is that they are equivalent to the K-theoretic version of Maulik and 
Okounkov's stable envelopes, see \cite{maulik2019quantum,aganagic2021elliptic}. The stable envelopes 
are elements in the $T \times \C^*$-equivariant K-theory of the cotangent bundle, $K_{T \times \C^*}(T^*(G/B))$. 
In this context, 
the formal variable $y$ may be identified to the (inverse) of the character given by the $\C^*$ fibre 
dilation on the cotangent bundle.
If $\iota:G/B \hookrightarrow T^*(G/B)$ is the inclusion of the zero section, then $\iota^*(\stab(w))$ is a multiple of the motivic Chern class of the (opposite) Schubert cell for $w$,
where $\stab$ is a stable envelope, appropriately normalized. 

The stable envelopes depend on three parameters: a chamber, a polarization, and a slope, 
and the precise normalizations are essential for this paper.
A variation in the chamber results in conjugating by the Borel subgroup \cite{AMSS:motivic}, 
and it is encoded in the left Weyl group
action \cite{mihalcea2020left} and certain $R$-matrix operators \cite{RTV:Kstable,rimanyi2017elliptic}.
Varying the polarization, or the slope, results in the multiplication of $\stab(w)$ by a line bundle $\mathcal{L}_\lambda$ 
pulled back from $G/B$; cf.~\cite{AMSS:motivic,okounkov:Klectures}, see also \S \ref{sec:K-stab} below. In particular,
the coefficients $C_{v,\lambda}^w(y)$ from \eqref{E:Chev-intro} 
give `wall-crossing' formulae, recording 
the change of stable envelopes when its defining parameters are varied.
While these wall crossing formulae have been worked out in 
\cite{okounkov:Klectures,su2020k,su2021wall} 
(see also \cite{koncki2023hecke}), in
\S \ref{sec:K-stab} we revisit some of these from the point of view
of \Cref{thm:chemc-intro}. In particular, we utilize 
the Chevalley formula to give an explicit combinatorial rule
relating the stable envelope for the fundamental alcove $\fA$ 
to the one corresponding to any translation $\fA + \lambda$; see
\Cref{prop:chestab}.

In addition to our application mentioned above to wall crossing formulae for stable envelopes, in 
\S \ref{sec:special-functions} we utilize known relations between motivic Chern classes
of Schubert cells, Whittaker functions, and Hall--Littlewood polynomials, to 
obtain new formulae for the latter.

Finally, in an Appendix we obtain an analogue of the Chevalley formula \eqref{E:Chev-intro}
for the homological analogue of the motivic Chern classes, the Chern--Schwartz--MacPherson classes.
While this formula may be obtained by a specialization argument as in \cite{AMSS:specializations}, the degenerate affine Hecke algebra 
is much simpler in this case, and a direct proof of the Chevalley formula may be obtained rather quickly. 

{\em Acknowledgments.} HN and CS thank Arun Ram and Andrei Okounkov for answering their questions and providing references; LM and CS thank Bogdan Ion for a related collaboration on Hall--Littlewood polynomials, and Paolo Aluffi and J{\"o}rg Sch{\"u}rmann for previous collaborations on motivic Chern classes and CSM classes; LM thanks Andrzej Weber for stimulating discussions. This project started in Spring 2021 while LM was in residence at the Institute for Computational and Experimental Research in Mathematics in Providence, RI, during the Combinatorial Algebraic Geometry program; LM wishes to thank the institute for excellent working conditions.

\subsection*{Notation}
We fix the notation utilized throughout the paper. Let $G$ be a simply connected complex Lie group with Borel subgroup $B$ and maximal torus $T\subset B$. Denote by $\mathfrak g= \mathrm{Lie}(G)$ and by $\fh= \mathrm{Lie}(T)$ be corresponding Lie algebras. Let $R^+\subset \fh^*:=\fh^*_\bbQ$ denote the positive roots, 
which by convention are the roots in $B$, and by $\Sigma =\{ \alpha_i: i \in I \}$ the set of simple roots. The set of all roots is $R:=R^+\sqcup -R^+$. 
We use $\alpha>0$ (resp. $\alpha<0$) to denote $\alpha\in R^+$ (resp. $\alpha\in -R^+$). For any root $\alpha\in R$, let $\alpha^\vee \subset \fh$ denote the corresponding coroot. Denote by 
$\langle \cdot,\cdot \rangle: \frak{h}^{*}\times \frak {h}\to \bbQ$ the evaluation pairing, and let $X^*(T)\subset \fh^*$ be the weight lattice. For any weight $\lambda\in X^*(T)$, let $\calL_\lambda:=G\times^B \bbC_\lambda$ be the line bundle on $G/B$ associated to $\lambda$. The Weyl group is $W = N_G(T)/T$ and it is generated by simple reflections $s_i= s_{\alpha_i}$ ($i \in I$). It is equipped with a length function $\ell:W \to \bZ_{\ge 0}$ defined as the length of a minimal expression of $w$ in terms of the simple reflections; we denote by $w_0$ the longest element.
The Bruhat order on $W$ is a partial order determined by the covering relations $u \le u s_\alpha$ where $\alpha \in R$ and 
{$\ell(us_\alpha)= \ell(u)+1$.}

For any $w\in W$, let $X(w)^\circ:=BwB/B\subset G/B$ and $Y(w)^\circ:=B^-wB/B\subset G/B$ be Schubert cells, where $B^-$ is the opposite Borel subgroup. Let $X(w):=\overline{X(w)^\circ}$ and $Y(w):=\overline{Y(w)^\circ}$ be the Schubert varieties, respectively. Let $P (\supseteq B)$ be a parabolic subgroup with simple roots $\Sigma_P\subset \Sigma$. Let $R_P^+$ denote the positive roots spanned by $\Sigma_P$. Let $W_P$ be the Weyl group generated by the simple reflections $s_\alpha$, $\alpha\in \Sigma_P$. Let $W^P\simeq W/W_P$ denote the set of minimal length representatives. For any $w\in W^P$, let $X(wW_P)^\circ:=BwP/P\subset G/P$ (resp. $Y(wW_P)^\circ:=B^-wP/P\subset G/P$) denote the Schubert cell with closure $X(wW_P)$ (resp. $Y(wW_P)$). Let $X^{*}(T)_P:=\{\lambda\in X^{*}(T)\mid \langle\lambda, \gamma^\vee\rangle=0 \text{ for all }
\gamma\in R^+_P \}$ be the set of integral weights which vanish on $ (R^+_P)^\vee$. For any 
$\lambda\in X^*(T)_P$, we still use $\calL_\lambda$ to denote the line bundle $G\times^P\bbC_\lambda\in\Pic(G/P)$, which has 
fiber over $1.P$ the one dimensional $T$-module of weight $\lambda$.

\section{Affine Hecke algebra via alcove walk algebra}\label{sec:affinehecke}
In this section, we introduce the alcove walk algebra, and a formula of Ram \cite{R06} describing a change of bases matrix for the affine Hecke algebra.

\subsection{Affine Hecke algebra}\label{sec:AHA}
The affine Hecke algebra $\bbH$ is a free $\bbZ[q,q^{-1}]$ module with basis $\{T_wX^\lambda|w\in W, \lambda\in X^*(T)\}$, such that
\begin{itemize}
	\item 
	For any $\lambda,\mu\in X^*(T)$, $X^\lambda X^\mu=X^{\lambda+\mu}$.
	\item
	For any simple root $\alpha$, $(T_{s_\alpha}+1)(T_{s_\alpha}-q)=0$.
	\item
	For any $w, y\in W$, such that $\ell(wy)=\ell(w)+\ell(y)$, $T_wT_y=T_{wy}$
	\item 
	For any simple root $\alpha$ and $\lambda\in X^*(T)$, 
	\[T_{s_\alpha}X^\lambda-X^{s_\alpha\lambda}T_{s_\alpha}=(1-q)\frac{X^{s_\alpha\lambda}-X^\lambda}{1-X^{-\alpha}}.\]
\end{itemize}
For our geometric application we will need two other bases of the affine Hecke algebra $\bbH$:
$\{T_{w^{-1}}^{-1}X^\lambda\mid w\in W, \lambda\in X^*(T)\}$ and $\{X^\lambda T_{w^{-1}}^{-1}\mid w\in W, \lambda\in X^*(T)\}$. Define the transition matrix coefficients $c_{u,\mu}^{w,\lambda}\in \bbZ[q,q^{-1}]$ by
\begin{equation}\label{equ:matrix1}
	T_{w^{-1}}^{-1}X^\lambda=\sum_{\mu\in X^*(T),u\in W}c_{u,\mu}^{w,\lambda}X^\mu T_{u^{-1}}^{-1}.
\end{equation}
The main result of this section is a formula for $c_{u,\mu}^{w,\lambda}$ obtained by Ram \cite{R06}, see \Cref{thm:ramformula} below. 

For the later application to the motivic Chern classes, we also introduce the Iwahori--Matsumoto $\bbZ[q,q^{-1}]$-algebra involution $\Theta$ on $\bbH$ defined by  
\[\Theta(T_{s_\alpha})=-qT_{s_\alpha}^{-1}, \textit{\quad and \quad}\Theta(X^\lambda)=X^{-\lambda},\]
where $s_\alpha$ is a simple reflection; see \cite[Section 5.1]{MR1432304}. Hence,  $\Theta(T_{w^{-1}}^{-1})= (-q)^{-\ell(w)}T_{w}$.
Applying $\Theta$ to Equation \eqref{equ:matrix1}, we obtain:
\begin{equation}\label{equ:matrix2}
	 T_{w}X^{-\lambda}=\sum_{\mu\in X^*(T),u\in W}(-q)^{\ell(w)-\ell(u)}c_{u,\mu}^{w,\lambda}X^{-\mu}T_{u}
\end{equation}

\subsection{Alcove walk algebra}
In this section, we review Ram's definition of the alcove walk algebra, and state his formula for the matrix coefficients $c_{u,\mu}^{w,\lambda}$. We refer the reader to \cite{R06} for a more detailed account of the alcove walk algebras. 
\subsubsection{Alcoves}\label{sec:alcoves}
Let $\ft^*_\bbR$ be the dual of the Lie algebra of the maximal torus $T$. For any root $\alpha$ and $j\in \bbZ$, define 
\[H_{\alpha,j}:=\{\lambda\in \ft^*_\bbR\mid \langle\lambda,\alpha^\vee\rangle=j\}.\]
Notice that $H_{\alpha,j}=H_{-\alpha,-j}$.
The connected components of $\ft^*_\bbR\setminus \cup_{\alpha>0,j\in \bbZ} H_{\alpha,j}$ are called {\bf alcoves}. The codimension 1 faces of any alcove are called the walls of that alcove. The {\bf fundamental alcove} $\fA$ 
is defined by:
\[\fA=\{\lambda\in \ft^*_\bbR\mid 0<\langle\lambda, \alpha^\vee\rangle<1, 
\textit{ for any positive root } \alpha\}.\]
If $\alpha_1,\alpha_2,\ldots,\alpha_r$ denote the simple roots, and
$\theta^\vee$ denotes the highest coroot, then the walls of the 
fundamental alcove $\fA$ are $H_{\theta,1}$ and $H_{\alpha_i,0}$ 
($1\leq i\leq r$). We label these walls of $\fA$ by $0,1,\ldots, r$ respectively. 

The affine Weyl group for the dual root system is defined by 
$W_{\aff}:=Q\rtimes W$, where $Q$ is the root lattice.
Then $W_{\aff}$ acts simply transitively 
on the set of alcoves, and this action is determined by 
the reflections across the hyperplanes
$h=H_{\alpha, j}$, given by 
\begin{equation}
	s_{\alpha,j}(\mu)=\hat{r}_{h}(\mu):=s_{\alpha}(\mu)+j \alpha \text{ for } \mu \in X^{*}(T),
\end{equation}
The affine Weyl group is a 
Coxeter group generated by the reflections $s_0:= s_{\theta, 1}$ and $s_i$ ($1\leq i\leq r)$ 
along the walls 
of $\fA$. In fact, $\fA$ is a fundamental domain for the action of $W_{\aff}$ on 
the set of alcoves, in the sense that any element in 
$\ft^*_\bbR\setminus \cup_{\alpha>0,j\in \bbZ} H_{\alpha,j}$
is sent to exactly one element in $\fA$. See \cite{R06,humphreys1990reflection}
for more details.

The extended affine Weyl group for the dual root system is $W^{\ext}_{\aff}=X^*(T)\rtimes W$, where $X^*(T)$ is the weight lattice. For any $\lambda\in X^*(T)$, let $t_\lambda$ denote the corresponding element in $W^{\ext}_{\aff}$. There is a length function $\ell$ on $W^{\ext}_{\aff}$ defined by the following formula (see \cite[Equation (2.8)]{macdonald1996affine}):
	\[ \ell(t_{\mu}w) = \sum_{\alpha \in R^+} | \langle \mu, w(\alpha^\vee) \rangle + \chi(w (\alpha))| \quad \textrm{ where } \quad \chi(\alpha) = \begin{cases} 0 & \alpha  \in R^+ \\ 1 & \alpha \in R^- \/. \end{cases}\]
	
Let $\Omega\subset W^{\ext}_{\aff}$ be the {subgroup} of length zero elements in $W^{\ext}_{\aff}$. Then $W_{\aff}^{\ext} \simeq W_{\aff} \rtimes \Omega$, see \cite[Equation (2.10)]{macdonald1996affine}, and $\ell(wg) = \ell(w)$ for any $w \in W_{\aff}$ and $g\in \Omega$. The elements in $\Omega$ preserve the fundamental alcove $\fA$ and act as automorphisms.

Utilizing a bijection between $W_{\aff}$ and the alcoves in $\ft^*_\bbR$, one can define a bijection between $W_{\aff}^{\ext} \simeq W_{\aff} \rtimes \Omega$ and the alcoves in $\Omega\times \ft^*_\bbR$ ($|\Omega|$ copies of $\ft^*_\bbR$, each tiled by alcoves).
We label the walls of every alcove in $\Omega\times \ft^*_\bbR$ in an $W^{\ext}_{\aff}$-equivariant way. This means that for each $w\in W^{\ext}_{\aff}$ the walls of $w\fA$ are $wH_{\alpha_i,0}$ ($1\leq i\leq n$) and $wH_{\theta,1}$, and they are labeled by $i$ and $0$, respectively. 
In particular, if two adjacent alcoves $A_1$ and $A_2$ are separated by a wall labeled by $i$ 
(in both $A_1$ and $A_2$), and $A_1=w\fA$ for some $w\in W^{\ext}_{\aff}$, then $A_2=ws_i\fA$. Equivalently, in terms of the wall crossings, if $w=gs_{i_1}s_{i_2}\cdot \ldots \cdot s_{i_\ell}\in W^{\ext}_{\aff}$, with $g\in \Omega$ and $0\leq i_j\leq r$, then the alcove $w\fA$ in $\Omega\times \ft^*_\bbR$ 
is the alcove obtained from rotating the fundamental alcove $\fA$ according to the automorphism $g$, then 
reflect along the walls labelled (in order) by $i_1, \ldots, i _\ell$; see also \Cref{lemma_alcove_path} below.

\begin{example}\label{exm:p1} We consider the example of the root system of type $A_1$.
	Let $\alpha$ be the positive root, and $\omega=\alpha/2$ be the fundamental weight. 
	The weight lattice $X^*(T)=\bbZ \omega$, the root lattice $Q=\bbZ\alpha$, 
	and the finite Weyl group $W=\{id,s_1=s_\alpha\}$. The affine Weyl group 
	$W_{\aff}=Q\rtimes W$ has Coxeter generators $s_1$ and $s_0=t_\alpha s_1$. 
	The subgroup of length zero elements in $W^{\textit{ext}}_{\aff}$ is 
	$\Omega=\{id, g=t_\omega s_1\} \simeq \bZ/ 2\bZ$.
	
	In the following picture for $\Omega\times \ft_\bbR^*$, the lower sheet is the identity sheet, while the upper sheet is the sheet $g\times \ft_\bbR^*$. Each alcove $w\fA$ is labeled by the corresponding $w\in W^{\textit{ext}}_{\aff}$, both in the Coxeter presentation and the translation presentation. In the lower sheet, the walls $H_{\alpha,n}$ are labeled by $1$ if $n$ is even, and $0$ if $n$ is odd. On the upper sheet, the labelings are in the opposite way. 
	
	
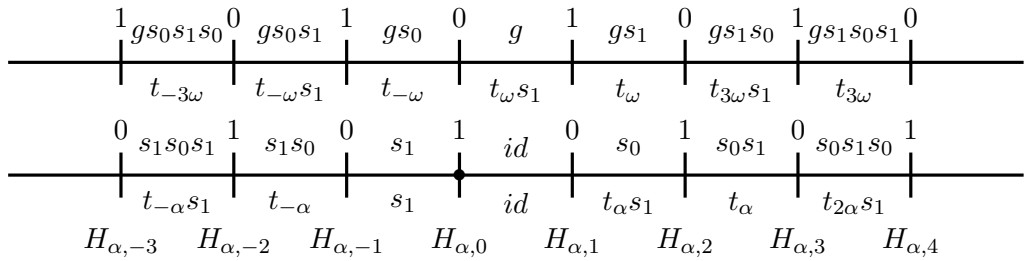
\begin{tikzpicture}[scale=1.5]
		\draw[very thick] (-4,1)--(5,1);
		\draw[very thick] (-4,0)--(5,0);
		\foreach \x in {-3,...,4}
		{\draw[very thick] (\x,-0.2)--(\x,0.2);
			\draw[very thick]  (\x, 0.8)--(\x,1.2);
			\node at (\x,-0.6) {$H_{\alpha,\x}$};}
		\foreach \n/\l [count=\i from -3]  in 
		{s_1 s_0 s_1/t_{-\alpha} s_1, s_1 s_0/t_{-\alpha}, s_1/s_1, id/id,
			s_0/t_{\alpha} s_1 , s_0 s_1/t_{\alpha} ,s_0 s_1 s_0/t_{2\alpha} s_1}
		{\node at (\i+0.5,0.25) {$\n$};\node at (\i+0.5,-0.25) {$\l$};}
		\foreach \n/\l [count=\i from -3]  in 
		{ g s_0 s_1 s_0/t_{-3\omega}, g s_0 s_1/t_{-\omega} s_1, 
			g s_0/t_{-\omega}, g/t_{\omega} s_1,
			g s_1/t_{\omega}, g s_1 s_0/t_{3\omega} s_1, g s_1 s_0 s_1/t_{3\omega}}
		{\node at (\i+0.5,1.25) {$\n$};\node at (\i+0.5,0.75) {$\l$};}
		\node at (0,0) {$\bullet$};
		\foreach \i in {-3,...,4}
		{\node at (\i,0.4) {\pgfmathparse{Mod(\i,2)==0?1:0}\pgfmathresult};
			\node at (\i,1.4) {\pgfmathparse{Mod(\i,2)==0?0:1}\pgfmathresult};}
	\end{tikzpicture}
	
\end{example}

\subsubsection{Alcove walk algebra}\label{sec:alcovewalk}
In this section we recall a realization of the Hecke algebra
in terms of alcove walks; see \cite{R06}. For each positive root $\alpha$ and hyperplane $H_{\alpha,j}$, set the positive side of it to be $\{\lambda\in \ft^*_\bbR\mid\langle\lambda,\alpha^\vee\rangle>j\}$. 
\begin{defin}
	The {\bf alcove walk algebra} is generated over $\bbZ[q,q^{-1}]$ by elements $g\in \Omega$, and for $0 \le i \le r$, the elements $c_i^+$ (positive $i$-crossing), $c_i^-$ (negative $i$-crossing), $f_i^+$ (positive $i$-fold) and $f_i^-$ (negative $i$-fold), subject to the following relations, sometimes called straightening laws:
	\[c_i^+=c_i^-+f_i^+,\quad c_i^-=c_i^++f_i^-\/, \quad \textrm{ and } \quad 
	gc_i^\pm=c_{g(i)}^\pm g \/, \quad gf_i^\pm=f_{g(i)}^\pm g \/. \]
\end{defin}
In terms of pictures, these generators can be drawn as follows:
\begin{center}
	\begin{minipage}{2.5cm}
		\begin{tikzpicture}[scale=0.7]
			\node at (-0.5,0.5) {$-$};
			\node at (0.5,0.5) {$+$};
			\node at (0,1) {$i$};
			\node at (0,-1.2) {$c_i^{+}$};
			\draw [very thick] (0,-0.7)--(0,0.7);
			\draw [->, very thick] (-0.8,0)--(0.8,0);
		\end{tikzpicture}
	\end{minipage}
	\begin{minipage}{2.5cm}
		\begin{tikzpicture}[scale=0.7]
			\node at (-0.5,0.5) {$-$};
			\node at (0.5,0.5) {$+$};
			\node at (0,1) {$i$};
			\node at (0,-1.2) {$c_i^{-}$};
			\draw [very thick] (0,-0.7)--(0,0.7);
			\draw [<-, very thick] (-0.8,0)--(0.8,0);
		\end{tikzpicture}
	\end{minipage}
	\begin{minipage}{2.5cm}
		\begin{tikzpicture}[scale=0.7]
			\node at (-0.5,0.5) {$-$};
			\node at (0.5,0.5) {$+$};
			\node at (0,1) {$i$};
			\node at (0,-1.2) {$f_i^{+}$};
			\draw [very thick] (0,-0.7)--(0,0.7);
			\draw [-> ,very thick] (0.8,-0.2)--(0.1,-0.2)--(0.1,0)--(0.8,0);
		\end{tikzpicture}
	\end{minipage}
	\begin{minipage}{2.5cm}
		\begin{tikzpicture}[scale=0.7]
			\node at (-0.5,0.5) {$-$};
			\node at (0.5,0.5) {$+$};
			\node at (0,1) {$i$};
			\node at (0,-1.2) {$f_i^{-}$};
			\draw [very thick] (0,-0.7)--(0,0.7);
			\draw [-> ,very thick] (-0.8,-0.2)--(-0.1,-0.2)--(-0.1,0)--(-0.8,0);
		\end{tikzpicture}
	\end{minipage}
	\begin{minipage}{2cm}
		$(0\leq i\leq r)$
	\end{minipage}
\end{center}
Here, $c_i^+$ represents a crossing of a wall labeled by $i$ from its negative to its positive side, and similarly for the other generators. The product is given as concatenation.
An \textbf{alcove walk} is a word in the generators such that,
\begin{itemize}
	\item the tail of the first step is in the fundamental alcove $\fA$;
	\item at every step, either we change the sheet according to an element in $\Omega$
	(thus rotating the alcove according to this elements), or the head of each arrow is in the same alcove as the tail of
	the next arrow.
\end{itemize} 
An alcove walk $p$ is called {\bf nonfolded} if there is no $f_i^\pm$ in its word. 
The {\bf length} of an alcove walk is the number of letters $c_i^{\pm}, f_i^{\pm}$
in an alcove walk. (In particular, rotation with respect to an element of $\Omega$ does not contribute to the length.) For a {\em minimal} alcove walk between two alcoves, one can show that  
the walk is non-folded, thus its length is the number of $c_i^\pm$ in the walk \cite{R06}. 
From this it follows that if $w \in W^{\ext}_{\aff}$, then $\ell(w)$=length of a minimal length walk from $\fA$ to $w\fA$.

Pick a square root $q^{\frac{1}{2}}$ of $q$. The following is proved by A. Ram.
\begin{prop}\label{Alcove_walk_to_affine_Hecke}
	\cite[\S 3.2]{R06}
	\begin{enumerate}
		\item[(a)]
		The affine Hecke algebra $\bbH$ is isomorphic as a $\bZ[q^{\frac{1}{2}},q^{-\frac{1}{2}}]$-algebra 
		to the quotient of the alcove walk algebra by the relations 
		\begin{equation}\label{equiv1}
			c_i^+=(c_i^-)^{-1},\quad f_i^+=(q^{\frac{1}{2}}-q^{-\frac{1}{2}}),\quad f_i^-=-(q^{\frac{1}{2}}-q^{-\frac{1}{2}})
		\end{equation}
		and\\
		\begin{equation}\label{equiv2}
			p=p'\text{ if }p\text{ and }p'\text{ are nonfolded alcove walks with }end(p)=end(p'),
		\end{equation}
		where $end(p)$ means the final alcove of $p$.
		\item[(b)] Under the previous isomorphism,
		for any $w\in W$ and $\lambda\in X^*(T)$ \footnote{The $q^2$ and $T_{s_i}$ in \cite[Proposition 3.2(b)]{R06} are our $q$ and $q^{-\frac{1}{2}}T_i$, respectively.}:
		\begin{itemize}
			\item a minimal length alcove walk from $\fA$ to $w\fA$ is sent to $q^{\frac{\ell(w)}{2}}T^{-1}_{w^{-1}}$, and,
			\item a minimal length alcove walk from $\fA$ to $t_\lambda \fA$ is sent to $X^\lambda$.
		\end{itemize}
	\end{enumerate}
\end{prop}

For an alcove walk $p$, define the functions {\bf weight} $wt(p)\in X^*(T)$, and {\bf final direction} $\varphi(p)\in W$ of $p$ 
by the condition that $p$ ends in the alcove $t_{wt(p)}\varphi(p)\fA$. Let
\begin{align*}
	f^-(p) &= \textit{number of negative folds of } p,\\
	f^+(p) &= \textit{number of positive folds of } p, \textit{ and}\\
	f(p) &= f^+(p)+f^-(p).
\end{align*}

Now we can state the formula for the matrix coefficients $c_{u,\mu}^{w,\lambda}$ defined in Equation \eqref{equ:matrix1}, see also Equation \eqref{equ:matrix2}.
\begin{theorem}\label{thm:ramformula}\cite[Theorem 3.3]{R06}
	Let $\lambda\in X^*(T)$ and $w\in W$. Fix a minimal length walk $p_w=c_{i_1}^-c_{i_2}^-\cdots c_{i_r}^-$ from $\fA$ to $w\fA$ and a minimal length walk 
	$p_\lambda=c_{j_1}^{\epsilon_1}c_{j_2}^{\epsilon_2}\cdots c_{j_s}^{\epsilon_s}g$ 
	from $\fA$ to $t_\lambda \fA$, where $g\in W^{ext}_{\aff} $ is defined by the condition 
		$gW_{\aff} = t_{\lambda} W_{\aff}$ \footnote{ We need to add this extra $g\in \Omega$ since $t_\lambda \fA$ may not on the same sheet as $\fA$, see Example \ref{exm:p1continued}. The stated conditions determine $g$ uniquely.}, and $\epsilon_i = \pm$ for each $i$. Then 
	\begin{equation}\label{transition}
		T^{-1}_{w^{-1}}X^\lambda=\sum_p (-1)^{f^-(p)} (q^{\frac{1}{2}}-q^{-\frac{1}{2}})^{f(p)}
		q^{\frac{\ell(\varphi(p))-\ell(w)}{2}}X^{wt(p)}T_{\varphi(p)^{-1}}^{-1},
	\end{equation}
	where the sum is over all alcove walks $p$ of the form
	\begin{equation}\label{alcove_path}
		p=c_{i_1}^-c_{i_2}^-\cdots c_{i_r}^-p_{j_1}p_{j_2}\cdots p_{j_s}g \text{ such that } p_{j_k} \in \{c_{j_k}^{\pm},f_{j_k}^{\epsilon_k} \} \/.\end{equation}
	Therefore the matrix coefficients $c_{u,\mu}^{w,\lambda}$ in \Cref{equ:matrix1} are given by:
	\begin{equation}
		c_{u,\mu}^{w,\lambda}=\sum_{\substack{p \text{ of the form  } (\ref{alcove_path})\\\varphi(p)=u,wt(p)=\mu}}
		(-1)^{f^{+}(p)} (1-q)^{f(p)} q^{\frac{\ell(u)-\ell(w)-f(p)}{2}}.
	\end{equation}
\end{theorem}

\begin{example}\label{exm:p1continued}
	Let $G=\SL(2,\bbC)$. We use the same notation as in Example \ref{exm:p1}. We check the above theorem for $w=s_1$ and $\lambda=\omega$. From the alcove picture in Example \ref{exm:p1}, $T_{s_1}^{-1}$ is represented by the minimal length walk $q^{-1/2} c_1^-$, while $X^\omega$ is represented by the walk $gc_1^+=c_0^+g$. Thus, the sum in the right hand side of the Theorem is over the alcove walks $c_1^-c_0^-g$ and $c_1^-f_0^+g$, which end at the alcoves $t_{-\omega} s_1\fA$ and $t_{-\omega} \fA$, respectively. 
	(Note that $c_1^- c_0^+ g$, which represents $T_{s_1}^{-1}X^\omega$, is not an alcove walk.)
	Therefore, the identity in the theorem is
	\[T_{s_1}^{-1}X^\omega=X^{-\omega}T_{s_1}^{-1}-q^{-1}(1-q)X^{-\omega}.\]
	On the other hand, it is easy to check the above equation using the definition of the affine Hecke algebra in Section \ref{sec:AHA}.
\end{example}

\begin{remark}\label{nonfolded}
In \Cref{thm:ramformula}, one may relax the hypotheses about 
	the alcove walks $p_w$ and $p_{\lambda}$ to be of non-minimal length. This follows from analyzing the proof of Ram's result in \cite{R06}. We do not use this, but it is consistent with our use of non-reduced 
	$\lambda$-chains in \S \ref{sec:lambda-chains} below.
\end{remark}

Consider an ordered collection of hyperplanes  $\mathcal{H} = \{ H_{\beta_1, k_1}, \ldots, H_{\beta_s, k_s} \}$
	and set $h_i:=H_{\beta_i, k_i}$. Associated to this sequence we define
	the elements 
	\[\hat{r}_{\mathcal{H}} = s_{\beta_1,k_1}\cdot \ldots \cdot s_{\beta_s,k_s} \in W_{\aff} \/; \quad 
	r_{\mathcal{H}} = {s}_{\beta_1}\cdot \ldots \cdot s_{\beta_s} \in W \/. \]
(Note that these depend on the order of $\mathcal{H}$.)
	We also define an $\mathcal{H}$-restricted version
	of the Bruhat order on $W$ by 
	\begin{equation}\label{cond_M}
		w\stackrel{\mathcal{H}}{\Longrightarrow} u  \iff 
		w>w s_{\beta_{1}}>w s_{\beta_{1}}s_{\beta_{2}}>\cdots>w s_{\beta_{1}}s_{\beta_{2}}\cdot \ldots \cdot s_{\beta_{s}}=u.
	\end{equation}
	
	The following is a mild generalization of \cite[Proposition 2.5]{lenart2011hall}.
	\begin{lemma}\label{lemma:paths-bij} Let $w \in W$ and $\lambda \in X^*(T)$. Fix:
		\begin{itemize} \item an alcove walk $p_w$ from $\fA$ to $w\fA$;
			\item an alcove walk 
			$p_{\lambda}=c_{j_1}^{\epsilon_1}c_{j_2}^{\epsilon_2}\cdots c_{j_s}^{\epsilon_s}g$
			from $\fA$ to $t_\lambda\fA$. 
		\end{itemize}
		Let $h_i = H_{\beta_i,k_i}$, for $1 \le i \le s$ be 
		the sequence of hyperplanes  
		defined by the walls of alcoves crossed by $p_{\lambda}$,
		with $\beta_i \in R^{\epsilon_i}$ and $k_i \in \bZ$.
		
		Then there is a bijection between the following two sets:
		\begin{enumerate} \item The set of alcove walks of the form $\bar{p}=p_w p_{j_1} \ldots p_{j_s}$ 
			such that $p_{j_k} \in \{ c_{j_k}^\pm, f_{j_k}^{\epsilon_k} \}$;
			
				\item[(2)] The set of subsets $\mathcal{M} = \{ h_{m_1} , \ldots, h_{m_t} \} \subset \mathcal{H} = \{ h_1, \ldots, h_s \}$ with $m_1 < m_2 < \ldots < m_t$, s.t. $w\stackrel{\mathcal{M}}{\Longrightarrow} w r_\mathcal{M}$.
		\end{enumerate}
		Under this bijection, the indices $m_i$ correspond to the positions of foldings 
		$f_{m_i}^{\epsilon_{m_i}}$. Furthermore, 
		\[ \varphi(\bar{p} g)=wr_\mathcal{M} \textrm{ and } 
		\mathrm{wt}(\bar{p}g)=w\hat{r}_\mathcal{M}(\lambda) \/. \]
	\end{lemma}
\begin{remark} This statement is a slight generalization of Lenart's result. We do not require $\lambda$ to be dominant, and therefore we need to allow $\beta_i \in R^{\epsilon_i}$ to be a negative root.\end{remark} 

\begin{proof} The proof follows the same outline as that of \cite[Proposition 2.5]{lenart2011hall}, so we will be brief.
To start, consider the {\em unique} unfolded alcove walk $p_0=p_w p_{j_1} \ldots p_{j_s}$ 
such that $p_{j_k} \in \{ c_{j_k}^\pm\}$. Any other alcove walk $\bar{p}$ as in the statement
is of the form
\[\bar{p}=\phi_{m_t}\cdots \phi_{m_1}(p_0) \/,\]
where $\phi_{m_j}$ is the folding operation at position $m_j$ (cf.~ \cite{lenart2011hall}) and
the $m_i$'s with $m_{i-1}<m_i$ are the folding positions of $\bar{p}$. This alcove walk has the property 
that if $k\notin \{m_1,\ldots,m_t\}$ then $p_{j_k} \in \{ c_{j_k}^\pm \}$  and $p_{j_{m_i}}=f_{j_{m_i}}^{\epsilon'_{m_i}}$ ($1\leq i\leq t$). In addition $\varphi(\bar{p} g)=wr_\mathcal{M}$ and 
$\mathrm{wt}(\bar{p}g)=w\hat{r}_\mathcal{M}(\lambda)$. Let 
$\mathcal{M}_i = \{ h_{m_1} , \ldots, h_{m_i} \} \subset \mathcal{H} = \{ h_1, \ldots, h_s \}$ 
for $1\leq i\leq t$ with the convention that
$\mathcal{M}_0 =\emptyset$ and $r_{\emptyset}=id$.
A key point is that $w r_{\mathcal{M}_{i-1}}> w r_{\mathcal{M}_{i}}$ if and only if
the folding orientations satisfy $\epsilon_{m_i}=\epsilon'_{m_i}$. (The proof uses the condition that if $\beta >0$, then 
$ws_\beta > w$ if and only if $w(\beta) >0$.) All this put together implies that
$\{ h_{m_1} , \ldots, h_{m_i} \} \leftrightarrow \phi_{m_i}\cdots \phi_{m_1}(p_0)$
gives the bijection in the statement.
\end{proof}
For a subset $\mathcal{M} \subset \mathcal{H}$ as in \Cref{lemma:paths-bij},
set $p(\mathcal{M})$ to be the alcove walk $\bar{p} g$ associated to $\mathcal{M}$, and
set $f^+(\mathcal{M})= f^+(p(\mathcal{M}))$. 

The following is a reformulation of Ram's result from \Cref{thm:ramformula} in terms of paths 
in the Bruhat order. 
\begin{corol}\label{corol:Chev} Let $u,w \in W$ and $\lambda, \mu \in X^*(T)$. Assume the same hypotheses
and notation as in \Cref{lemma:paths-bij}.
Then 
\[ c_{u,\mu}^{w,\lambda}=
		\sum_{\mathcal{M}} (-1)^{f^{+}(\mathcal{M})}  (1-q)^{|\mathcal{M}|} q^{\frac{\ell(u)-\ell(w)-|\mathcal{M}|}{2}},\]
where the sum is over subsets $\mathcal{M} \subset \mathcal{H}$ which satisfy
$w \stackrel{\mathcal{M}}{\Longrightarrow}u = w r_{\mathcal{M}}$ and $\mu = w \hat{r}_{\mathcal{M}}(\lambda)$. 
\end{corol}

\begin{remark} {\em A prior}i, the coefficients $c_{u,\mu}^{w,\lambda}$ appearing in the walk algebra are in $\bZ[q^{\pm \frac{1}{2}}]$. However, after matching these with the initial definition 
of Hecke algebras, it turns out that $c_{u,\mu}^{w,\lambda} \in  \bZ[q^{\pm 1}]$. This can also be seen directly in the formula above, by observing that
$\ell(u)-\ell(w) - |\mathcal{M}|$ is even. (This uses that a reflection has odd length, and the cancellation property of non-reduced expressions.)
\end{remark}

\section{A $\lambda$-chain formula for the transition coefficients $c_{u,\mu}^{w,\lambda}$}\label{sec:lambda-chains}
In this section, we reformulate the alcove walk formula \Cref{corol:Chev} 
in terms of the notion of $\lambda$-chains introduced in \cite{lenart.postnikov:affine}. This notion was utilized
to obtain a $K$-theory Chevalley formula for the structure sheaves of Schubert varieties. 
The main results of this section are \Cref{thm:lambda-chain1} and \Cref{thm:lambda-chain2}. Specializing $y \mapsto 0$, 
these recover \cite[Theorem 6.1]{lenart.postnikov:affine}.
Throughout this section, we only consider alcoves on the identity sheet $\ft_{\bbR}^*$. Recall that 
$\fA+\lambda:=\{ x+\lambda\mid x\in \fA\}$ is the alcove on $\mathfrak t_\bbR^{*}$.
(If $\lambda$ is not a root, the alcove $t_{\lambda}(\fA)$ is not the alcove $\fA+\lambda$.)

\subsection{Alcove paths and $\lambda$-chains}
For any two alcoves $A$ and $B$, which are separated by a common wall lying on a hyperplane $H_{\beta,k}$, write $A\stackrel{\beta}{\longrightarrow} B$ if the root $\beta$ points from $A$ to $B$. 
\begin{defin}\cite[Definition 5.2]{lenart.postnikov:affine}
An {\bf alcove path} is a sequence of alcoves $(A_0,A_1,\ldots, A_l)$ such that $A_{j-1}$ and $A_{j}$ are adjacent, $1\leq j\leq l$. We denote this by
$$A_0\stackrel{-\beta_1}{\longrightarrow} A_1 \stackrel{-\beta_2}{\longrightarrow} \cdots \stackrel{-\beta_l}{\longrightarrow}  A_l.$$
When the length $l$ is minimal among all alcove paths from $A_0$ to $A_l$, it is called a {\bf reduced alcove path}.
\end{defin}

\begin{remark}\label{-lambdaTo_alcove_walk} 
	By definition, there is a one-to-one correspondence between 
	alcove paths 
	\begin{equation}
		\fA=A_0  \stackrel{-\beta_1}{\longrightarrow}  A_1  \stackrel{-\beta_2}{\longrightarrow}  \cdots  \stackrel{-\beta_l}{\longrightarrow} A_l=\fA-\lambda
	\end{equation}
	from $\fA$ to $A_l=v_{-\lambda}(\fA)$
	and
	alcove walks from $A_0$ to $A_l$ 
	of the form
	$c^{\epsilon_1}_{i_1} c^{\epsilon_2}_{i_2}\cdots c^{\epsilon_l}_{i_l}$,
	where $-\beta_j\in R^{\epsilon_j}$ for  $1\leq j\leq l$.  
\end{remark}

Recall that $s_{0}=s_{\theta,1}$ is the affine reflection along the hyperplane $H_{\theta,1}$
with $\theta^\vee$ the highest coroot. Define
$\alpha_0 = -\theta$,
$\bar{s}_0=s_{\theta}$ and $\bar{s_i}=s_i$ for $1\leq i\leq r$.

\begin{lemma}\label{lemma_alcove_path}
\cite[Lemma 5.3]{lenart.postnikov:affine}
For any $v\in W_{\aff}$, there is a one-to-one correspondence between
decompositions of $v$ in the simple reflections $s_i \in W_{\aff}$ ($0 \le i \le r$) 
and alcove paths from $\fA$ to $v(\fA)$ as follows.

For any decomposition $v=s_{i_1} s_{i_2}\cdots s_{i_l}$, define
\[\beta_j:=\bar{s}_{i_1}\bar{s}_{i_2}\cdots \bar{s}_{i_{j-1}}(\alpha_{i_j}),\quad 1\leq j\leq l.\]
Then 
\[\fA=A_0\stackrel{-\beta_1}{\longrightarrow} A_1 \stackrel{-\beta_2}{\longrightarrow} \cdots \stackrel{-\beta_l}{\longrightarrow}  A_l=v(\fA)\] 
is an alcove path from $\fA$ to $v(\fA)$. 

Moreover, the affine reflection along the $j$-th hyperplane separating $A_{j-1}$ and $A_j$ is 
$s_{i_1} s_{i_2}\cdots s_{i_{j-1}}s_{i_j} s_{i_{j-1}}\cdots s_{i_2}s_{i_1}$; in particular, the separating
wall is labeled by $i_j$.
Under this correspondence, a reduced alcove path corresponds to a reduced decomposition for $v$.
\end{lemma}

Let $\lambda$ be an integral weight and let $v_{-\lambda}\in W_{\aff}$ be the affine Weyl group element 
which satisfies $v_{-\lambda}(\fA)=\fA-\lambda$, i.e., $t_{-\lambda}=v_{-\lambda}g$ with 
$g\in \Omega$. Choose a (possibly non-reduced) decomposition
$v_{-\lambda}=s_{i_1}\cdots s_{i_l}$ and let $\beta_j$ be defined by 
\[\beta_j:=\bar{s}_{i_1}\bar{s}_{i_2}\cdots \bar{s}_{i_{j-1}}(\alpha_{i_j}),\quad 1\leq j\leq l \/,\]
with the convention that $\alpha_0 = -\theta$
(see \Cref{lemma_alcove_path}). Then 
$$\fA=A_0\stackrel{-\beta_1}{\longrightarrow} A_1 \stackrel{-\beta_2}{\longrightarrow} \cdots \stackrel{-\beta_l}{\longrightarrow}  \fA-\lambda$$
is an alcove path from $\fA$ to $\fA - \lambda$. 
\begin{defin}\cite[Definition 5.4]{lenart.postnikov:affine}
	The sequence of roots $(\beta_1, \ldots, \beta_l)$ is a {\bf $\lambda$-chain} of roots
associated to the decomposition of $v_{-\lambda}$. A $\lambda$-chain is called reduced if the decomposition 
	$v_{-\lambda}=s_{i_1}\cdots s_{i_l}$ is reduced.
\end{defin}
Let $H_{-\beta_j, d_j}$ be the hyperplane separating the alcoves $A_j$ and $A_{j+1}$. 
The sequence of integers $d_j$ are determined by the sequence of roots $\beta_j$, but
we occasionally keep the information of $d$'s in the notation, and we refer to 
the sequence of pairs 
$(\beta_1, d_1),(\beta_2, d_2),\ldots , (\beta_l, d_l)$ as a $\lambda$-chain 
\begin{footnote}{If $\lambda$ is dominant, this definition was extended to the Kac--Moody situation 
in \cite{lenart2008combinatorial}.}\end{footnote}. 
Following \cite[Prop.~6.8]{lenart.postnikov:affine} we recall a combinatorial construction of a $\lambda$-chain
for an integral weight $\lambda$. 

Fix a linear order on the index of Dynkin nodes 
(for example $1<2<\cdots<r$ in $I=\{1,2,\ldots,r\}$). The 
set $\mathcal{R}_\lambda \subset W_\mathrm{aff}$ of affine reflections
$s_{\alpha,k}$ for the hyperplanes $H_{\alpha, k}$ 
separating the fundamental alcove $\fA$ and $\fA - \lambda$ is given by
\[ \mathcal{R}_\lambda = \bigcup_{\alpha \in R^+} \begin{cases}
 \{ s_{\alpha,k} : 0 \ge k > - \langle \lambda, \alpha^\vee \rangle \} & \textrm{ if } \langle \lambda, \alpha^\vee \rangle > 0 \\
  \{ s_{\alpha,k} : 0 < k \le - \langle \lambda , \alpha^\vee \rangle \} & \textrm{ if } \langle \lambda, \alpha^\vee \rangle < 0 \\
\emptyset & \textrm{ otherwise } 
\end{cases}
\]
One defines a `height function' 
\[ h:\mathcal{R}_\lambda\to \mathbb R^{r+1} \/; \quad 
h(s_{\alpha, k})=\frac{1}{\langle \lambda, \alpha^\vee \rangle}
(-k,
 \langle \varpi_1, \alpha^\vee \rangle, 
\cdots, \langle \varpi_r, \alpha^\vee \rangle) \/. \]
It turns out that this function is injective.
Now order the images of $h$ in lexicographic order, so we obtain
$h(s_{\alpha_1, k_1}) < \ldots < h(s_{\alpha_l, k_l})$. Define another function 
$b: \mathcal{R}_\lambda \to R^+ \cup R^-$ by
\[ b(s_{\alpha,k}) = \begin{cases} \alpha & \textrm{ if } k \le 0 , \alpha \in R^+ \/;\\
-\alpha & \textrm{ if } k > 0 , \alpha \in R^+ \/. \end{cases} \] 
Then 
$b(s_{\alpha_1,k_1}), \ldots, b(s_{\alpha_l,k_l})$ is a (reduced) $\lambda$-chain of roots.

\begin{remark} A particularly nice situation occurs for 
a minuscule fundamental weight $\varpi_i$, i.e. when
$\langle \varpi_i, \alpha^\vee \rangle \in \{0, 1 \}$
for any positive root $\alpha$. In this case all $k_i=0$ and
$v_{-\varpi_i} = (w^{P_i})^{-1} \in W$, with $w^{P_i}$ being
the longest minimal length representative 
for cosets of $W/W_{P_i}$, where
$W_{P_i} = \Stab_W(\varpi_i)$; equivalently, the Schubert variety 
$X(w^{P_i}W_{P_i}) = G/P_i$.
A reduced decomposition and the associated $\varpi_i$-chain may be read from the associated 
Young poset of $G/P_i$ - see e.g. \cite{proctor:dynkin,stembridge2001minuscule} 
and also \cite[\S 3.1]{BCMP:QKChev} and \Cref{ex:v-o2} below. 
It may also be obtained as
a reverse linear extension of the heap $H(w^{P_{i}})$, and 
this gives a one-to-one correspondence
between reduced $\varpi_i$-chains and 
reverse linear extensions of the heap $H(w^{P_{i}})$.
We refer the readers to \cite{mihalcea2022hook,naruse2019skew} 
for the heap perspective. \end{remark}


\begin{example}\label{ex:v-o2}
Consider $G=\mathrm{SL}_5$ and the fundamental weight $\varpi_2$. The stabilizer
is the maximal parabolic $P_2$ so that $G/P_2$ is the 
Grassmannian $\Gr(2,5)$. The inversion set and a reduced decomposition
of $v_{-\varpi_2}$ may be read from the Young diagrams below, with the notation
$(i-j) = \alpha_i +\ldots + \alpha_{j-1}$. 
\[ 
\ytableausetup{centertableaux,boxsize=2.4em}
\begin{ytableau}
{\tiny{2-3}}   & {{\tiny 2-4}}  & {\tiny{2-5}} \\
{\tiny{1- 3}}   & {\tiny{1 - 4}}  & {\tiny{1-5}}
\end{ytableau} 
\hskip1cm
\ytableausetup{centertableaux,boxsize=2.4em}
\begin{ytableau}
{\alpha_2}   & {\alpha_3}  & {\alpha_4} \\
{\alpha_1}   & {\tiny{\alpha_2}}  & {\alpha_3}
\end{ytableau} 
\]
Then $v_{-\varpi_2} =s_2 s_3 s_4 s_1 s_2 s_3$
and a $\varpi_2$-chain of roots is given
by $\{\alpha_2, \alpha_2+\alpha_3, \alpha_2+\alpha_3+\alpha_4, \alpha_1 +  \alpha_2,  
\alpha_1+ \alpha_2+\alpha_3, \alpha_1+ \alpha_2+\alpha_3+\alpha_4\}$. 

\end{example}
\begin{example}
Let $G=\SL_3$, with the Weyl group $W=S_3$, and consider 
$\lambda=2\varpi_1-2\varpi_2$. An example of alcove path from
$\fA$ to $\fA-\lambda$ is
 $$\fA=
 A_0  \stackrel{-\beta_1}{\longrightarrow}  
 A_1  \stackrel{-\beta_2}{\longrightarrow} 
 A_2  \stackrel{-\beta_3}{\longrightarrow} 
 A_3  \stackrel{-\beta_4}{\longrightarrow} 
 A_4  \stackrel{-\beta_5}{\longrightarrow} 
 A_5  \stackrel{-\beta_6}{\longrightarrow} 
 A_6=\fA-\lambda\;
 (A_{i}=r_i A_{i-1}, 1\leq i\leq 6),
 $$
which is the red path below,
and it corresponds to
the decomposition $v_{-\lambda}=s_0 s_1 s_0 s_1 s_2 s_1$.
The corresponding alcove walk is 
$\bar{p} =c_0^{+} c_1^{+} c_0^{-} c_1^{-} c_2^{-} c_1^{+}$, and
the corresponding $\lambda$-chain of roots is
$(\beta_1,1), (\beta_2,1), (\beta_3, 0), (\beta_4,-1), (\beta_5, 1), (\beta_6,2)$,
as calculated below. Here we included the $d$'s in the notation.

\begin{minipage}{7cm}

\begin{tikzpicture}[scale=1.0]
\newcommand*\rows{4}


\newcommand*\view{3.3}
\clip (-\view,-\view+1) rectangle (\view,\view+0.05);
\foreach \row in {-\rows,...,\rows}{
\draw[gray] (-\rows,{\row*sqrt(3)/2}) -- (\rows,{\row*sqrt(3)/2});}
\foreach \row in {-\rows,...,\rows}{
 \draw[gray] ({-\rows+\row},{-\rows*sqrt(3)})--({\rows+\row},{\rows*sqrt(3)});
 \draw[gray] ({-\rows+\row},{\rows*sqrt(3)})--({\rows+\row},{-\rows*sqrt(3)});
 };


\draw [lightgray,fill] (0,0)--(0.5,{sqrt(3)/2})--(-0.5,{sqrt(3)/2})--cycle;
\draw [lightgray,fill] (2,0)--(2.5,{sqrt(3)/2})--(2-0.5,{sqrt(3)/2})--cycle;

   \draw[brown,thick] (-\rows,{-\rows*sqrt(3)}) -- (\rows,{\rows*sqrt(3)});
   \draw[brown, thick] (-\rows,{\rows*sqrt(3)}) -- (\rows,{-\rows*sqrt(3)});
   \draw[brown,thick] (-\rows,0) -- (\rows,0); 

\node[brown,rotate=0] at ({-2+0.5},{sqrt(3)*1.8-0.2}) {\small$H_{\alpha_2,0}$};
\node[brown,rotate=0] at ({2-0.5},{sqrt(3)*1.8-0.2}) {\small$H_{\alpha_1,0}$};
\node[brown,rotate=0] at ({2.5},{sqrt(3)*0}) {\small$H_{\alpha_1+\alpha_2,0}$};

\draw[name path=circleA,thick] (0,0) circle (0.07);

\draw [->,thick] (0,0)--(-1.5,{sqrt(3)/2});
\node at (-1.5,{sqrt(3)/2+0.2}) {$\alpha_1$};
\draw [->,thick] (0,0)--(1.5,{sqrt(3)/2});
\node at (1.5,{sqrt(3)/2+0.2}) {$\alpha_2$};
\draw [->,thick] (0,0)--(0,{-sqrt(3)});
\node at (-0.3,{-sqrt(3)}) {$\alpha_0$};

\draw [->,thick,blue] (0,0)--(-0.5,{sqrt(3)/2});
\node [blue]  at (-0.5,{sqrt(3)/2+0.2}) {$\varpi_1$};
\draw [->,thick,blue] (0,0)--(0.5,{sqrt(3)/2});
\node [blue] at (0.5,{sqrt(3)/2+0.2}) {$\varpi_2$};

\foreach \x/\y in {0/1,0/-1, 0.75/0.5,-0.75/-0.5,-0.75/0.5, 0.75/-0.5}
{\node[gray] at (\x,{\y*sqrt(3)/2}) {\small 0};
\node[gray] at (\x,{(\y+2)*sqrt(3)/2}) {\small 0};
\node[gray] at ({\x+1.5},{(\y+1)*sqrt(3)/2}) {\small 0};
};

\foreach \x/\y in {0/1,0/-1, 0.75/0.5,-0.75/-0.5,-0.75/0.5, 0.75/-0.5}
{\node[gray] at ({\x+1},{\y*sqrt(3)/2}) {\small 1};
\node[gray] at ({\x+1},{(\y+2)*sqrt(3)/2}) {\small 1};
};
\foreach \x/\y in {0/0,0/2,1.5/1,1.5/-1}
{\node[gray] at ({\x+0.5},{\y*sqrt(3)/2}) {\small 2};
\node[gray] at ({\x-0.25},{\y*sqrt(3)/2+sqrt(3)/4}) {\small 2};
\node[gray] at ({\x-0.25},{\y*sqrt(3)/2-sqrt(3)/4}) {\small 2};
}

\node at (0,0.55) {\tiny$\fA$};
\node at (2,0.68) {\tiny$\fA-\lambda$};

\draw [->,thick,red] (0, {sqrt(3)*1/3})--(0,{sqrt(3)*2/3})--(0.5,{sqrt(3)*5/6})
--(1,{sqrt(3)*2/3})--(1,{sqrt(3)*1/3})--(1.5,{sqrt(3)*1/6})--(2,{sqrt(3)*1/3});
\end{tikzpicture}
\end{minipage}
\hspace{0.3cm}
\begin{minipage}{7.6cm}

$
\begin{array}{lll}
i_1=0,&\beta_1=\alpha_0=-(\alpha_1+\alpha_2)\\
i_2=1,&\beta_2=\bar{s_0}(\alpha_1)=-\alpha_2\\
i_3=0,&\beta_3=\bar{s_0}\bar{s_1}(\alpha_0)=\alpha_1\\
i_4=1,&\beta_4=\bar{s_0}\bar{s_1}\bar{s_0}(\alpha_1)=\alpha_1+\alpha_2\\
i_5=2,&\beta_5=\bar{s_0}\bar{s_1}\bar{s_0}\bar{s_1}(\alpha_2)=\alpha_1\\
i_6=1,&\beta_6=\bar{s_0}\bar{s_1}\bar{s_0}\bar{s_1}\bar{s_2}(\alpha_1)=-\alpha_2\\
\end{array}
$
\vspace{0.2cm}

\hspace{0.15cm}$s_0=\hat{r}_{H_{-\beta_1,1}}$

\hspace{0.15cm}$s_0 s_1 s_0=\hat{r}_{H_{-\beta_2,1}}$

\hspace{0.15cm}$(s_0 s_1) s_0 (s_1 s_0)=s_1=\hat{r}_{H_{-\beta_3,0}}$

\hspace{0.15cm}$(s_0 s_1 s_0) s_1(s_0 s_1 s_0)= s_0=\hat{r}_{H_{-\beta_4,-1}}$

\hspace{0.15cm}$(s_0 s_1 s_0 s_1) s_2 (s_1 s_0 s_1 s_0) = \hat{r}_{H_{-\beta_5,1}}$

\hspace{0.15cm}$(s_0 s_1 s_0 s_1 s_2) s_1 (s_2 s_1 s_0 s_1 s_0) = \hat{r}_{H_{-\beta_6,2}}$
\vspace{0.2cm}

\end{minipage}

\noindent
We can choose another $\lambda$-chain for $\lambda=2\varpi_1-2\varpi_2$ by using the reduced decomposition $v_{-\lambda}=s_1 s_0 s_2 s_1$.
The corresponding reduced $\lambda$-chain is
$(\alpha_1,0),( -\alpha_2,1),(\alpha_1,1),(-\alpha_2, 2)$.

\end{example}

\begin{lemma}\label{-lambda_chain}
 \cite[Remark 5.5]{lenart.postnikov:affine}
 Let $L=(\beta_1,\ldots,\beta_l)$ be a $\lambda$-chain. Then
 $\bar{L}:=(-\beta_l,\ldots,-\beta_1)$ is a $(-\lambda)$-chain.
 If $H_{-\beta_j, d_j}$ is the $j$-th hyperplane of the alcove path from $\fA$ to $\fA-\lambda$ determined by $L$, then the $j$-th hyperplane 
 of alcove path from $\fA$ to $\fA+\lambda$ determined by
 $\bar{L}$ is $H_{\beta_{l+1-j}, \langle \lambda, \beta_{l+1-j}^\vee \rangle -d_{l+1-j}}$.
 If $L$ is reduced, then $\bar{L}$ is also reduced.
\end{lemma}

\subsection{$\lambda$-chain formulae}\label{ss:lambda-chain}
Next we state the main theorem of this section. For any root hyperplane $h=H_{\beta,k}$, let $r_h$ denote the reflection along the hyperplane $H_{\beta,0}$, and $\hat{r}_h$ be the reflection along $h$.

Assume $\lambda$ is an integral weight and 
fix a reduced $\lambda$-chain 
$(\beta_1, \beta_2,\ldots, \beta_l)$, which corresponds to an alcove walk from $\fA$ to $\fA - \lambda$, with separating hyperplanes 
$H_{-\beta_j, d_j}=:h_j$.

For a subset $J=\{j_1<j_2<\cdots <j_t\}\subset \{1,2,\ldots, l\}$, define the following relation
$$
u\stackrel{J_{>}}
{\longrightarrow} w \stackrel{def.}{\iff}
u<u r_{h_{j_t}}<u r_{h_{j_t}}r_{h_{j_{t-1}}}<\cdots <u r_{h_{j_{t}}}\cdots r_{h_{j_1}}=w,
$$ 
and let 
\begin{equation}\label{E:hatr}\hat{r}_{J_{<}}:=  \hat{r}_{h_{j_1}}\cdots \hat{r}_{h_{j_t}}.\end{equation}
Hence, $w\stackrel{J}{\Longrightarrow} u $ from \Cref{cond_M} is equivalent to $u\stackrel{J_{>}}
	{\longrightarrow} w$.
Therefore, \Cref{corol:Chev} can be restated as follows.
\begin{theorem}\label{thm:lambda-chain1} 
In the above setting,
\begin{equation}\label{trans:lambda1}
c_{u,\mu}^{w,-\lambda}=
\sum
(-1)^{n(J)} (1-q)^{|J|} q^{\frac{\ell(u)-\ell(w)-|J|}{2}},
\end{equation}
where the sum is over subsets $J\subset \{ 1,2,\ldots, l \}$ such that 
$u\stackrel{J_{>}}{\longrightarrow} w$ and $w \hat{r}_{J_{<}}(-\lambda)=\mu$,
and where $n(J):=\#\{ j\in J\mid \beta_j<0\}$.
\end{theorem}

\begin{proof}
	In the summation in \Cref{corol:Chev}, we let $J\subset  \{ 1,2,\ldots, l \}$ be the indices of the hyperplanes in $\calM$. Then by \Cref{-lambdaTo_alcove_walk}, $f^+(\calM)=n(J)$. This finishes the proof.
\end{proof}

Using the transformation between a $\lambda$-chain and a $(-\lambda)$-chain, we can get rid of the negative sign in front of $\lambda$ in \Cref{thm:lambda-chain1}  as follows.

First of all,  $(-\beta_l, -\beta_{l-1},\ldots, -\beta_1)$ is a $(-\lambda)$-chain, which corresponds to an alcove walk from $\fA$ to $\fA + \lambda$, with the $j$-th separating hyperplane being 
\begin{equation}\label{E:h'} h'_j:=H_{\beta_{l+1-j}, \langle \lambda, \beta^\vee_{l+1-j}\rangle - d_{l+1-j}} \/. \end{equation}
Then $r_{h'_j}=r_{h_{l+1-j}}$. Let $\tilde{r}_{h_{l+1-j}}:=\hat{r}_{h'_j}$ be the reflection along $h'_j$.
Define
\[
u\stackrel{J_{<}}{\longrightarrow} w  \stackrel{def.}{\iff}
u<u r_{h_{j_1}}<u r_{h_{j_1}}r_{h_{j_2}}< \cdots <
u r_{h_{j_1}}\cdots r_{h_{j_t}}=w \/,\]
and
\begin{equation}\label{E:defrtilde}
\tilde{r}_{J_{>}}:= \tilde{r}_{h_{j_t}}\cdots \tilde{r}_{h_{j_1}}.
\end{equation}

\begin{theorem}\label{thm:lambda-chain2} 
The following holds,
	\begin{equation}\label{trans:lambda2}
		c_{u,\mu}^{w,\lambda}=
		\sum
		(-1)^{n(J)} (q-1)^{|J|} q^{\frac{\ell(u)-\ell(w)-|J|}{2}},
	\end{equation}
	where the sum is over subsets $J\subset \{ 1,2,\ldots, l \}$ such that 
	$u\stackrel{J_{<}}{\longrightarrow} w$ and $w \tilde{r}_{J_{>}}(\lambda)=\mu$,
	where $n(J):=\#\{ j\in J\mid \beta_j<0\}$.
\end{theorem}

\begin{proof}
Applying \Cref{thm:lambda-chain1} to the $(-\lambda)$-chain $(-\beta_l, -\beta_{l-1},\ldots, -\beta_1)$, we get
\[c_{u,\mu}^{w,\lambda}=\sum (-1)^{\#\{j_i\mid -\beta_{l+1-j_i}<0\}}(1-q)^tq^{\frac{\ell(u)-\ell(w)-t}{2}},\]
where the summation is over subsets $J':=\{1\leq j_1<j_2<\dots<j_t\leq l\}$, such that 
\begin{equation}\label{equ:cond1}
	u<ur_{h'_{j_t}}<\cdots<ur_{h'_{j_t}}\cdots r_{h'_{j_1}}=w
\end{equation}
and 
\begin{equation}\label{equ:cond2}
	w\hat{r}_{h'_{j_1}}\cdots \hat{r}_{h'_{j_t}}(\lambda)=\mu.
\end{equation}
Let 
\[J:=\{1\leq l+1-j_t<l+1-j_{t-1}<\cdots <l+1-j_1\leq l\}.\]
Then 
\[(-1)^{\#\{j_i\mid -\beta_{l+1-j_i}<0\}}(1-q)^t=(-1)^{n(J)}(q-1)^t,\]
where $n(J)$ is defined in \Cref{thm:lambda-chain1}.
Since $r_{h'_j}=r_{h_{l+1-j}}$, the condition \Cref{equ:cond1} is equivalent to $u\stackrel{J_{<}}{\longrightarrow} w$.
On the other hand, the condition \Cref{equ:cond2} is equivalent to $w \tilde{r}_{J_{>}}(\lambda)=\mu$ as $\tilde{r}_{h_{l+1-j}}=\hat{r}_{h'_j}$. This finishes the proof.
\end{proof}
For further use, we also record the following technical result, which will allow to rewrite
the elements $w\hat{r}_{J_{<}}(-\lambda)$ from \Cref{thm:lambda-chain1} and 
$w \tilde{r}_{J_{<}}$ from \Cref{thm:lambda-chain2}. 
\begin{prop}\label{lambda-chain:lemma}
Consider a $\lambda$-chain $\beta_1,\ldots,\beta_l$. 
For any subsequence $J=\{j_1<j_2<\ldots<j_t\}\subset \{1,2,\ldots,l\}$, we have 
$$-\hat{r}_{J_<}(-\lambda)=r_{J} \; \tilde{r}_{J_>}(\lambda) , \text{ where } {r}_{J}={r}_{h_{j_1}} {r}_{h_{j_2}}\cdots {r}_{h_{j_t}}.
$$
\end{prop}
\begin{proof} Let $m_j=\langle \lambda, \beta_j^\vee\rangle$ ($1\leq j\leq l$).
By induction on $|J|$, it is easy to show the following three equalities:
$$\begin{array}{lcl}
\hat{r}_{h_{j_1}} \hat{r}_{h_{j_2}}\cdots  \hat{r}_{h_{j_t}}(-\lambda)&=&
-\lambda+(m_{j_1}-d_{j_1})\beta_{j_1}+(m_{j_2}-d_{j_2})r_{h_{j_1}}\beta_{j_2}+\cdots\\
& &+(m_{j_t}-d_{j_t})r_{j_1}r_{h_{j_2}}\cdots r_{h_{j_{t-1}}}\beta_{j_t},\\
\tilde{r}_{h_{j_t}} \tilde{r}_{h_{j_{t-1}}}\cdots  \tilde{r}_{h_{j_1}}(\lambda)&=&
\lambda-d_{j_t} \beta_{j_t}-d_{j_{t-1}} r_{h_{j_t}}\beta_{j_{t-1}}-\cdots -d_{j_1} r_{h_{j_t}}r_{h_{j_{t-1}}}\cdots r_{h_{j_2}}\beta_{j_1},\\
r_{h_{j_1}} r_{h_{j_2}}\cdots r_{h_{j_t}}(\lambda)&=&\lambda-m_{j_1}\beta_{j_1}-m_{j_2} r_{h_{j_1}}\beta_{j_2}-\cdots-
m_{j_t} r_{h_{j_1}}r_{h_{j_2}}\cdots r_{h_{j_{t-1}}}\beta_{j_t}.
\end{array}
$$
As $r_J=r_{h_{j_1}} r_{h_{j_2}}\cdots r_{h_{j_t}}$ is a linear transformation, we get
$$r_J \tilde{r}_{J_>}(\lambda)=
r_J (\lambda)+d_{j_t}r_{h_{j_{1}}}\cdots r_{h_{j_{t-1}}} \beta_{j_t}+d_{j_{t-1}}  r_{h_{j_{1}}}\cdots r_{h_{j_{t-2}}} \beta_{j_{t-1}}
+\cdots +d_{j_1} \beta_{j_1}=-\hat{r}_{J_<}(-\lambda).$$
\end{proof}

\begin{remark}\label{rem:w-to-u}
With the notation as above, 
the condition that 
$u\stackrel{J_{>}}{\longrightarrow} w$ in \Cref{thm:lambda-chain1} 
implies that $ur_J^{-1}=w$, thus by \Cref{lambda-chain:lemma},
\[ w \hat{r}_{J_{<}}(-\lambda) = -u \tilde{r}_{J_{>}}(\lambda) \/. \]
Similarly, in \Cref{thm:lambda-chain2}, we have that $w=ur_J$, 
thus
\[ w\tilde{r}_{J_{>}}(\lambda)= -u \hat{r}_{J_<}(-\lambda) \/.\]
This leads to alternative formulae in the aforementioned theorems 
\ref{thm:lambda-chain1} and \ref{thm:lambda-chain2}.
\end{remark}

In the appendix \S \ref{ex:A_2-1} below we included a fully worked out example illustrating calculations
of some coefficients $c_{u,\mu}^{w,\lambda}$ utilizing \Cref{thm:lambda-chain1} and \Cref{thm:lambda-chain2}.

\section{Motivic Chern classes of Schubert cells}\label{sec:MCclasses}
In this section, we recall basic properties of the motivic Chern classes of the Schubert cells in the (partial) flag varieties.

\subsection{Definition of the Motivic Chern classes}
Let $X$ be a quasi-projective complex algebraic variety, and let $G_0(var/X)$ be the (relative) Grothendieck group of varieties over $X$. It consists of isomorphism classes of morphisms $[f: Z \to X]$ modulo the scissor relations. Brasselet, Sch{\"u}rmann and Yokura \cite{brasselet.schurmann.yokura:hirzebruch} defined the {\em motivic Chern transformation} $MC_y: G_0(var/X) \to K(X)[y]$ with values in the K-theory group of coherent sheaves in $X$ to which one adjoins a formal variable $y$. The transformation $MC_y$ is a group homomorphism, it is functorial with respect to proper push-forwards, and if $X$ is smooth it satisfies the normalization condition \[ MC_y[id_X: X \to X] = \sum [\wedge^j T^*_X] y^j \/. \] Here $[\wedge^j T^*_X]$ is the K-theory class of the bundle of degree $j$ differential forms on $X$. As explained in \cite{brasselet.schurmann.yokura:hirzebruch}, the motivic Chern class is related by a Hirzebruch--Riemann--Roch type statement to the Chern--Schwartz--MacPherson (CSM) class 
in the homology of $X$; see \Cref{CSM:def}.

There is also an equivariant version of the motivic Chern class transformation, which uses equivariant varieties and morphisms, and has values in the suitable equivariant K-theory group. Its definition was given in \cite{AMSS:motivic,feher2021motivic}, {following closely  the approach of \cite{brasselet.schurmann.yokura:hirzebruch}.  

Assume $X$ is smooth and there is a torus $T$ acting on $X$. Let $K_T(X)$ denote the equivariant K-theory of $X$, see \cite{chriss2009representation}. If $X$ is a point, $K_T(pt)=K^0(\Rep(T))=\bbZ[T]$. For any $\calF\in K_T(X)$, let 
\[\chi_T(X,\calF):=\sum_i(-1)^i H^i(X,\calF)\in K_T(pt).\] 
Let $\langle-,-\rangle$ be the non-degenerate pairing on $K_T(X)$ defined by 
\[\langle\calF,\calG\rangle=\chi_T(X,\calF\otimes \calG)\in K_T(pt),\]
where $\calF,\calG\in K_T(X)$. For a vector bundle $E$, the $\lambda_y$-class of $E$ is the class \[ \lambda_y(E):= \sum_k [\wedge^k E] y^k \in K_T(X)[y] \/.\] The $\lambda_y$-class is multiplicative, i.e. for any short exact sequence $0 \to E_1 \to E \to E_2 \to 0$ of equivariant vector bundles there is an equality $\lambda_y(E) = \lambda_y(E_1) \lambda_y(E_2)$ as elements in $K_T(X)[y]$. 

Recall the (relative) motivic Grothendieck group ${G}_0^T(var/X)$ of varieties over $X$ is the free abelian group generated by the isomorphism classes $[f: Z \to X]$ where $Z$ is a quasi-projective $T$-variety and $f: Z \to X$ is a $T$-equivariant morphism modulo the usual additivity relations 
$$[f: Z \to X] = [f: U \to X] + [f:Z \setminus U \to X]$$ for $U \subset Z$ an open $T$-invariant subvariety. For any equivariant morphism $f:X \to Y$ of quasi-projective $T$-varieties there are well defined push-forwards $f_!: G_0^T(var/X) \to G_0^T(var/Y)$ given by composition.
The equivariant motivic Chern class is defined by the following theorem.
\begin{theorem}\cite{AMSS:motivic, feher2021motivic}\label{thm:existence} Let $X$ be a quasi-projective, non-singular, complex algebraic variety with an action of the torus $T$. There exists a unique natural transformation $MC_y: G_0^T(var/X) \to K_T(X)[y]$ satisfying the following properties:
\begin{enumerate} \item[(1)] It is functorial with respect to $T$-equivariant proper morphisms of non-singular, quasi-projective varieties. 

\item[(2)] It satisfies the normalization condition \[ MC_y[id_X: X \to X] = \lambda_y(T^*_X) = \sum y^i [\wedge^i T^*_X]_T \in K_T(X)[y] \/. \]
\end{enumerate}
The transformation $MC_y$ satisfies a Verdier--Riemann--Roch (VRR) formula: for any smooth, $T$-equivariant morphism $\pi: X \to Y$ of quasi-projective and non-singular algebraic varieties, and any $[f: Z \to Y] \in G_0^T(var/Y)$, the following holds:
\[ \lambda_y(T^*_\pi) \cap \pi^* MC_y[f:Z \to Y] = MC_y[\pi^* f:Z \times_Y X \to X] \/. \]
\end{theorem}
Define the Grothendieck--Serre dual operator $\caD: K_T(X) \to K_T(X)$ as follows. For any $[\caF]\in K_T(X)$, define
\begin{equation}\label{E:SerreD}\caD[\caF]:=[RHom(\caF,\omega^\bullet_{X})]:=[\omega_{X}^\bullet \otimes \caF^\vee ]\in K_T(X)\/,\end{equation}
where $\omega^\bullet_{X}\simeq \omega_{X}[\dim X]$ is the {(equivariant)} dualizing complex of $X$ (the canonical bundle $\omega_{X}$ shifted by dimension). The class $[\caF^\vee]$ is obtained by taking an equivariant resolution of $\caF$ by vector bundles, and then taking duals. Extend the operators $\caD$ and $(-)^\vee$ to $K_T(X)[y^{\pm1}]$ by sending  $y \mapsto y^{-1}$. We also let $(-)^\vee$ denote the operator on $K_T(pt)[y^{\pm1}]$, which sends $e^\lambda$  to $e^{-\lambda}$ and $y$ to $y^{-1}$.
\begin{defin}\label{def:smc}
Assume $\Omega \hookrightarrow X$ is a $T$-stable subvariety.
\begin{enumerate}
\item
The {\em motivic Chern} (MC) class of $\Omega$ is 
\[ MC_y(\Omega) :=MC_y[\Omega \hookrightarrow X] \in K_T(X)[y].\] 
\item 
If $\Omega$ is pure dimensional, the {\em Segre motivic class} (SMC) $SMC_y(\Omega) \in K_T(X)[[y]]$ is the class 
\[ SMC_y(\Omega):= (-y)^{\dim \Omega}\frac{\caD(MC_y(\Omega))}{\lambda_y(T^*_X)} \in K_T(X)[[y]]\/.\]
\end{enumerate}
\end{defin}

\subsection{Complete flag variety case}

Both $\{MC_y(X(w)^\circ)\mid w\in W\}$ and $\{SMC_y(Y(w)^\circ)\mid w\in W\}$ are bases for the localized equivariant K-theory 
\[ K_T(G/B)[[y]]_{\loc}:=K_T(G/B)[[y]]\otimes_{K_T(pt)}\Frac K_T(pt)\/.\]
These classes can be calculated recursively using the Demazure--Lusztig operators as follows. 

The left multiplication action of $G$ on $G/B$ induces a left Weyl group action on $K_T(G/B)$. For any $w\in W$, let $w^L$ denote this action. 
\begin{defin}\cite[Section 5.3]{mihalcea2020left}
 For any simple reflection $s_i$, we define the following left Demazure--Lusztig operator on 
 $K_T(G/B)_{\loc}$:
 \begin{equation}\label{equ:leftDL}
 	\calT^L_i:=\frac{1+ye^{-\alpha_i}}{1-e^{-\alpha_i}}s_i^L-\frac{1+y}{1-e^{-\alpha_i}}.
 \end{equation}
\end{defin}
These operators enjoy the following properties.
\begin{lemma}\cite{mihalcea2020left}\label{lem:leftDL}
	The operators $	\calT^L_i$ satisfy the braid relation and the quadratic relation 
	\[(\calT^L_i+1)(\calT^L_i+y)=0.\]
	Moreover, they commute with the operators of tensoring by elements in $K_G(G/B)$.
\end{lemma}
Moreover, it is easy to verify the following lemma. 
\begin{lemma}\label{lem:heckeactionk}
	The following map $\Psi$ defines an action of the affine Hecke algebra $\bbH$ (with $q=-y$) on 
$K_T(G/B)_{\loc}$,
	\[T_i\mapsto \calT^L_i, \textit{\quad and \quad} X^\lambda\mapsto e^\lambda,\]
	where $e^\lambda\in K_T(\pt)$ acts on 
$K_T(G/B)_{\loc}$,
by multiplication.
\end{lemma}

We have the following theorem. 
\begin{theorem}\label{thm:MCBorel} 
\begin{enumerate}
\item\cite[Theorem 7.6]{mihalcea2020left}
For $w \in W$ and a simple root $\alpha_i$,  
\[ \mathcal{T}^L_i (MC_y(X(w)^\circ)) =\begin{cases} MC_y( X(s_iw)^\circ) & \textrm{ if } s_iw>w ;\\ -(y+1) MC_y(X(w)^\circ ) - y MC_y(X(s_iw )^\circ ) & \textrm{ if } s_iw < w \/. \end{cases}\]
In particular, 
\[MC_y(X(w)^\circ)=\calT^L_{w}([\calO_{X(id)}]).\]
\item\cite[Theorem 7.1]{mihalcea2020left}
For any $w,u\in W$: 
\[ \langle MC_y(X(w)^\circ), SMC_y(Y(u)^\circ) \rangle =  \delta_{u,w} \/. \]
\end{enumerate}
\end{theorem}
\begin{remark}\label{rem:y=0}
\begin{enumerate}
\item 
By \cite[Thm. 5.1 and Cor. 5.3]{AMSS:specializations},
\[MC_0(X(w)^\circ)=[\calO_{X(w)}(-\partial X(w))],\textit{\quad and \quad}SMC_0(Y(w)^\circ)=[\calO_{Y(w)}],\]
where $\partial X(w):=X(w)\setminus X(w)^\circ$. Thus, the duality in the second part of the theorem reduces to the classical fact
\[\langle [\calO_{X(w)}(-\partial X(w))], [\calO_{Y(u)}] \rangle =  \delta_{u,w}.\]
\item 
By definition, for any $w\in W$, 
\[MC_y(Y(w))=\sum_{u\geq w}MC_y(Y(u)^\circ).\]
Besides, by the linearity of the Grothendieck--Serre dual operator $\caD$,
\[SMC_y(Y(w))=\sum_{u\geq w}(-y)^{\ell(u)-\ell(w)}SMC_y(Y(u)^\circ).\]
Therefore, 
\[SMC_0(Y(w))=[\calO_{Y(w)}].\]
\end{enumerate}

\end{remark}

\subsection{Partial flag variety case}
For any $w\in W$, let $\ell(wW_P)$ denote the length of the minimal length representative in $wW_P$. Let $\pi:G/B \to G/P$ be the natural projection. It is proved in \cite[Remeark 5.7]{AMSS:motivic} that \begin{equation}\label{E:MCpush} \pi_* MC_y(X(w)^\circ) = (-y)^{\ell(w) - \ell(w W_P)} MC_y(X(wW_P)^\circ) \/.\end{equation} In particular, if $w$ is a minimal length representative, then $\pi_* MC_y(X(w)^\circ) = MC_y(X(wW_P)^\circ)$; this also follows directly from the functoriality property of the motivic classes. 
\begin{prop}\label{prop:segrepb}
\begin{enumerate}
\item \cite[Proposition 6.3]{mihalcea2020left}
Let $\Omega \subset G/P$ be a $T$-stable subvariety of $G/P$ and $\pi:G/B \to G/P$ be the projection. Then 
\[\pi^*SMC_y(\Omega) =  SMC_y(\pi^{-1} \Omega) \in K_T(G/B) [[y]] \/.\]
\item\cite[Theorem 7.2]{mihalcea2020left}
 Let $u,w \in W^P$. The Segre motivic classes are dual to the motivic Chern classes for any $G/P$, i.e. \[ \langle MC_y(X(wW_P)^\circ), SMC_y(Y(uW_P)^\circ) \rangle = \delta_{u,w} \/. \]
\end{enumerate}
\end{prop} 
\begin{remark}\label{rem:Py=0}
By the first property in the proposition, for any $w\in W^P$,
\[\pi^*(SMC_y(Y(wW_P)))=SMC_y(Y(w)).\]
Letting $y=0$, we get
\[\pi^*(SMC_0(Y(wW_P)))=SMC_0(Y(w))=[\calO_{Y(w)}],\]
where the second equality follows from Remark \ref{rem:y=0}. By the definition of the SMC class, 
\begin{equation}\label{equ:SMCSchubertP}
SMC_y(Y(wW_P))=\sum_{u\in W^P, u\geq w}(-y)^{\ell(u)-\ell(w)}SMC_y(Y(uW_P)^\circ).
\end{equation}
Hence, 
\[\pi^*(SMC_0(Y(wW_P)^\circ))=\pi^*(SMC_0(Y(wW_P)))=[\calO_{Y(w)}].\]
Since $\pi^*([\calO_{Y(wW_P)}])=[\calO_{Y(w)}]$ and $\pi^*$ is injective, we get
\[SMC_0(Y(wW_P)^\circ)=SMC_0(Y(wW_P))=[\calO_{Y(wW_P)}].\]
Combining with the second part of the proposition, we get
\[MC_0(X(wW_P)^\circ)=[\calO_{X(wW_P)}(-\partial X(wW_P))].\]
This can also be proved using Remark \ref{rem:y=0}(1), Equation \eqref{E:MCpush}, and the fact that the pushforward of an ideal sheaf is an ideal sheaf, see \cite{brion:Kpos}.
\end{remark}

\section{Chevalley formulae for the motivic Chern classes}\label{sec:che}
In this section, we obtain several Chevalley formulae for the motivic Chern classes, 
in terms of alcove walks,  $\lambda$-chains, and certain operators. The main technique is
to reinterpret formulae from Hecke algebras such as \Cref{thm:lambda-chain1} in terms of multiplications
of motivic Chern classes by line bundles.  
Our main results are \Cref{thm:lambda} and \Cref{thm:operator}.
In \S \ref{sec:positivity}
we discuss several positivity properties and conjectures of the Chevalley coefficients.
Finally, in \S \ref{sec:parabolic} we discuss parabolic Chevalley formulae. 

\subsection{Chevalley coefficients}\label{sec:checoeff}
Consider a torus weight $\lambda\in X^*(T)$ and $u, w\in W$. The Chevalley coefficient $C_{u,\lambda}^w$ 
is defined by the following formula:
\begin{equation}\label{Chev:MC}
	\calL_\lambda\otimes MC_y(X(w)^\circ)=\sum_{u\leq w} C_{u,\lambda}^w MC_y(X(u)^\circ).
\end{equation}
Note that for any simple reflection $s_i$ there is a short exact
sequence of equivariant sheaves
\[ 0 \to \mathcal{L}_{\varpi_i} \otimes \C_{-w_0(\varpi_i)} \to \cO_{G/B} \to \cO_{X(w_0 s_i)} \to 0 \]
with $\varpi_i$ the fundamental weight,
see e.g., \cite[\S 8]{buch.m:nbhds}. Therefore,  
\[[\calO_{X(w_0s_i)}]=1-e^{-w_0(\omega_i)}\calL_{\varpi_i}\in K_T(G/B) \/,\]
and the coefficients
from \eqref{Chev:MC} for $\lambda=\varpi_i$ also recover the multiplication
of $[\calO_{X(w_0s_i)}]$ with the MC classes of the Schubert cells.

The coefficients $C_{u,\lambda}^w$ also arise from Chevalley formulae involving Segre motivic classes:
\begin{lemma}\label{lem:mcsmc} The following equation holds:
\begin{equation}\label{Chev:SMC}
\calL_\lambda\otimes SMC_y(Y(u)^\circ)=\sum_{w\geq u} C_{u,\lambda}^w SMC_y(Y(w)^\circ).
\end{equation}
\end{lemma}
\begin{proof} This follows from the (Poincar{\'e}) duality in \Cref{thm:MCBorel}(2), as the Chevalley coefficients in 
\Cref{Chev:MC} and \Cref{Chev:SMC} are given by  
\[ C_{u,\lambda}^w=\langle \calL_\lambda\otimes MC_y(X(w)^\circ), SMC_y(Y(u)^\circ)\rangle \/. \]
\end{proof}
We now relate the Chevalley coefficients above to the coefficients $c_{u,\mu}^{w,\lambda}$ from 
\Cref{equ:matrix2} in the Hecke algebra.
\begin{theorem}\label{thm:chemc} Let $\lambda$ be any weight in $X^*(T)$.
The following Chevalley formula holds in $K_T(G/B)[y]$:
\[\calL_{-\lambda}\otimes MC_y(X(w)^\circ)=\sum_{\mu\in X^*(T),u\in W}y^{\ell(w)-\ell(u)}e^{-\mu}c_{u,\mu}^{w,\lambda}|_{q=-y}MC_y(X(u)^\circ).\]
In particular, the following equation holds for the Chevalley coefficients in \eqref{Chev:MC}:
	 \begin{equation*}
	C_{u,-\lambda}^w=\sum_{\mu\in X^*(T)}y^{\ell(w)-\ell(u)}e^{-\mu}c_{u,\mu}^{w,\lambda}|_{q=-y}.
\end{equation*}
\end{theorem}
From \Cref{Chev:MC}, the Chevalley coefficients 
$C_{u,\lambda}^v$ belong to a localization 
$K_T(pt)[y]$ which allows division by 
$1 + e^\alpha y$ for any root $\alpha$. However, the expansion from 
\Cref{equ:matrix2} implies that the coefficients $(-q)^{\ell(w) - \ell(v)} c_{u,\mu}^{w,\lambda}$
are {\em polynomials} in $\bZ[q]$. 
Then it follows that $C_{u,\lambda}^v$ are in fact polynomials in 
$K_T(pt)[y]$. This will be seen explicitly in \Cref{cor:alcovewalk} below.
\begin{proof}[Proof of \Cref{thm:chemc}]
	Applying the map $\Psi$ in \Cref{lem:heckeactionk} to \Cref{equ:matrix2}, we get
		 \begin{equation*}
		\calT^L_{w}e^{-\lambda}=\sum_{\mu\in X^*(T),u\in W}y^{\ell(w)-\ell(u)}e^{-\mu}c_{u,\mu}^{w,\lambda}|_{q=-y}\calT^L_{u}\in \End_{\bbC}K_T(G/B)[y].
	\end{equation*}
Then the theorem follows by applying both sides to $[\calO_{X(id)}]$, and 
utilizing that
\begin{align*}
	\calT^L_{w}e^{-\lambda}([\calO_{X(id)}])=&\calT^L_{w}(\calL_{-\lambda}\otimes[\calO_{X(id)}])\\
	=&\calL_{-\lambda}\otimes MC_y(X(w)^\circ).
\end{align*}
Here the first equality follows from $\calL_{-\lambda}\otimes[\calO_{X(id)}]=e^{-\lambda}[\calO_{X(id)}]$, while the second one follows from \Cref{lem:leftDL} and \Cref{thm:MCBorel}. 
\end{proof}
\begin{remark}\label{rmk:bundle-mult} The above argument can be generalized to the case when 
the line bundle $\calL_{\lambda}$ is replaced by any homogeneous bundle 
$\mathcal{V} =G\times^B V \to G/B$ associated to a 
$B$-representation of $V$. If the character of $V$ is 
$ch(V) = \sum_\lambda a_\lambda e^{\lambda}$, then a localization
argument shows that the class of $\mathcal{V}$ in $K_T(G/B)$ 
is equal to 
\[ [\mathcal{V}] = \sum a_\lambda \mathcal{L}_\lambda \/. \] 
It follows that for any $w \in W$, 
\[ \mathcal{V}  \otimes MC(X(w)^\circ) = \sum_\lambda a_\lambda \mathcal{L}_\lambda   \otimes MC(X(w)^\circ)= \sum_{u} \sum_\lambda  a_\lambda C_{u,\lambda}^w MC(X(u)^\circ) \/. \]

We illustrate this for $G/B=\Fl(n)$, the complete flag manifold. This is equipped
with the tautological sequence $\mathcal{F}_0 = 0 \subset \mathcal{F}_ 1 \subset \mathcal{F}_2 \subset \ldots \subset \mathcal{F}_n= \C^n$. For $1 \le i \le n-1$ define $X_i= \mathcal{F}_i/\mathcal{F}_{i-1}$ 
regarded in $K_T(G/B)$. Then  
\[ \wedge^j \mathcal{F}_ i = e_j(X_1, \ldots, X_i)\/, \quad Sym^j \mathcal{F}_ i = h_j(X_1, \ldots, X_i) \]
where $e_j$ and $h_j$ denote the elementary symmetric function, respectively the complete homogeneous
symmetric function. Note that if $\varpi_i$ denotes the $i$th fundamental weight, then
$X_i = \mathcal{L}_{\varpi_i - \varpi_{i-1}}$ for $1 \le i \le n-1$, with the convention that $\varpi_0=0$. 
\Cref{thm:chemc} gives a formula for the multiplication by monomials 
$X_1^{a_1} \cdot \ldots \cdot X_{n-1}^{a_{n-1}}$, which in turn gives formulae to multiply by $e_j(X_1, \ldots, X_i)$ and $h_j(X_1, \ldots , X_i)$. 
\end{remark}
In the next section, we give explicit formulae for the Chevalley coefficients 
$C_{\lambda,w}^{u}$ based
on the formulae for the Hecke algebra coefficients $c_{u,\mu}^{w,\lambda}$. 
\subsection{Chevalley formulae via alcove walks and the $\lambda$-chains}
Let us recall the setting of \Cref{corol:Chev}. For $\lambda\in X^{*}(T)$, 
choose a  minimal length alcove walk 
$p_{v_\lambda}=c_{j_1}^{\epsilon_1}c_{j_2}^{\epsilon_2}\cdots c_{j_s}^{\epsilon_s}$  
from $\fA$ to $\fA+\lambda$, and
let  $\mathcal{H} = \{ h_1,h_2,\ldots,h_s\}$ be the ordered sequence of hyperplanes  
defined by the walls of alcoves crossed by $p_\lambda$.
Define $\beta_i\in R^{\epsilon_i}$
and $k_i\in \mathbb Z$ ($1\leq i\leq s$)
by the condition $h_i=H_{\beta_i, k_i}$.
Combining \Cref{corol:Chev} and \Cref{thm:chemc}, we get the following formula. 
\begin{theorem}[Alcove walk formula for the Chevalley coefficients]\label{cor:alcovewalk} 
\begin{equation*}
C_{u,-\lambda}^{w}= \sum_{\mathcal{M}}
(-1)^{f^{+}(\mathcal{M})}  (-1-y)^{|\mathcal{M}|} (-y)^{\frac{\ell(w)-\ell(u)-|\mathcal{M}|}{2}} 
{e^{-w \hat{r}_{\mathcal{M}}(\lambda)}}
\end{equation*}
where the sum is over ordered subsets $\mathcal{M} \subset \mathcal{H}$ such that
$w\stackrel{\mathcal{M}}{\Longrightarrow} u= w r_{\mathcal{M}}$.
\end{theorem}
We now recall the setup of \Cref{thm:lambda-chain1}. 
Assume $\lambda$ is an integral weight and fix a reduced 
$\lambda$-chain $(\beta_1, \beta_2,\ldots, \beta_l)$, 
which corresponds to an alcove walk from $\fA$ to 
$\fA - \lambda$, with separating hyperplanes $h_j:=H_{-\beta_j, d_j}$. 

Combining \Cref{thm:lambda-chain1} , \Cref{thm:lambda-chain2}, and \Cref{thm:chemc}, we get the following formulae.
\begin{theorem}[$\lambda$-chain formula for the Chevalley coefficients]\label{thm:lambda} 
The following hold:
\begin{equation}\label{equ:-lambda}
	C_{u,-\lambda}^w= \sum_{J\subset \{ 1,2,\ldots, l \}}
	(-1)^{n(J)} (1+y)^{|J|} (-y)^{\frac{\ell(w)-\ell(u)-|J|}{2}}
	e^{-w\tilde{r}_{J_{>}}(\lambda)},
\end{equation}
where the sum is over subsets $J = \{ j_1 < \ldots < j_t \} \subset \{ 1, 2, \ldots, l \}$
such that $u < u r_{h_{j_1}} < u r_{h_{j_1}} r_{h_{j_2}} < \ldots < u r_{h_{j_1}} r_{h_{j_2}} \cdot \ldots \cdot r_{h_{j_t}} =w$ and $\tilde{r}_{J_>}$ is defined in \eqref{E:defrtilde}.
Furthermore,
\begin{equation}\label{equ:lambda}
	C_{u,\lambda}^w= \sum_{J\subset \{ 1,2,\ldots, l \}}
	(-1)^{n(J)}
	(-1-y)^{|J|} (-y)^{\frac{\ell(w)-\ell(u)-|J|}{2}}
	e^{-w \hat{r}_{J_{<}}(-\lambda)},
\end{equation}
where the sum is over subsets $J = \{ j_1 < \ldots < j_t \} \subset \{ 1, 2, \ldots, l \}$
such that $u < u r_{h_{j_t}} < u r_{h_{j_t}} r_{h_{j_{t-1}}} < \ldots < u r_{h_{j_t}} 
\cdot \ldots \cdot r_{h_{j_1}} =w$ and $\hat{r}_{J_{<}}$ is defined in \eqref{E:hatr}.
\end{theorem}
\begin{remark}\label{rem:lambda_chev}
(1) It follows from \Cref{rem:w-to-u} that
$-w\tilde{r}_{J_{>}}(\lambda)= u \hat{r}_{J_<}(-\lambda)$
in \Cref{equ:-lambda}, and that 
$-w \hat{r}_{J_{<}}(-\lambda)= u \tilde{r}_{J_{>}}(\lambda)$
in \Cref{equ:lambda}, giving alternative ways to calculate these.

(2)
By \Cref{rem:y=0} (1), 
the MC and SMC classes specialize at $y=0$ to the 
ideal sheaves and the structure sheaves, respectively. 
Under this specialization, and using that 
$-w\tilde{r}_{J_{>}}(\lambda)= u \hat{r}_{J_<}(-\lambda)$,
\Cref{equ:-lambda}
reduces to the equivariant $K$-theory Chevalley formula of Lenart--Postnikov 
\cite[Theorem 6.1]{lenart.postnikov:affine}.
\end{remark}
One can also consider the more general situation of Kac--Moody flag varieties defined 
e.g. in \cite{kumar:book}. These are ind-varieties, and one can define the motivic Chern 
classes of the finite-dimensional Schubert cells using the ind-structure. 
There are analogues of the (left and right) Demazure operators, and 
the notion of $\lambda$-chains extends to this setting, by results of Lenart and Postnikov 
\cite{lenart2008combinatorial}. Note that for an infinite Weyl group $W$, a 
$\lambda$-chain may be an infinite sequence, i.e. $l=\infty$. But, 
for given $u\leq w$, $t=|J|$ is finite, and the number of $J$ which satisfies 
the condition in \Cref{thm:lambda} is also finite. Based on these similarities
to the finite case, we expect the following conjecture to hold.
\begin{conj}\label{con:km}
For a Kac--Moody Weyl group $W$ and a dominant integral weight $\lambda$,
the analogues of the equations \Cref{equ:-lambda} and \Cref{equ:lambda} hold.
\end{conj}

\subsection{Miscelaneous}\label{sec:positivity} In this section 
we discuss positivity properties of 
the coefficients from the Chevalley formula,
{and some special properties of the multiplication
by line bundles given by minuscule weights.}
We start with the following consequence
of \Cref{thm:lambda}.

\begin{prop}\label{prop:lambda-pos} Let $\lambda$ be a dominant weight and set $q=-y$. Then:

(a) $C_{u,\lambda}^w$ may be written as a combination of
terms of the form $e^\mu q^a (q-1)^b$ with non-negative integer coefficients. 

(b) $C_{u,-\lambda}^w$ may be written as a combination of
terms of the form $(-1)^b e^\mu q^a (q-1)^b$ with non-negative integer coefficients.

Furthermore, in both situations $b$ has the same parity as $\ell(w)-\ell(u)$. 
\end{prop}
\begin{proof} Both statements follow from \Cref{thm:lambda}, using that 
for a reduced $\lambda$-chain with $\lambda$ dominant we have $n(J) =0$, and that 
$\ell(w) - \ell(u) - |J|$ is an even integer.  
\end{proof}
\begin{example}\label{exam:chev} Consider
 type $A_2$, with $\lambda=2\varpi_1+\varpi_2$, and with the
$\lambda$-chain from \S \ref{ex:A_2-1}.
Take $w=s_2 s_1$.
Then from \Cref{thm:lambda} we get (with $q=-y$): 
\[ \begin{split} {\mathcal L}_{\lambda}\otimes MC_{-q}(X(s_2 s_1)^\circ )
& =  e^{\varpi_1-3\varpi_2}MC_{-q}(X(s_2 s_1)^\circ)\\
& +(q-1)
( e^{-\varpi_2}+e^{-\varpi_1+\varpi_2}+e^{-2 \varpi_1+3 \varpi_2})
MC_{-q}(X(s_1)^\circ)\\
& +(q-1)(e^{2\varpi_1-2\varpi_2}+e^{3\varpi_1-\varpi_2}) MC_{-q}(X(s_2)^\circ)\\
&+ (q-1)^2 
(e^{2\varpi_1+\varpi_2}+e^{2\varpi_2}+e^{\varpi_1})
MC_{-q}(X(id)^\circ).
\end{split}
\]
\[\begin{split} {\mathcal L}_{-\lambda}\otimes MC_{-q}(X(s_2 s_1)^\circ) & =  
e^{-\varpi_1+3\varpi_2} MC_{-q}(X(s_2 s_1)^\circ)\\
&-
(q-1)
 (e^{\varpi_1-\varpi_2} +e^{\varpi_2}+e^{-\varpi_1+3\varpi_2} )MC_{-q}(X(s_1)^\circ)\\
&-
(q-1)(e^{-2\varpi_1+2\varpi_2} +e^{-\varpi_1+3\varpi_2})MC_{-q}(X(s_2)^\circ)\\
&+
(q-1)^2 (e^{-\varpi_1}+e^{\varpi_1-\varpi_2}+e^{-2\varpi_1+2\varpi_2}+e^{\varpi_2}+e^{-\varpi_1+3\varpi_2} )MC_{-q}(X(id)^\circ) \/. 
\end{split}
\]
\end{example}
We now investigate the special multiplication by $\mathcal{L}_{\pm \varpi_i}$ in the case
of minuscule fundamental weights. In this case the coefficients have a particularly pleasing 
factorization.
\begin{lemma}\label{lemma:coeffs-minuscule}
If $\lambda = \varpi_i$ is a minuscule weight, then for any $u,w\in W$,
\[C_{u,\lambda}^w = e^{u(\lambda)} P_{u,\lambda}^w(y) \quad \textrm{ and } \quad 
C_{u,-\lambda}^w = (-1)^{\ell(w)-\ell(u)} e^{-w(\lambda)} P_{w_0w,\lambda}^{w_0u}(y) \/,\] 
where $P_{u,\lambda}^w(y)\in \bbZ[y]$. Furthermore, these polynomials satisfy
\[P_{u,\lambda}^w(y^{-1}) \cdot y^{\ell(w)-\ell(u)} = P_{u,\lambda}^w(y) \/, \]
i.e., they are palindromic.
\end{lemma}
\begin{proof} Since $\lambda=\varpi_i$ is minuscule, 
for any reduced $\lambda$-chain 
$(\beta_1, \beta_2,\ldots, \beta_l)$, 
the separating hyperplanes must be of the form 
$h_j:=H_{-\beta_j, 0}$, thus $\hat{r}_{h_{j}}=r_{h_{j_1}}$. 
Therefore the first equality follows from 
\Cref{equ:lambda}, since 
$-w \hat{r}_{J_{<}}(-\lambda)=-w r_{J_{<}}(-\lambda)=u(\lambda)$. The second equality 
follows this and from the `Star duality' in \Cref{prop:star} below. 
The palindromic property follows from \Cref{prop:dual}(a) below.
\end{proof}
\begin{remark}\label{rmk:former-conj} In an earlier ar$\chi$iv version of this note,
we conjectured that if $\lambda=\varpi_i$ is a minuscule weight, then 
the coefficients $C_{u,\lambda}^w(y)$ in the multiplication
\[ MC(X(w)^\circ) \cdot \mathcal{L}_\lambda = \sum C_{u,\lambda}^w(y) MC(X(u)^\circ) \/\]
are polynomials with non-negative coefficients. Since then, we found counterexamples
to this conjecture in Lie types $D_6,E_6$ and $A_7$.\end{remark} 

The next example shows that even in the cases when a coefficient $C_{u,\lambda}^w(y)$ happens to have 
positive coefficients, cancellations
may still occur in the formula \eqref{equ:lambda} which calculates it. 
\begin{example}\label{ex:cancellations}
Consider type $A_3$, $\lambda=\varpi_2$, $w=s_1 s_2 s_3 s_1 s_2 s_1$ and $u=s_3 s_1$.
An $\varpi_2$-chain is
$\beta_1=\alpha_2, \beta_2=\alpha_2+\alpha_3, \beta_3=\alpha_1+\alpha_2, 
\beta_4=\alpha_1+\alpha_2+\alpha_3$; see \Cref{ex:v-o2}. We have
two  paths from $u$ to $w$ in \eqref{equ:lambda}:
\begin{itemize} \item $J_1=\{2,3\}$ which gives $u<u s_{\beta_3}<u s_{\beta_3}s_{\beta_2}=w$, 
where $\ell(u s_{\beta_3})+3=\ell(w)$;
\item $J_2=\{1,2,3,4\}$ which gives $u<u s_{\beta_4}<u s_{\beta_4} s_{\beta_3}<
u s_{\beta_4} s_{\beta_3} s_{\beta_2}<u s_{\beta_4} s_{\beta_3} s_{\beta_2} s_{\beta_1}=w$.
\end{itemize}
The path $J_1$ gives coefficient 
$(-1-y)\times (-1-y) (-y) e^{u(\lambda)}=-y(y+1)^2 e^{u(\lambda)}$, 
and the path
$J_2$ gives coefficient $(-1-y)^4 e^{u(\lambda)}=(y+1)^4 e^{u(\lambda)}$.
Therefore 
\[ C_{u,\lambda}^w = e^{u(\lambda)} \left( (y+1)^4-y(y+1)^2
\right)=e^{u(\lambda)} (y^2+y+1)(y+1)^2 \/. \]
\end{example}

\subsection{Operator formula}
In this section, we reformulate the $\lambda$-chain Chevalley formula from \Cref{thm:lambda} via operators generalizing to motivic Chern classes similar ones from \cite{lenart.postnikov:affine}.

Let $h:=(\rho, \theta^\vee)+1$ be the Coxeter number,
where $\rho:=\sum_{i=1}^{r} \varpi_{i}$ and $\theta^\vee$ is the highest coroot.
Let $\tilde{R}(T):=\bbZ[e^{\pm \varpi_1/h},\ldots, e^{\pm \varpi_r/h}]$, and let 
$$\tilde{K}_T(G/B):=K_T(G/B)[y]\otimes_{K_T(\pt)[y]} \Frac(\tilde{R}(T)[y]).$$
Then $\tilde{K}_T(G/B)$ has a basis over $\Frac(\tilde{R}(T)[y])$ given by the motivic Chern classes of the Schubert cells. Define $\Frac(\tilde{R}(T)[y])$-linear operators $B_{\beta}$ $(\beta\in R^{+})$ and $E^\mu$ 
($\mu \in X^*(T)$) on $\tilde{K}_T(G/B)$ by
$$\begin{array}{lll}
	B_{\beta} (MC_y(X(w)^\circ )):=
	\begin{cases}
		(-1-y) (-y)^{\frac{\ell(w)-\ell(w  s_\beta)-1}{2} }MC_y(X(w s_\beta)^\circ ) & \text{ if }w s_\beta<w\\[0.2cm]
		0 & \text{ otherwise}
	\end{cases}, \\
	\\
	E^\mu(MC_y(X(w)^\circ )):=e^{\frac{w(\mu)}{h}} MC_y(X(w)^\circ ).
\end{array}
$$
If $\beta\in R^{-}$, define $B_{\beta}:=-B_{-\beta}$. Then 
\[E^\mu E^{\mu'}=E^{\mu+\mu'},\textit{\quad and \quad} B_{\beta}E^{s_\beta\mu}=E^{\mu}B_{\beta}.\]

Given a $\lambda$-chain  $(\beta_1, \ldots, \beta_l)$, define
$$
R^{[\lambda]}:=R_{\beta_l}R_{\beta_{l-1}}\cdots R_{\beta_1},
\;\text{ where }\;
R_{\beta}:=E^\rho(E^\beta+  B_{\beta})E^{-\rho}=E^\beta+ E^{(\rho,\beta^\vee)\beta} B_{\beta} \/.
$$
Then we have the following operator formula.
\begin{theorem}	[Operator Chevalley  formula]\label{thm:operator}
For any integral weight $\lambda\in X^*(T)$,
	\begin{equation}\label{operator_lambda_MC}
		\calL_{\lambda}\otimes MC_y(X(w)^\circ)=R^{[\lambda]} \left( MC_y(X(w)^\circ) \right).
	\end{equation}
\end{theorem}
\begin{remark}
	\begin{enumerate}
		\item 
		The theorem implies that the definition of $R^{[\lambda]}$ does not depend on the choice of the $\lambda$-chain, which is equivalent to the Yang--Baxter equations (\cite[Definition 9.1]{lenart.postnikov:affine}) satisfied by the operators $R_\beta$. 
		\item The formula analogous to \eqref{operator_lambda_MC} involving SMC classes is obtained by replacing the operator $R^{|\lambda|}$ by an operator defined via the adjoint operators of $B_\beta$ and $E^\mu$.
		\item 
		Recall from \Cref{rem:y=0}(1) that $MC_0(X(w)^\circ)=[\calO_{X(w)}(-\partial X(w))]$. Specializing $y=0$ in the Theorem, we get a dual version of Lenart--Postnikov's formula  \cite[Theorem 13.1]{lenart.postnikov:affine}.
	\end{enumerate}
\end{remark}

For the proof, we need to recall the following result from \cite{lenart.postnikov:affine}. Recall $(\beta_1, \ldots, \beta_l)$ is a $\lambda$-chain. The hyperplane $h_j:=H_{-\beta_j,d_j}$ separates the alcoves $A_{j-1}$ and $A_j$ in the corresponding alcove path, and $\hat{r}_{h_j}$ is the reflection along $h_j$.
\begin{lemma}\cite[Proof of Prop. 14.5]{lenart.postnikov:affine}\label{lem:lp}
	For any $1\leq j_1<j_2<\cdots <j_t\leq l$,
	\begin{align*}
	&-\rho+\beta_1+\cdots+\beta_{j_1-1}+s_{\beta_{j_1}}(\beta_{j_1+1}+\cdots \beta_{j_2-1})+\cdots+ s_{\beta_{j_1}}\cdots s_{\beta_{j_t}}(\beta_{j_t+1}+\cdots \beta_l+\rho)\\
	=&-h\hat{r}_{j_1}\cdots \hat{r}_{j_t}(-\lambda).
	\end{align*}
\end{lemma}

\begin{proof}[Proof of \Cref{thm:operator}]
By definition, 
\begin{align*}
	R^{[\lambda]}=&R_{\beta_l}R_{\beta_{l-1}}\cdots R_{\beta_1}\\
	=&E^\rho(E^{\beta_l}+  B_{\beta_l})\cdots(E^{\beta_2}+  B_{\beta_2})(E^{\beta_1}+  B_{\beta_1})E^{-\rho}\\
	=&\sum_{J} R^{[\lambda]}_J.
\end{align*}
Here $J=\{j_1<j_2<\cdots <j_t\}$ is a subset of $\{1,2,\dots,l\}$, and $R^{[\lambda]}_J$ is the term that contains $B_{\beta_j}$ if $j\in J$, and $E^{\beta_j}$, otherwise. Moving all the $B$-operators to the left, we get 
\begin{align*}
	R^{[\lambda]}_J=B_{\beta_{j_t}}\cdots B_{\beta_{j_1}}E^{-h\hat{r}_{j_1}\cdots \hat{r}_{j_t}(-\lambda)}=B_{\beta_{j_t}}\cdots B_{\beta_{j_1}}E^{-h\hat{r}_{J<}(-\lambda)}.
\end{align*}
Here we have used the relation $B_{\beta}E^{s_\beta\mu}=E^{\mu}B_{\beta}$ and \Cref{lem:lp}. Then the Theorem follows from this and \Cref{thm:lambda}.
\end{proof}	

\subsection{Parabolic case}\label{sec:parabolic}
In this section, we extend the above Chevalley formula to the partial flag variety case $G/P$. 

\begin{theorem} [Chevalley formula for the $G/P$ case]\label{thm:parab}
Let $w\in W^P$ and $\lambda\in X^*(T)_P$. Then we have
\begin{equation*}
	\calL_\lambda\otimes MC_y(X(w W_P)^\circ)=
	\sum_{u\in W^P}
	\left(\sum_{v\in u W_P} (-y)^{\ell(v)-\ell(u)}C^{w}_{v, \lambda}
	\right)
	MC_y(X(u W_P)^\circ),
\end{equation*}
and
\begin{equation*}
\calL_\lambda\otimes SMC_y(Y(u W_P)^\circ)=
\sum_{w\in W^P}
\left(\sum_{v\in u W_P} (-y)^{\ell(v)-\ell(u)}C^{w}_{v, \lambda}
\right)
SMC_y(Y(w W_P)^\circ),
\end{equation*}
where 
$C^{v}_{u,\lambda}$ are the Chevalley coefficients for full flag $G/B$ 
given explicitly
\Cref{thm:lambda}.
\end{theorem}
\begin{proof}
	The first equality follows from equations \eqref{E:MCpush} and \eqref{Chev:MC}, 
	while the second equality follows from the first one and the duality in \Cref{prop:segrepb}(2).
\end{proof}

\begin{remark}
In a paper in preparation we use this theorem to give a combinatorial Chevalley
formula for minuscule flag varieties and a
$K$-theoretic generalization of Nakada's
colored hook formula. See also \cite{fan2024pieri} for a Pieri formula for the motivic Chern classes of Schubert cells in the equivariant K-theory of Grassmannians.
\end{remark}

\section{Dualities of Chevalley coefficients}\label{sec:chev-dualities} Recall the expansion from \Cref{Chev:MC}:
\[
\calL_\lambda\otimes MC_y(X(w)^\circ)=\sum_{u\leq w} C_{u,\lambda}^w MC_y(X(u)^\circ).
\]
In this section we state and prove several duality properties 
of the Chevalley coefficients $C_{u,\lambda}^w$. All these dualities have 
a geometric origin (Serre duality, star duality, Dynkin automorphisms) and 
we name the Chevalley dualities correspondingly.

Recall the Grothendieck-Serre operator $\mathcal{D}$ from \eqref{E:SerreD}, and the 
duality functor $(-)^\vee$ on $K_T(G/B)[y^{\pm 1}]$ and $K_T(\pt)[y^{\pm 1}]$, 
which sends a vector bundle to its dual and $y^i$ to $y^{-i}$. 
This induces the following property of the Chevalley coefficients in \Cref{Chev:MC}.
\begin{prop}[Serre duality]\label{prop:serre}
	$$C^{w}_{u,\lambda}=(-y)^{\ell(w)-\ell(u)}w_0(C^{w_0u}_{w_0w,-\lambda})^\vee. $$
\end{prop}
\begin{proof}
	Since 
	\[SMC_y(Y(w)^\circ)=(-y)^{\dim Y(w)}\frac{\calD(MC_y(Y(w)^\circ))}{\lambda_y(T^*(G/B))},\]
	\Cref{Chev:SMC} can be written as
	$$\mathcal L_{\lambda}\otimes\calD(MC_y(Y(u)^\circ))=\sum_{w\in W} C^{w}_{u,\lambda} (-y)^{\ell(u)-\ell(w)}\calD(MC_y(Y(w)^\circ)).$$
	Taking the dual $\calD$ on both sides, we have
	$$\mathcal L_{-\lambda}\otimes MC_y(Y(u)^\circ)=\sum_{w\in W} (C^{w}_{u,\lambda})^\vee (-y)^{\ell(w)-\ell(u)}MC_y(Y(w)^\circ).$$
	Finally, applying the left $w_0$ action to the above identity, we get
	$$\mathcal L_{-\lambda}\otimes MC_y(X(w_0u)^\circ)=\sum_{w\in W} w_0(C^{w}_{u,\lambda})^\vee (-y)^{\ell(w)-\ell(u)}MC_y(X(w_0w)^\circ).$$
	This finishes the proof of the theorem by comparing the above equation with \Cref{Chev:MC}.
\end{proof}

The star duality $*$ acts on $K_T(G/B)[y^{\pm 1}]$ and $K_T(pt)[y^{\pm 1}]$ by sending a vector bundle to its dual,
and leaving $y^i$ unchanged. Consider the composition
\[ \iota:=w_0* : K_T(\pt)[y^{\pm 1}] \to K_T(\pt)[y^{\pm 1}] \/.\] 
Then we have the following result by combining Theorem 9.1(a) and Remark 4.7 of \cite{AMSS:specializations}. 
\begin{equation}\label{equ:star}
	\bbC_{-\rho}\otimes\calL_{-\rho}\otimes MC_y(X(w)^\circ)=(-1)^{\dim G/B-\ell(w)}\prod_{\alpha>0}(1+ye^{-\alpha})*\bigg(SMC_y(X(w)^\circ)\bigg).
\end{equation}
\begin{prop}[Star duality]\label{prop:star}
		\[C_{u,\lambda}^w=(-1)^{\ell(w)-\ell(u)}\iota(C_{w_0w,-\lambda}^{w_0u}).\]
\end{prop}
\begin{proof}
	Applying the star duality functor $*$ to \Cref{Chev:MC}, then use \Cref{equ:star} to get
	\[\calL_{-\lambda}\otimes SMC_y(X(w)^\circ)=\sum_{u}(-1)^{\ell(w)-\ell(u)}*(C_{u,\lambda}^w)SMC_y(X(u)^\circ).\]
	Applying the left $w_0$ action to the above equation and comparing with \Cref{Chev:SMC}, we get the result.
\end{proof}

The map sending $\alpha_i \mapsto -w_0(\alpha_i)$ for every simple root $\alpha_i$
induces an automorphism on the Dynkin diagram, hence also on the flag variety $G/B$.
This automorphism maps $X(w)^\circ$ to $X(w_0ww_0)^\circ$, and it induces a ring 
automorphism $\phi$ on $K_T(G/B)$, which sends $\calL_\lambda$ to $\calL_{-w_0\lambda}$,
and twists the base ring $K_T(\pt)$ by the map $\iota$ above. 
\begin{prop}[Dynkin duality]\label{prop:w0}
	$$C_{u,\lambda}^w=\iota(C_{ w_0uw_0,-w_0\lambda}^{w_0ww_0}).$$
\end{prop}
\begin{proof}
	Applying the Dynkin automorphism $\phi$ to \Cref{Chev:MC}, we obtain
	\begin{equation*}
		\mathcal L_{-w_0\lambda}\otimes MC_y(X(w_0ww_0)^\circ)=\sum_{u\in W} \iota(C^{w}_{u,\lambda}) MC_y(X(w_0uw_0)^\circ).
	\end{equation*}
	Then the claim follows from the definition of the coefficients $C_{u,\lambda}^w$.
\end{proof}
Combining the above dualities we obtain the following. 
\begin{prop}\label{prop:w_0 lam}
	\begin{enumerate}
		\item (Serre duality + star duality)
		$$(C^{w}_{u,\lambda})|_{y\to y^{-1}} \times y^{\ell(w)-\ell(u)}=C^{w}_{u,\lambda}\/.$$
		\item (Star duality + Dynkin duality)
		$$C^{w}_{u,\lambda}=(-1)^{\ell(w)-\ell(u)} C^{u w_0}_{w w_0,w_0 \lambda}\/.$$
	\end{enumerate}
\end{prop}
\begin{proof}
	The first part follows directly from \Cref{prop:serre} and \Cref{prop:star}, and the second 
	part follows from \Cref{prop:star} and \Cref{prop:w0}.	
This finishes the proof.
\end{proof}

Using the $\lambda$-chain formula in \Cref{thm:lambda}, we can also give a 
direct combinatorial proof for the various duality identities in this section.
The proofs are similar to \cite[Theorem 8.6]{lenart.postnikov:affine} and
	\cite[Theorem 8.7]{lenart.postnikov:affine}.

\begin{prop}\label{prop:dual}
	For any integral weight $\lambda$ and $w,u\in W$,
	\begin{itemize} 
		\item[(a)]
		$C^w_{u,\lambda}\in K_T(pt)[y]$ and $(C^{w}_{u,\lambda})|_{y\to y^{-1}} \times y^{\ell(w)-\ell(u)}=C^{w}_{u,\lambda}$ i.e., $C^{w}_{u,\lambda}$ is palindromic regarded as a polynomial in $y$;
		(cf.~\Cref{prop:w_0 lam} (1)); 
		
		\item[(b)] 
		$C^{w}_{u,\lambda}=(-1)^{\ell(w)-\ell(u)} C^{u w_0}_{w w_0,w_0 \lambda};$\;
		(cf.~\Cref{prop:w_0 lam} (2));

		\item[(c)] 
		$C^{w}_{u,\lambda}=(-1)^{\ell(w)-\ell(u)} \iota\left(C^{{w_o} u }_{{w_o} w, -\lambda} \right);$\;
		(cf.~\Cref{prop:star}).
		
		\item[(d)] 
		$C^{w}_{u,\lambda}= \iota\left(C^{w_0 w w_0}_{ w_0 u w_0, -w_0 \lambda} \right)$;~ 
		(cf.~\Cref{prop:w0} ).
		
	\end{itemize}
	
\end{prop}
\begin{proof}
	Part (a) is straightforward from \Cref{equ:lambda}; part
	(d) follows from (b) and (c). Parts
	(b) and (c) follow from known properties of $\lambda$-chain, i.e.,
	if $(\beta_1,\ldots, \beta_l)$ is a $\lambda$-chain, then 
	$(w_0\beta_l,\ldots, w_0\beta_1)$ is a $(w_0 \lambda)$-chain and
	$(-\beta_l,\ldots, -\beta_1)$ is a $(-\lambda)$-chain.
	\end{proof}

\begin{remark} The Serre duality extends to any $G/P$.
	For $w,u\in W^P$and $\lambda\in X^*(T)_P$ recall the expansion: 
	\begin{equation*}
		\calL_\lambda\otimes MC_y(X(w W_P)^\circ)=
		\sum_{u\in W^P}
		C^{w,P}_{u,\lambda}
		MC_y(X(u W_P)^\circ).
	\end{equation*}
	Then  \Cref{thm:parab} can be rewritten as 
	$C^{w,P}_{u,\lambda}=\displaystyle\sum_{v\in u W_P} (-y)^{\ell(v)-\ell(u)}C^{w}_{v, \lambda}$.
	By the same argument as in the proof of \Cref{prop:serre},
	we have the Serre duality on parabolic Chevalley coefficients;
	\begin{equation}\label{parabolic:serre}
		C^{w,P}_{u,\lambda}=(-y)^{\ell(w)-\ell(u)}w_0(C^{\overline{w_0u},P}_{\overline{w_0w},-\lambda})^\vee,
	\end{equation}
	where $\overline{w_0w}$ and $\overline{w_0u}\in W^P$ are minimal coset representatives of
	$w_0w W_P$ and $w_0u W_P$.
\end{remark}

\section{K-theoretic stable envelopes for $T^*(G/B)$}}\label{sec:K-stab}
In this section we apply the Chevalley formula for motivic classes to calculate the transformation 
of stable envelopes in $T^*(G/B)$ under the change of arbitrary alcoves. For alcoves adjacent to the fundamental
alcove, a formula for this transformation was obtained in \cite[Theorem 5.4]{su2020k}, see also \cite{koncki2023hecke}.
\subsection{Definition of the stable envelopes} 
The stable envelopes were defined by Maulik and Okounkov in their seminal work on quantum cohomology of Nakajima quiver varieties \cite{maulik2019quantum}. Later, this was generalized by Okounkov and his collaborators to $K$-theory and elliptic cohomology \cite{okounkov:Klectures, aganagic2021elliptic}. We recall next the definition of the stable envelopes for $T^*(G/B)$.

The torus $T$ acts by left multiplication on $G/B$. Hence, it induces a natural action on $T^*(G/B)$. There is also a natural dilation $\bbC^*$ action on the cotangent fibers by a character of 
$q^{-1}$. 
Throughout this section, we use $q^{1/2}$ to denote the standard representation of $\bbC^*$, so that  
$K_{T \times \bbC^*}(\pt) = K_T(\pt)[q^{\pm 1/2}]$. 

The definition of the stable envelopes depends on three parameters:
\begin{itemize}
	\item 
	a {\em chamber} $\calC$ in the Lie algebra of the maximal torus $T$.
	\item 
	a {\em polarization} $T^{\frac{1}{2}}\in K_{T\times\bbC^*}(T^*(G/B))$ of the
	tangent bundle $T(T^*(G/B))$, i.e., 
	a solution of the equation 
	\[
	T^{\frac{1}{2}}+q^{-1}(T^{\frac{1}{2}})^\vee=T(T^*(G/B)) 
	\] 
	in the ring $K_{T\times\bbC^*}(T^*(G/B))_{\loc}$.
	\item 
	an {\em alcove} $A$ in $\ft^*_\bbR$, which is also called a {\em slope} in \cite{okounkov:Klectures}.
\end{itemize}
Given a polarization $T^{\frac{1}{2}}$, there is an opposite polarization defined by 
$T^{\frac{1}{2}}_{\opp}:=q^{-1}(T^{\frac{1}{2}})^\vee$. 
A typical example of a polarization is $T(G/B)$ (pulled back from $G/B$ to $T^*(G/B)$). 
Its opposite is equal to $T^*(G/B)$.
Because we will work with fibers of polarizations over fixed points, in what follows we 
will assume that the polarization is given by a subbundle of $T(T^*(G/B))$, or 
possibly a virtual vector subbundle, i.e.,~a formal difference of such subbundles.

The torus fixed point set $(T^*(G/B))^{T}= (G/B)^T$ is in one-to-one
correspondence with the Weyl group $W$. For every $w\in W$, 
recall that $e_w$ 
denotes the corresponding fixed point. For a chosen Weyl chamber
$\fC$ in $\Lie T$, pick any cocharacter $\sigma\in
\fC$. 
The attracting set of the fixed point $e_w$,  
also called the 
Bia\l{}ynicki-Birula cell in the literature, is defined by
\[
\Attr_\fC(w)=\left\{x\in T^*(G/B) \mid
\lim\limits_{z\rightarrow 0}\sigma(z)\cdot x=e_w \right\}.
\] 
By analyzing the (signs of the roots in the) weight
space decomposition of $T_w(T^*(G/B))$, one may show that 
$\Attr_\fC(w)$ is the conormal bundle of the attracting variety of $w$ in $G/B$,
i.e., the conormal bundle of the Schubert cell stable 
under the Borel subgroup associated to the chamber $\fC$.\begin{footnote}
{We note that $T^*(G/B)$ is not compact, so not all points have well defined limits at $0$.
For example, if $\fC$ is the positive chamber, then the points in the open 
set $T^*(X(w_0)^\circ)\setminus X(w_0)^\circ \subset T^*(G/B)$ do not have limits at $0$.}\end{footnote}
Define a partial order on the fixed point set $W$ to be the (transitive 
closure of the) following relation:
\[
e_w\preceq_\fC e_v \text{\quad if \quad}
\overline{\Attr_\fC(v)}\cap e_w\neq \emptyset.
\]
Then the order determined by the positive (resp., negative) chamber is
the same as the Bruhat order (resp., the opposite Bruhat order).

Any chamber $\fC$ determines a decomposition of the tangent
space $N_w:= T_w(T^*(G/B))$ as $N_w=N_{w,+}\oplus N_{w,-}$ into
$T$-weight spaces which are positive and negative with respect to
$\fC$ respectively. For every polarization $T^{\frac{1}{2}}$,
denote 
$N_w\cap T^{1/2}|_w$ by $N_w^{\frac{1}{2}}$.
Similarly, we have
$N_{w,+}^{\frac{1}{2}}$ and $N_{w,-}^{\frac{1}{2}}$. In particular,
$N_{w,-}=N^{\frac{1}{2}}_{w,-}\oplus
q^{-1}(N_{w,+}^{\frac{1}{2}})^\vee$. Consequently, we have
\[
N_{w,-}- N_w^{\frac{1}{2}}=q^{-1}(N_{w,+}^{\frac{1}{2}})^\vee-
N_{w,+}^{\frac{1}{2}}
\] 
as virtual vector bundles. The determinant bundle of the virtual
bundle $N_{w,-}- N_w^{\frac{1}{2}}$ is a complete square and its
square root will be denoted by $\left(\frac{\det N_{w,-}}{\det
	N_w^{\frac{1}{2}}}\right)^{\frac{1}{2}}$;
cf.~\cite[\S9.1.5]{okounkov:Klectures}. For instance, if we choose
the polarization $T^{1/2} = T(G/B)$, the positive chamber, and $w=\id$
then both $N_{\id}^{\frac{1}{2}}$ and $N_{\id, -}$ have weights $-
\alpha$, where $\alpha$ varies in the set of positive roots; in this
case the virtual bundle $N_{\id,-}- N_{\id}^{\frac{1}{2}}=0$. 

Let $f:=\sum_\mu f_\mu e^\mu\in K_{T\times \bbC^*}
(\pt)$ be a Laurent polynomial, where $e^\mu\in K_T(\pt)$ 
and $f_\mu\in \bbQ[q^{1/2},q^{-1/2}]$. The {\em Newton polytope\/}
of $f$, denoted by $\deg_Tf$, is
\[
\deg_T f=\mbox{Convex hull } (\{\mu| f_\mu\neq 0\})\subseteq
X^*(T)\otimes_\bbZ \bbQ,
\] 
where $X^*(T)$ denotes the character lattice of $T$. ~The 
following theorem defines the $K$-theoretic stable envelopes.
\begin{theorem}\label{thm:geostable}~\cite[\S 9.1]{okounkov:Klectures}, \cite[Thm.~1]{okounkov2022quantum}.
	For every chamber $\fC$,  a polarization~$T^{1/2}$, and an alcove $A$, there exists a unique map of
	$K_{T\times\bbC^*}(\pt)$-modules
	\[
	\stab_{\fC,T^{\frac{1}{2}}, A}:K_{T\times
		\bbC^*}((T^*(G/B))^T)\rightarrow K_{T\times\bbC^*}(T^*(G/B))
	\]
	such that for every $w\in W$, the class
	$\Gamma:=\stab_{\fC,T^{\frac{1}{2}}, A}(w)$ satisfies:
	\begin{enumerate}
		\item (\textit{Support}) 
		$\Supp \Gamma\subseteq \cup_{z\preceq_\fC
			w}\overline{\Attr_\fC(z)}$;
		\item (\textit{Normalization}) 
		$\Gamma|_w=(-1)^{\rk N_{w,+}^{\frac{1}{2}}}\left(\frac{\det
			N_{w,-}}{\det N_w^{\frac{1}{2}}}\right)^{\frac{1}{2}}
		\calO_{\Attr_\fC(w)}|_w$;
		\item (\textit{Degree}) For every $e_v\prec_\fC e_w$,
		\[ 
		\deg_T\Gamma|_v\subseteq \deg_T\stab_{\fC,T^{\frac{1}{2}},
			A}(v)|_v+ v\lambda-w\lambda \/,
		\]
		where $\lambda\in (X^*(T)\otimes_{\bbZ}\bbQ)\cap A$ is any rational weight in the alcove $A$.
	\end{enumerate}
\end{theorem}
Strictly speaking, $\stab_{\fC,T^{\frac{1}{2}}, A}(w) \in K_{T \times \C^*}(G/B)$ denotes the image of $1 \in K_{T \times \C^*}(e_w)$
under the map $\stab_{\fC,T^{\frac{1}{2}}, A}$. 
From the definition, it is immediate to see that $\{\stab_{\fC,T^{\frac{1}{2}}, A}(w)\mid w\in W\}$ forms a basis for the localized equivariant K-theory $K_{T\times\bbC^*}(T^*(G/B))_{\loc}$,
called the {\em stable basis}. Explicit combinatorial formulae 
and recursions for the localizations
of stable envelopes in $T^*(G/B)$ may be found in \cite[\S 8.3]{AMSS:motivic}.

\subsection{Changing the polarizations}
A natural question is to study the change of the stable envelopes
when we vary the above three parameters. To start, the change of chambers
is encoded in the left Weyl group action. More precisely, 
the group $G$ acts on $T^*(G/B)$ by left multiplication, which induces a left Weyl group action on $K_{T\times\bbC^*}(T^*(G/B))$, see \cite{mihalcea2020left}. If we change the chamber $\fC$ to another chamber $w(\fC)$ ($w \in W$), we have the following formula (see \cite[Lemma 8.2(a)]{AMSS:motivic})
\[w\cdot(\stab_{\calC, T^{\frac{1}{2}},
	A}(u))=\stab_{w(\calC), w(T^{\frac{1}{2}}),A}(wu).\]

Next we consider the change of polarizations. In this case the results 
are stated e.g. in \cite{okounkov:Klectures} in the more general setting of symplectic resolutions; 
for the convenience of the reader we include proofs for $T^*(G/B)$. 
We have the following lemma relating 
two arbitrary polarizations see \cite[Section 7.5.8]{okounkov:Klectures}. 
\begin{lemma}\label{lem:balance}
	For any two polarizations $T_1^{\frac{1}{2}}$ and $T_2^{\frac{1}{2}}$, there exists a class 
$\calF\in K_{T\times\bbC^*}(T^*(G/B))$ 	
such that $T_1^{\frac{1}{2}}-T_2^{\frac{1}{2}}=\calF-q^{-1}\calF^\vee$.
\end{lemma}
\begin{remark}
	Classes of the form $\calF-q^{-1}\calF^\vee$ are called balanced classes in \textit{loc.~cit.}
\end{remark}
\begin{proof}
It suffices to prove the following fact: any solution of the 
equation 
$\calG+q^{-1}\calG^\vee=0\in K_{T\times\bbC^*}(T^*(G/B))$ 
must be of the form 
$\calG=\calF-q^{-1}\calF^\vee$ for some $\calF\in K_{T\times\bbC^*}(T^*(G/B))$.
Since $\calG\in K_{T\times\bbC^*}(T^*(G/B))\simeq K_T(G/B)[q,q^{-1}]$, we can write $\calG=\calF+q^{-1}\calF'$ for some $\calF\in K_T(G/B)[q]$ and $\calF'\in q^{-1}K_T(G/B)[q^{-1}]$. Thus, 
\[\calF+q^{-1}\calF'+q^{-1}\calF^\vee+(\calF')^\vee=0.\]
Since $\calF,(\calF')^\vee\in K_T(G/B)[q]$ and $q^{-1}\calF', q^{-1}\calF^\vee\in q^{-1}K_T(G/B)[q^{-1}]$, we get
\[\calF+(\calF')^\vee=0,\textit{\quad and \quad }q^{-1}\calF'+ q^{-1}\calF^\vee=0.\]
Therefore, $\calF'=-\calF^\vee$, and $\calG=\calF-q^{-1}\calF^\vee$.
\end{proof}

For any two polarizations $T_1^{\frac{1}{2}}$ and $T_2^{\frac{1}{2}}$
let $\mathcal{F}$ be defined by $T_1^{\frac{1}{2}}-T_2^{\frac{1}{2}}=\calF-q^{-1}\calF^\vee$
as in the above lemma.
Define a $T\times\bbC^*$-equivariant line bundle $\calL$ on $T^*(G/B)$ by the following formula
\[\calL:=\det\calF.\]
Here for a virtual bundle $\calF=\calV_1-\calV_2$, $\det \calF:=\det \calV_1/\det\calV_2$.
\begin{example} Consider the polarization
$T_1^{\frac{1}{2}}:=T(G/B)$ and the opposite polarization $T_2^{\frac{1}{2}}=q^{-1}(T_1^{\frac{1}{2}})^\vee=T^*(G/B)$. Then the element $\calF$ in \Cref{lem:balance} can be taken to be $T(G/B)$, therefore $\calL=\det \calF=\calL_{-2\rho}$.
\end{example}

The next proposition shows that the change of polarizations results in a multiplication
by a line bundle, therefore it is encoded in the Chevalley formula for the stable envelopes.
\begin{prop}\cite[Exercise 9.1.12]{okounkov:Klectures}\label{prop:polchange}
		The following holds:
		\[\stab_{\fC,T_2^{\frac{1}{2}}, A}(w)=(-1)^{\rk N_{w,+,2}^{\frac{1}{2}}-\rk N_{w,+,1}^{\frac{1}{2}}}q^{\frac{\rk \calF|_w}{2}}\calL\otimes\stab_{\fC,T_1^{\frac{1}{2}}, A}(w),\]
		where $N_{w,+,i}^{\frac{1}{2}}:=N_{w,+,i}\cap T^{\frac{1}{2}}_i$ for $i=1,2$, and if $\calF=T_1^{1/2}-T_2^{1/2} = \calV_1-\calV_2$ is the virtual bundle from \Cref{lem:balance}, then $\rk\calF|_w:=\rk\calV_{1}|_w-\rk\calV_2|_w.$
\end{prop}
\begin{proof} From the characterization \Cref{thm:geostable}, 
it suffices to show that the right hand side satisfies the defining 
properties of the stable envelope on the left hand side. 
The support condition is immediate. For the degree condition, we need to check that for every $e_v\prec_\fC e_w$,
\[ 
\deg_T(\calL\otimes\stab_{\fC,T_1^{\frac{1}{2}}, A}(w))|_v\subseteq \deg_T(\calL\otimes\stab_{\fC,T_1^{\frac{1}{2}}, A}(v))|_v+ v\lambda-w\lambda \/,
\]
for some $\lambda\in A$. The terms $\calL|_v$ on both sides cancel, and the above inclusion reduces to the degree condition for $\stab_{\fC,T_1^{\frac{1}{2}}, A}(w)$.
The normalization condition follows from the next calculation.
We utilize our assumption that $\calF$ is 
represented by a subbundle of $T (T^*(G/B))$, and we note 
also that the weights of $N_w= T_w(T^*(G/B))$ are distinct. 
	\begin{align*}
		&(-1)^{\rk N_{w,+,2}^{\frac{1}{2}}-\rk N_{w,+,1}^{\frac{1}{2}}}\frac{\stab_{\fC,T_2^{\frac{1}{2}}, A}(w)|_w}{\stab_{\fC,T_1^{\frac{1}{2}}, A}(w)|_w}\\
		=&	\left(\frac{\det N_{w,-}}{\det N_{w,2}^{\frac{1}{2}}}\right)^{\frac{1}{2}}\left(\frac{\det N_{w,-}}{\det N_{w,1}^{\frac{1}{2}}}\right)^{-\frac{1}{2}}\\
		=&\left(\frac{\det N_{w,1}^{\frac{1}{2}}}{\det N_{w,2}^{\frac{1}{2}}}\right)^{\frac{1}{2}}\\
		=&\left(\det N_w\cap (\calF|_w-q^{-1}\calF^\vee|_w)\right)^{\frac{1}{2}}\\
		=&\left(\frac{\det(N_w\cap \calF|_w)}{\det(N_w\cap q^{-1}(\calF|_w)^\vee)}\right)^{\frac{1}{2}}\\
		=&q^{\frac{\rk \calF|_w}{2}}\calL|_w.			
	\end{align*}
The reason for the last equality is as follows. 
We have that $N_w\cap\calF|_w=\calF|_w$, and suppose $e^\lambda$ is a 
torus weight of it. Since $N_w$ is a symplectic vector space, $q^{-1}e^{-\lambda}$ is a weight of $N_w$. Hence 
$q^{-1}e^{-\lambda}$ is also a weight of the intersection $N_w\cap q^{-1}(\calF|_w)^\vee$ and in fact
$N_w\cap q^{-1}(\calF|_w)^\vee= q^{-1}(\calF|_w)^\vee$, giving the equality.
The case of  $\calF=\calV_1-\calV_2$ being a virtual bundle follows from linearity of the constructions.  
\end{proof}
\subsection{Changing the alcoves} We now turn to what happens under the change of alcoves.
This can be answered utilizing the recursive 
formulae from \cite{su2020k,su2021wall}, see \Cref{thm:szz} below. Our Chevalley formula provides a non-recursive answer, in terms of $\lambda$-chains, relating the stable envelope for 
the fundamental alcove $\fA$ to the stable envelope for a translate $\fA+\lambda$; see \Cref{prop:chestab} below. We recall some of the formulae below, and state our Chevalley based formula in \Cref{prop:chestab}.

The alcoves in $\ft^*_\bbR$ are of the form $x(\fA)+\lambda$ 
for some $x\in W$ and some $\lambda$ in the root lattice.
It was proved in \cite[Lemma 8.2]{AMSS:motivic},\cite[Rmk. 2.3]{su2020k} that
\begin{equation}\label{equ:shift2}
	\stab_{\fC,T^{\frac{1}{2}}, A+\lambda}(w)=e^{-w\lambda}\calL_\lambda\otimes \stab_{\fC,T^{\frac{1}{2}}, A}(w) \/,
\end{equation}
where $\calL_\lambda$ is the pullback of $G\times^B\bbC_\lambda$ from 
$G/B$ to $T^*(G/B)$. Fix the chamber $\calC$ to be the anti-dominant Weyl chamber 
\[ \calC :=\{\lambda\in \ft^*_\bbR\mid \langle\lambda,\alpha^\vee\rangle<0, \textit{ for any positive root } \alpha \} \] and the polarization $T^{\frac{1}{2}}=T^*(G/B)$. 
To simplify the notations, we denote $\stab_{\calC, T^*(G/B),A}(w)$ by $\stab_A(w)$.

Let $Z=T^*(G/B)\times_\calN T^*(G/B)$ be the Steinberg variety where $\calN\subset \mathfrak{g}$ denotes the nilpotent cone. 
There is an algebra isomorphism due to Kazhdan--Lusztig and Ginzburg \cite{kazhdan1987proof,chriss2009representation}
\begin{equation*}
\bbH \simeq K_{G\times \bbC^*}(Z),
\end{equation*}
where $K_{G\times \bbC^*}(Z)$ has an algebra structure given by convolution. 
The convolution induces an action of
the Hecke algebra $\bbH$ on $K_{T\times \bbC^*}(T^*(G/B))$ which we recall next. 
For a simple root $\alpha_i$, and the corresponding minimal parabolic subgroup 
$P_i \supset B$, define the operator $T_i$
on $K_{T\times \bbC^*}(T^*(G/B))$ by the following formula:
\[T_{i}(\calF):=-\calF-\pi_{1*}(\pi_2^*\calF\otimes \pi_2^*\calL_{\alpha_i}).\]
Here $\calF\in K_{A\times \bbC^*}(T^*(G/B))$, $Y_i:=G/B\times_{G/P_i}G/B\subset G/B\times G/B$,
$T^*_{Y_i}$ is the conormal bundle of $Y_i$ inside 
$G/B\times G/B$, and 
$\pi_j:T^*_{Y_i}\to  T^*(G/B)$ $(j=1,2)$ are the two projections. 
These operators satisfy the quadratic relations and the braid 
relations in $\bbH$. In particular, $T_w$ is well defined for any $w\in W$.
Recall that $\fA$ denotes the fundamental alcove.

The following have been proved in \cite{su2020k,su2021wall}.
\begin{theorem}\label{thm:szz}
(a) \cite[Theorem 4.5]{su2020k}
		\begin{equation*}
			T_i(\stab_{\fA}(w))=\begin{cases}
				(q-1)\stab_{\fA}(w)+q^{1/2}\stab_{\fA}(ws_i),&\textit{\quad if \quad} ws_i<w,\\
				q^{1/2}\stab_{\fA}(ws_i), & \textit{\quad if \quad } ws_i>w.
			\end{cases}
		\end{equation*}	
		
(b) \cite[Theorem 5.4]{su2021wall} Let $x \in W$. Then 
\[ \stab_{x(\fA)}(w)=q^{-\ell(x)/2}T_x(\stab_{\fA}(wx)) \/. \]

(c) \cite[Lemma 3.5 and Corollary 5.3]{su2021wall} Assume $A_1$ and $A_2$ are two adjacent alcoves separated by a wall of the form $H_{\alpha,n}$, where $\alpha>0$. Assume $A_2$ is on the positive side of $H_{\alpha,n}$, i.e., for any $\mu\in A_2$, $(\mu,\alpha^\vee)>n$. Then \begin{equation*}
		\stab_{A_1}(w)=\begin{cases}
			\stab_{A_2}(w)+e^{-nw\alpha}(q^{1/2}-q^{-1/2})\stab_{A_2}(ws_\alpha) & \textit{\quad if \quad } ws_\alpha>w,\\
			\stab_{A_2}(w) & \textit{\quad if \quad} ws_\alpha<w.
		\end{cases}
	\end{equation*}
\end{theorem}
Part (a) of the theorem implies that $T_x(\stab_{\fA}(wx))$ may be 
recursively calculated in terms of 
$\{\stab_{\fA}(w)\mid w\in W\}$. Part (b) implies that the same is true for $\stab_{x(\fA)}(w)$.
Finally, part (c) may be used to relate directly the stable bases for any two adjacent alcoves, and 
therefore recursively relate the stable bases for two arbitrary alcoves.

We focus next on relating \eqref{equ:shift2} to our 
Chevalley formulae obtained earlier in this paper. Together with (a) and (b) from \Cref{thm:szz}
above, this gives an alternative recursion to (c), 
calculating the stable envelope for an arbitrary
alcove $x \fA + \lambda$ starting from the stable envelope for $\fA$.

Fix $\lambda$ an integral weight. 
By \Cref{thm:szz}(a), $\stab_{x(\fA)}(w)$ 
may be written as a linear combination of $\{\stab_{\fA}(w)\mid w\in W\}$. By \eqref{equ:shift2},
to determine $\stab_{x(\fA)+\lambda}(w)$, it suffices to find a formula for 
$\calL_\lambda\otimes \stab_{\fA}(w)$. 
This can be achieved by the Chevalley formula for the motivic Chern classes.
The key is the following result. 
\begin{lemma}\cite[Theorem 8.6]{AMSS:motivic} \label{lem:stabmc}
Let $\iota:G/B\hookrightarrow T^*(G/B)$ denote the inclusion of the zero section. 
For any $w\in W$,
\[\iota^*(\stab_{\fA}(w))=(-1)^{\dim G/B}q^{\dim G/B-\frac{\ell(w)}{2}}MC_{-q^{-1}}(Y(w)^\circ)\otimes \calL_{-2\rho}.\]
\end{lemma} 
Recall the operator $(-)^\vee$ on $K_T(pt)[y^{\pm1}]$, which sends $e^\mu$  to $e^{-\mu}$ and $y$ to $y^{-1}$. We have the following Chevalley formula for the stable bases. 
\begin{theorem}\label{prop:chestab}
Let $\lambda \in X^*(T)$ be a weight and fix 
$\beta_1, \ldots , \beta_l$ a reduced $\lambda$-chain corresponding
to an alcove walk from  $\fA $ to $\fA - \lambda$. Then
\[\calL_\lambda\otimes \stab_{\fA}(u)=\sum_w q^{\frac{\ell(u)-\ell(w)}{2}} (C_{u,-\lambda}^w)^\vee|_{y=-q^{-1}}\stab_{\fA}(w),\]
where $C_{u,\lambda}^w$ are the coefficients defined in \Cref{Chev:MC}.

One can write this in terms of $\lambda$-chains as follows (with the notation
from \Cref{thm:lambda}): 
	\begin{equation*}\label{stab_lambda}
		\calL_\lambda\otimes \stab_{\fA}(u)=\sum_{J\subset \{ 1,2,\ldots, l \}}	(-1)^{n(J)} (q^{-1/2}-q^{1/2})^{|J|}
		e^{w\tilde{r}_{J_{>}}(\lambda)}\stab_{\fA}(ur_{h_{j_1}}r_{h_{j_2}}\cdots r_{h_{j_t}}),
	\end{equation*}
where the sum is over subsets $J = \{ j_1 < \ldots < j_t \} \subset \{ 1, 2, \ldots, l \}$
such that $u < u r_{h_{j_1}} < u r_{h_{j_1}} r_{h_{j_2}} < \ldots < u r_{h_{j_1}} r_{h_{j_2}} \cdot \ldots \cdot r_{h_{j_t}}$.
\end{theorem}
\begin{remark}
    Recall we have the following duality between the stable bases (see \cite[Proposition 1]{okounkov2022quantum}):
    \[\langle \stab_{\calC,T^\frac{1}{2},A}(u), \stab_{-\calC,T_{\opp}^\frac{1}{2},-A}(w)\rangle=\delta_{u,w}.\]
    Therefore, similar arguments as in the proof of \Cref{lem:mcsmc} will give a Chevalley formula for the dual stale basis $ \stab_{-\calC,T(G/B)),-\fA}(u)$.
\end{remark}
\begin{proof}
By the definition of the SMC classes from definition \ref{def:smc}(2) , \eqref{Chev:SMC} becomes
\begin{equation*}
	\calL_\lambda\otimes \calD(MC_y(Y(u)^\circ))=\sum_{w\geq u} (-y)^{\ell(u)-\ell(w)}C_{u,\lambda}^w \calD(MC_y(Y(w)^\circ)).
\end{equation*}
Taking $(-)^\vee$ on both sides of the equation, we get
	\begin{equation*}
		\calL_{-\lambda}\otimes MC_y(Y(u)^\circ)=\sum_{w\geq u}(-y)^{\ell(w)-\ell(u)} (C_{u,\lambda}^w)^\vee MC_y(Y(w)^\circ).
	\end{equation*}
Then the first equation of the theorem follows from this and \Cref{lem:stabmc}. The second 
equation is a consequence of the first and of \Cref{equ:-lambda} in \Cref{thm:lambda}.
\end{proof}
\begin{example}
In type $A_2$, set $u=s_2 s_1$ and $\lambda=2\varpi_1+\varpi_2$.
A $\lambda$-chain of roots is given by 
$\beta_1=\alpha_2,\;\beta_2=\alpha_1+\alpha_2,\;
\beta_3=\alpha_1,\;\beta_4=\alpha_1+\alpha_2,\;
\beta_5=\alpha_1,\;\beta_6=\alpha_1+\alpha_2$; see \Cref{ex:A_2-1}.
From \Cref{stab_lambda}, we have
\[\begin{array}{lll}
	\calL_\lambda\otimes \stab_{\fA}(s_2 s_1)&=&
	e^{\varpi_1-3\varpi_2}  \stab_{\fA}(s_2 s_1)
+(q^{-\frac{1}{2}}-q^{\frac{1}{2}})  e^{-\varpi_1-2\varpi_2}  \stab_{\fA}(s_1 s_2 s_1).
\end{array}
\]
Therefore by \eqref{equ:shift2}, using that
$-u(\lambda)=-\varpi_1+3\varpi_2$, we have
\[ \begin{array}{lll}
	\stab_{\fA+\lambda}(s_2 s_1)
	&=&
	\stab_{\fA}(s_2 s_1)
+(q^{-\frac{1}{2}}-q^{\frac{1}{2}}) e^{-\alpha_1} \stab_{\fA}(s_1 s_2 s_1).
\end{array}
\]
\end{example}
\begin{example}
Consider $u=s_2$, $\lambda=2\varpi_1+\varpi_2$,
$w_0 u=s_2 s_1$.

We can use 
Serre duality (\Cref{prop:serre}) and \Cref{exam:chev}
to get

$\begin{array}{lll}
\calL_\lambda\otimes \stab_{\fA}(s_2)&=&
e^{3\varpi_1-\varpi_2}  \stab_{\fA}(s_2)\\
&&+(q^{-\frac{1}{2}}-q^{\frac{1}{2}}) (e^{\varpi_1}+e^{-\varpi_1+\varpi_2}+ e^{-3\varpi_1+2\varpi_2}) \stab_{\fA}(s_1 s_2)\\
&&+(q^{-\frac{1}{2}}-q^{\frac{1}{2}})  (e^{2\varpi_1-2\varpi_2}+ e^{\varpi_1-3\varpi_2} ) 
\stab_{\fA}(s_2 s_1)\\
&&+(q^{-\frac{1}{2}}-q^{\frac{1}{2}})^2  (e^{-\varpi_1-2\varpi_2}+e^{-2\varpi_1}+ e^{-\varpi_2})
\stab_{\fA}(s_1 s_2 s_1).\\
\end{array}
$

As $- s_2(\lambda)=-3\varpi_1+\varpi_2$, we have

$\begin{array}{lll}
\stab_{\fA+\lambda}(s_2)&=&
 \stab_{\fA}(s_2)\\
&&+(q^{-\frac{1}{2}}-q^{\frac{1}{2}}) (e^{-\alpha_1}+ e^{-2\alpha_1}+ e^{-3\alpha_1}) \stab_{\fA}(s_1 s_2)\\
&&+(q^{-\frac{1}{2}}-q^{\frac{1}{2}})  (e^{-\alpha_1-\alpha_2}+ e^{-2\alpha_1-2\alpha_2} ) 
\stab_{\fA}(s_2 s_1)\\
&&+(q^{-\frac{1}{2}}-q^{\frac{1}{2}})^2  (e^{-3\alpha_1-2\alpha_2}+e^{-3\alpha_1-\alpha_2}
+ e^{-2\alpha_1-\alpha_2})
\stab_{\fA}(s_1 s_2 s_1).\\
\end{array}
$
\end{example}

\section{Whittaker functions and Hall--Littlewood polynomials}\label{sec:special-functions}
In this section we apply the Chevalley formula to obtain combinatorial expressions for 
Whittaker functions and Hall--Littlewood polynomials. Variants of the formulae 
we obtain were already available in the literature, and our approach based on the cohomological 
calculations adds a geometric perspective to this.
\subsection{Whittaker functions} In this section we study Whittaker functions. These appear in $p$-adic 
representation theory, and in this note we utilize a cohomological construction of these functions
from \cite{AMSS:motivic}, see also \cite{mihalcea2022whittaker}. 
Recall the definition of the Demazure--Lusztig operators on $K_T(pt)[y]$:
\[\widetilde{T_i}(e^\lambda)=-e^\lambda\frac{1+y}{1-e^{-\alpha_i}}+
e^{s_i\lambda}\frac{1+ye^{\alpha_i}}{1-e^{-\alpha_i}},\] 
and
\[\widetilde{T_i^\vee}(e^\lambda)=-e^\lambda\frac{1+y}{1-e^{-\alpha_i}}+e^{s_i\lambda}\frac{1+ye^{-\alpha_i}}{1-e^{-\alpha_i}}.\]
The operators satisfy the usual quadratic and braid relations in the Hecke algebra, 
therefore for any $w \in W$ there are operators $\widetilde{T_w}$ and 
$\widetilde{T^\vee_w}$ acting on $K_T(pt)[y]$, and defined using any reduced
decomposition of $w$. The following has been proved in \cite[Thm. 1.1]{mihalcea2022whittaker}:
\begin{prop}\label{prop:geo} 
Let $\lambda_y(id):= \prod_{\alpha >0} (1+ye^\alpha)$, and
denote by \[ MC'_y(X(w)^\circ):= \lambda_y(id) \frac{MC_y(X(w)^\circ)}{\lambda_y(T^*_{G/B})} \/. \] 
Then for any $\lambda \in X^*(T)$,
\begin{enumerate}
\item
$\chi_T\left(G/B, \calL_\lambda\otimes MC_y(X(w)^\circ)\right)=\widetilde{T^\vee_{w}}(e^\lambda)$;
\item 
$\chi_T\left(G/B, \calL_\lambda\otimes MC_y'(X(w)^\circ)\right)=\widetilde{T_{w}}(e^\lambda).$
\end{enumerate}
\end{prop}
We note that for any anti-dominant weight $\lambda$ and $w\in W$, 
\begin{equation}\label{equ:iwahoriwhittaker}
\chi_T \left(G/B, \calL_\lambda\otimes MC_{y}'(X(w)^\circ)\right)=\widetilde{T_{w}}(e^\lambda)=\calW_{\lambda,w},
\end{equation}
where $\calW_{\lambda,w}$ is the Iwahori--Whittaker function for the Langlands dual 
group over a non-archimedean local field; we refer to \cite{mihalcea2022whittaker} 
for more details, including the number 
theoretic definition of $\calW_{\lambda,w}$. Here we identified $y$ with $-q^{-1}$, where $q$ is the number of elements in the residue field. As explained in \cite{mihalcea2022whittaker}, from the fact that
\[ \sum_{w \in W} MC_{y}'(X(w)^\circ) =1 \]
by the additivity motivic Chern classes, one recovers the 
Casselman--Shalika formula for the spherical 
Whittaker function \cite{casselman1980unramifiedII}:
\begin{equation}\label{equ:cs}
\sum_{w\in W}\calW_{\lambda,w}=\prod_{\alpha >0} (1+ye^\alpha)\chi_T (G/B, \calL_\lambda)=\prod_{\alpha >0} (1+ye^\alpha)\chi_{w_0\lambda} \/.
\end{equation}
Here 
$\chi_{w_0\lambda}$ denotes the character for the irreducible representation of $G$ of highest weight
$w_0(\lambda)$. We also note that, in type A, an interpretation of the Iwahori--Whittaker function in terms of the partition function of the Iwahori lattice model has been obtained in \cite{brubaker2019colored}.

Using the Chevalley coefficients in \Cref{Chev:MC}, we obtain the following formula for the Iwahori--Whittaker function $\calW_{\lambda,w}$. Let $\rho$ denote the half sum of the positive roots. 
\begin{theorem}
For any anti-dominant weight $\lambda$ and $w\in W$,
\[ \calW_{\lambda,w}=e^{\rho}\sum_{u}(-1)^{\ell(u)}C_{\lambda-\rho,u}^w|_{y\mapsto y^{-1}}y^{\ell(w)-\ell(u)}\/.\]
\end{theorem}
\begin{proof} For any $u\in W$,
\begin{equation}\label{equ:mceuler}
\chi_T(G/B, MC_y(X(u)^\circ))=MC_y[X(u)^\circ\rightarrow \pt]=MC_y[\bbA^1\rightarrow \pt]^{\ell(u)}=(-y)^{\ell(u)},
\end{equation}
where the second equality follows from \cite[Theorem 4.2(3)]{AMSS:motivic}.
Therefore, taking the equivariant Euler characteristics of both sides of \Cref{Chev:MC}, we get
\[\widetilde{T^\vee_{w}}(e^\lambda)=\chi_T\left(G/B, \calL_\lambda\otimes MC_y(X(w)^\circ)\right)=\sum_uC_{\lambda,u}^w(-y)^{\ell(u)}.\]
On the other hand, it is immediate to check the following relation between the two Demazure--Lusztig operators:
\begin{equation}\label{equ:twoDL}
\widetilde{T_{w}}=e^\rho \widetilde{T^\vee_{w}}|_{y\mapsto y^{-1}}e^{-\rho}y^{\ell(w)}.
\end{equation}
Hence,
\[\chi_T\left(G/B, \calL_\lambda\otimes MC_y'(X(w)^\circ)\right)=\widetilde{T_{w}}(e^\lambda)=e^{\rho}\sum_{u}(-1)^{\ell(u)}C_{\lambda-\rho,u}^w|_{y\mapsto y^{-1}}y^{\ell(w)-\ell(u)}\/.\]
The proof ends by applying \Cref{prop:geo}.
\end{proof}
\begin{remark}
Notice that $C_{0,u}^w=\delta_{u,w}$. The above proof shows
\[\chi_T\left(G/B, \calL_\rho\otimes MC_y'(X(w)^\circ)\right)=(-1)^{\ell(w)}e^{\rho}.\]
\end{remark}
As a corollary we prove a variant of the Casselman--Shalika formula \eqref{equ:cs}, obtained by Li \cite{li1992some}, see also \cite[Proposition 9.4]{brubaker2019colored}. First define 
\begin{equation}\label{equ:R}
R_\lambda(y):=\chi_T(G/B, \lambda_y (T^{*}_{G/B})\otimes \calL_\lambda)\in K_T(pt)[y].
\end{equation}
\begin{corol} Let $\lambda$ be an anti-dominant integral weight. Then 
\[ \sum_{w}y^{-\ell(w)}\calW_{\lambda,w} = e^\rho R_{\lambda-\rho}(y^{-1}) \/. \]
\end{corol}
\begin{proof}
By the additivity of the motivic Chern classes and Proposition \ref{prop:geo}(1), 
\begin{equation}\label{equ:Rmc}
	R_\lambda(y)=\sum\chi_T\left(G/B, \calL_\lambda\otimes MC_y(X(w)^\circ)\right)=\sum_w\widetilde{T^\vee_{w}}(e^\lambda).
\end{equation}
On the other hand, for an anti-dominant weight $\lambda$, we have
\begin{align*}
	\sum_{w}y^{-\ell(w)}\calW_{\lambda,w}&=\sum_wy^{-\ell(w)}\widetilde{T_{w}}(e^\lambda)\\
	&=e^\rho\sum_w y^{-\ell(w)}e^{-\rho}\widetilde{T_{w}}e^{\rho}(e^{\lambda-\rho})\\
	&=e^\rho\sum_w\widetilde{T^\vee_{w}}|_{y\mapsto y^{-1}}(e^{\lambda-\rho})\\
	&=e^\rho R_{\lambda-\rho}(y^{-1}).
\end{align*}
Here, the first equality follows from \Cref{prop:geo}(2), the third one follows from \Cref{equ:twoDL}, 
and the last one follows from \Cref{equ:Rmc}. 
\end{proof}
\subsection{Hall--Littlewood polynomials}
In this section, we assume either $\lambda$ or $-\lambda$ to be a dominant integral weight. Set 
$\Sigma_\lambda:=\{\alpha\in \Sigma\mid \langle\lambda,\alpha^\vee\rangle=0\}$, and
let $R^+_\lambda$ be the set of positive roots which are linear combinations of the simple roots 
in $\Sigma_\lambda$. Denote by $W_\lambda\subset W$ the subgroup generated 
by the simple reflections $s_\alpha$, where $\alpha\in \Sigma_\lambda$.
Let $W^\lambda$ be the set of minimal length representatives for the cosets $W/W_\lambda$.
Finally, let $P_\lambda$ be the parabolic subgroup containing the Borel 
subgroup $B$ defined by the condition that $W_{P_\lambda}=W_\lambda$. 

\begin{defin}\label{def:H_lambda}
	\begin{enumerate}
		\item Define
			$H_\lambda(y):=
			\chi_T(G/P_\lambda, \lambda_y (T^{*}_{G/P_\lambda})\otimes \calL_\lambda)\in K_T(pt)[y]$.
		\item (Hall--Littlewood polynomial cf. \cite[P.208 (2.2)]{macdonald1998symmetric})
		For a dominant weight $\lambda$, define
		$$HL_\lambda(\bfx;t):=\sum_{w\in W^{\lambda}} w 
		\left( \bfx^\lambda \prod_{\alpha\in R^{+}\setminus R^{+}_{\lambda}} 
		\frac{1-t \bfx^{-\alpha}}{1-\bfx^{-\alpha}}
		\right)$$ where $\bfx^\lambda$ denotes $e^\lambda$.
		\end{enumerate}
		 
\end{defin}
Let $\pi_\lambda:G/B \to G/P_\lambda$ be the natural projection. Then
\[ \lambda_y(G/B) = \pi_\lambda^*(\lambda_y(G/P_\lambda)) \cdot \lambda_y(P_\lambda/B) \/. \]
By the projection formula, and using that $\pi_\lambda^*(\mathcal{L}_\lambda)=\mathcal{L}_\lambda$, we have that
\[ H_\lambda(y)  = \chi_T(G/P_\lambda, \lambda_y (T^{*}_{G/P_\lambda})\otimes \calL_\lambda) 
= \frac{\chi_T(G/B, \lambda_y (T^{*}_{G/B})\otimes \calL_\lambda)}{\chi_T(P_\lambda/B, \lambda_y (T^{*}_{P_\lambda/B}))} = \frac{R_\lambda(y)}{\sum_{w\in W_\lambda}(-y)^{\ell(w)}} \]
Here the last equality follows from \Cref{equ:R}, and the fact that 
\[ \chi_T(P_\lambda/B, \lambda_y (T^{*}_{P_\lambda/B})) = \sum_{ w \in W_\lambda} \chi(MC_y(X(w)^\circ)) =
\sum_{ w \in W_\lambda} (-y)^{\ell(w)} \/. \]
The relation between $H_\lambda$ and the Hall--Littlewood polynomial, summarized next, was obtained in a related (upcoming) collaboration with B. Ion.
\begin{lemma}
We have the following formulae for $H_\lambda(y)$:
\begin{equation}\label{parabolicMC}
H_\lambda(y)=\sum_{w\in W^P} \sum_{u\in W} C^{w}_{u,\lambda} (-y)^{\ell(u)},
\end{equation}
and
\begin{equation}\label{localization}
	H_\lambda(y)
	=\displaystyle
	\sum_{w\in W^{\lambda}} w \left( e^{\lambda}
	\prod_{\alpha\in R^{+}\setminus R^{+}_{\lambda}}\frac{1+y e^{\alpha}}{1-e^{\alpha}}
	\right).
\end{equation}
\end{lemma}

\begin{proof}
\Cref{parabolicMC} follows because
$\lambda_y(T^{*}_{G/P})=MC_y(G/P)=\displaystyle\sum_{w\in W^P}MC_y(X(w W_P)^\circ)$, \Cref{thm:parab},
and $\chi_T(MC_y(X(uW_P)^\circ))=(-y)^{\ell(u)}$ for any $u\in W^P$.
\Cref{localization} follows from the localization formula (cf. \cite{nielsen:diag}, \cite[Theorem 2.1 (c)]{mihalcea2022whittaker}.) To be more specific, the torus fixed points in $G/P_\lambda$ are $\{wP_\lambda\mid w\in W^P\}$, and the torus weights of the tangent space at the fixed point $wP_\lambda$ are $\{-w\alpha\mid \alpha\in R^+\setminus R^+_\lambda\}$.
\end{proof}

\begin{corol} For a dominant integral $\lambda$,
the Hall--Littlewood polynomial $HL_\lambda({x};t)$ can be expressed using 
$H_{-\lambda}(y)$ or
$H_{\lambda}(y)$
as follows:
	\begin{equation}\label{MC-HL1} 
		HL_\lambda(\bfx;t)=H_{-\lambda} (y) |_{e^{\alpha}\mapsto \bfx^{-\alpha},y\mapsto -t},
	\end{equation}
and
	\begin{equation}\label{MC-HL2}
		HL_\lambda(\bfx;t)=\left(\frac{1}{(-y)^{\dim G/P_\lambda}} 
		H_\lambda (y)\right) |_{e^{\alpha}\mapsto \bfx^{\alpha},y\mapsto -t^{-1}}.
	\end{equation}
\end{corol}
We assume next that $\lambda$ is a dominant integral weight. Fix a reduced $(-\lambda)$-chain $\Gamma=(-\beta_1,-\beta_2,\cdots, -\beta_l)$ and
the  sequence of hyperplanes $H_{\beta_1,d_1},H_{\beta_2,d_2},\ldots, H_{\beta_l, d_l}$. This corresponds to an alcove path from $\fA$ to $\fA+\lambda$. Since $\lambda$ is dominant,
$\beta_j>0, d_j>0$.

Then we can recover the following known formula for the Hall--Littlewood polynomial, see \cite[Theorem 2.7]{lenart2011hall}.
\begin{prop}\cite{schwer2006galleries,R06,lenart2011hall}\label{HL_Lenart}
\begin{equation}\label{HL:formula1}
HL_\lambda({\bf x};t)=\sum_{(w,J,u)\in \mathcal A(\Gamma)}
t^{\frac{\ell(w)+\ell(u)-|J|}{2}} (1-t)^{|J|} {\bf x}^{w\hat{r}_{J_<}(\lambda)},
\end{equation}
where (with notation as in \S \ref{ss:lambda-chain}), 
\[ \mathcal A(\Gamma)=\{(w,J,u) \mid w\in W^P, u\in W,J\subset \{1,2,\ldots, l\},
	u\stackrel{J_{>}}{\longrightarrow} w \} \/. \]
\end{prop}	
\begin{proof} 
This is a direct consequence of \Cref{MC-HL1}, \Cref{parabolicMC}, and
\Cref{equ:lambda} from \Cref{thm:lambda}, applied to the $(-\lambda)$-chain $(-\beta_1,-\beta_2,\cdots, -\beta_l)$.
\end{proof}
We also get a new formula for $HL_\lambda(\bfx;t)$ as follows.
\begin{prop}\label{new_HL}
\begin{equation}\label{HL:formula2}
HL_\lambda({\bf x};t)=
\sum_{(u,J,w)\in\; \mathcal A^{op}(\Gamma)}
t^{\frac{2\dim G/P_\lambda-\ell(w)-\ell(u)-|J|}{2}} (1-t)^{|J|} {\bf x}^{u \hat{r}_{J_<}(\lambda))},
\end{equation}
where $\mathcal A^{op}(\Gamma)=\{(u,J,w) \mid
	w\in W^P, u\in W,
	J\subset \{1,2,\ldots, l\}, 	u\stackrel{J_{<}}{\longrightarrow} w\}$.
\end{prop}
\begin{proof} 

This is a direct consequence of \Cref{MC-HL2}, \Cref{parabolicMC}, and
\Cref{equ:-lambda} from \Cref{thm:lambda}, applied to the $(-\lambda)$-chain
together with \Cref{rem:lambda_chev} (1).
\end{proof}
\begin{remark} When $P_\lambda=B$, the equations \eqref{HL:formula1} 
and \eqref{HL:formula2} give the same formula. The correspondence may be seen using 
the Serre duality in \Cref{parabolic:serre}. However, the formulae are in 
general different, as shown by the examples below. 
\end{remark}
\begin{example}
(Type $A_2$) Let $G=GL_3(\mathbb C)$, $T=(\mathbb C^*)^3$, and $x_i=e^{\varepsilon_i}$, for $i=1,2,3$. Let $\lambda=\varpi_1=\varepsilon_1$, then $W_\lambda=\langle s_2\rangle\subset W=\langle s_1,s_2\rangle$, and $W^\lambda=\{id, s_1,s_2 s_1\}$.  
Fix a reduced $(-\lambda)$-chain is  $(-\beta_1=-\alpha_1-\alpha_2, -\beta_2=-\alpha_1)$.
\vspace{0.2cm}

\Cref{HL_Lenart} sums over the following seven terms:

$
\begin{array}{c|c|c|r}
w&J&u&\\
\hline
\hline
id & \emptyset&id&x_1\\
\hline
s_1 & \emptyset& s_1&t x_2\\
& \{2 \}&id&(1-t) x_2\\
\hline
\end{array}
$\hspace{1cm}
$\begin{array}{c|c|c|r}
\hline
s_2s_1&\emptyset& s_2 s_1&t^2 x_3\\
&\{1\}& s_1& t(1-t) x_3\\
&\{2\}& s_2&t(1-t) x_3\\
&\{1,2\}&id& (1-t)^2 x_3\\
\hline\end{array}
$.\\
Hence, with $\bfx = (x_1, x_2, x_3)$, 
\[HL_\lambda(\bfx;t)=(x_1)+(t x_2+(1-t) x_2)+(t^2 x_3+t(1-t) x_3+ t(1-t) x_3+(1-t)^2 x_3)
=x_1+x_2+x_3.\]
On the other hand, \Cref{new_HL} sums over the following six terms:

$
\begin{array}{c|c|c|r}
w&J&u&\\
\hline
\hline
id & \emptyset&id&t^2 x_1\\
\hline
s_1 & \emptyset&s_1& t x_2\\
& \{2 \}&id&t(1-t) x_1\\
\hline
\end{array}
$
\hspace{1cm}
$\begin{array}{c|c|c|r}
\hline
s_2s_1&\emptyset&s_2 s_1&  x_3\\
&\{1\}& s_1&(1-t) x_2\\
&\{2\}&s_2 &(1-t) x_1\\
\hline\end{array}
$.\\
Hence,
\[HL_\lambda(\bfx;t)=
(t^2 x_1)+(t x_2 +t (1-t) x_1)+(x_3+(1-t)x_2+
(1-t)x_1))
=x_1+x_2+x_3.\]

\end{example}

\begin{example}
Same setup as in the above example, but with $\lambda=2 \varpi_2$. Then
$W_\lambda=\langle s_1 \rangle$, $W^\lambda=\{id, s_2, s_1 s_2 \}$, and $d=\dim G/P_\lambda=2$.
A $(-\lambda)$-chain is 
$(-\beta_1=-(\alpha_1+\alpha_2),-\beta_2=-\alpha_2, -\beta_3=-(\alpha_1+\alpha_2),-\beta_4=-\alpha_2)$
, and $d_1=1,d_2=1,d_3=2,d_4=2$.

\Cref{HL_Lenart} sums over the following twelve terms:

$
\begin{array}{c|c|c|r}
w&J&u&\\
\hline
\hline
id & \emptyset& id& x_1^2 x_2^2\\
\hline
s_2 & \emptyset& s_1& t x_1^2 x_3^2\\
& \{2 \}& id&(1-t) x_1^2 x_2 x_3\\
& \{4 \}&id&(1-t) x_1^2 x_3^2\\
\hline
\end{array}
$
\hspace{1cm}
$\begin{array}{c|c|c|r}

\hline
w=s_1s_2&\emptyset& s_1 s_2&t^2 x_2^2 x_3^2\\
&\{1\}& s_2 &t(1-t) x_1 x_2 x_3^2\\
&\{2\}& s_1 & t(1-t) x_1 x_2^2 x_3\\
&\{3\}& s_2& t (1-t) x_2^2 x_3^2\\
&\{4\}& s_1&t(1-t) x_2^2 x_3^2\\
&\{1,2\}&id& (1-t)^2 x_1 x_2^2 x_3\\
&\{1,4\}&id& (1-t)^2 x_1 x_2 x_3^2\\
&\{3,4\}& id&(1-t)^2 x_2^2 x_3^2\\
\hline
\end{array}.
$\\

We get 
$HL_\lambda(\bfx;t)=s_{22}-t s_{211}$, where
\[ s_{22}=x_1^2 x_2^2+x_1^2 x_3^2+x_2^2 x_3^2+
x_1^2 x_2 x_3+x_1 x_2^2 x_3+x_1 x_2 x_3^2 \/,
\]
and
\[ s_{211}=x_1^2 x_2 x_3+x_1 x_2^2 x_3+x_1 x_2 x_3^2 \/. \]
The \Cref{new_HL} sums over the following ten terms: \\
$
\begin{array}{c|c|c|r}
w&J&u&\\
\hline
\hline
id & \emptyset& id & t^2 x_1^2 x_2^2\\
\hline
s_2 & \emptyset& s_2& t x_1^2 x_3^2\\
& \{2 \}& id &t (1-t) x_1^2 x_2 x_3\\
& \{4 \}&id&t (1-t) x_1^2 x_2^2\\
\hline
\end{array}
$\hspace{1cm}
$\begin{array}{c|c|c|r}
\hline
s_1s_2&\emptyset& s_1 s_2& x_2^2 x_3^2\\
&\{1\}&  s_2&(1-t) x_1 x_2 x_3^2\\
&\{2\}& s_1&(1-t) x_1 x_2^2 x_3\\
&\{3\}&  s_2&(1-t) x_1^2 x_3^2\\
&\{4\}& s_1&(1-t) x_1^2 x_2^2\\
&\{2,3 \}& id&(1-t)^2 x_1^2 x_2 x_3\\
\hline
\end{array}.
$\\

The summation also gives 
$HL_\lambda(\bfx;t)=(x_1^2 x_2^2+x_1^2 x_3^2+x_2^2 x_3^2)+(1-t)(x_1^2 x_2 x_3+x_1 x_2^2 x_3+x_1 x_2 x_3^2)
=s_{22}-t s_{211}$.

\end{example}

\appendix{
\section{Chevalley formulae for the Chern--Schwartz--MacPherson classes of Schubert cells}\label{eqCSM}

In this appendix, we give a short proof of the Chevalley formulae for the equivariant Chern--Schwartz--MacPherson (CSM) and Segre--MacPherson (SM) classes of the Schubert cells in the (partial) flag varieties, see \cite[Thm~ 9.10]{AMSS:shadows}
and \Cref{thm:che4} below. Our proof again relies on the action of the appropriate Hecke algebra, this time on the equivariant cohomology of $G/P$. {Besides the intrinsic interest in these Chevalley formulae, we mention that recursions based on it were recently utilized to obtain proofs of Nakada's colored hook formula for finite Weyl groups \cite{mihalcea2022hook} and more general Coxeter groups \cite{mihalcea2024Coxhook}.}

\subsection{Degenerate affine Hecke algebra}
\subsubsection{A change of bases formula}
Recall that the degenerate affine Hecke algebra $\calH$ is generated by $T_w, w\in W$ and $x_\lambda$, $\lambda\in X^*(T)$, such that
\begin{itemize}
	\item
	$T_wT_u=T_{wu}$ for any $w,u\in W$;
	\item 
	$x_\lambda x_\mu=x_\mu x_\lambda$ for any $\lambda,\mu \in X^*(T)$;
	\item 
	$x_{\lambda+\mu}=x_\lambda+ x_\mu$ for any $\lambda,\mu \in X^*(T)$;
	\item
	for any simple root $\alpha_i$, $T_{s_i}x_\lambda-x_{s_i\lambda}T_{s_i}=-\langle\lambda,\alpha_i^\vee\rangle$.
\end{itemize}

We have the following commutation relation.
\begin{lemma}\label{lem:comm}
	For any $w\in W$ and $\lambda\in X^*(T)$, the following holds in $\calH$
	\[T_wx_\lambda =x_{w\lambda}T_w -\sum_{\alpha>0,ws_\alpha <w}\langle\lambda,\alpha^\vee\rangle T_{ws_\alpha }.\]
\end{lemma}
\begin{proof}
	We utilize induction on $\ell(w)$. The claim is clear when $\ell(w)=0$ or $1$. Now assume $\ell(w)>1$, and 
	that the claim holds for any Weyl group element with length smaller than $\ell(w)$. Pick a simple root $\alpha_i$, such that $w>ws_i$. By induction,
	\[T_{ws_i}x_{s_i\lambda} = x_{w\lambda}T_{ws_i}-\sum_{\alpha>0, ws_is_\alpha<ws_i}\langle s_i\lambda,\alpha^\vee\rangle T_{ ws_is_\alpha}.\]
	Using the commutation of $T_{s_i}$ and $x_\lambda$ and the induction hypothesis we obtain:
\[ 
\begin{split}
T_wx_\lambda -x_{w\lambda}T_w = & \quad T_{ws_i}x_{s_i\lambda}T_{s_i}-\langle\lambda,\alpha_i^\vee\rangle T_{ws_i}-x_{w\lambda}T_w \\
		=&-\sum_{\alpha>0,ws_is_\alpha <ws_i}\langle \lambda,s_i\alpha^\vee\rangle T_{ws_is_\alpha}T_{s_i}-\langle\lambda,\alpha_i^\vee\rangle T_{ws_i}\\
		=&-\sum_{\alpha>0,ws_\alpha <w}\langle\lambda,\alpha^\vee\rangle T_{ws_\alpha }.
\end{split}
\]
	Here in the last equality, we have used the facts that if $ws_i <w$ then
	\begin{align*}
		\{\alpha>0\mid ws_\alpha <w\}&=\{\alpha>0\mid w\alpha<0\}\\
		&=\{\alpha_i\}\sqcup\{s_i\alpha\mid \alpha>0,ws_is_\alpha <ws_i\},
	\end{align*} 
and $T_{ws_is_\alpha}T_{s_i}=T_{ws_is_\alpha s_i}=T_{ws_{s_i(\alpha)}}$.
\end{proof}

\subsubsection{The Hecke action on the equivariant cohomology}
Since $G$ acts on $G/P$ by left multiplication, there is a natural left Weyl group action on $H_T^*(G/P)$ for any partial flag variety $G/P$ and which acts on the base ring $H_T^*(\pt)$ by the usual Weyl group action; see e.g. \cite{mihalcea2020left}.
For any $w\in W$, we use $w^L$ to denote this action.

For any simple root $\alpha_i$, define the left Demazure--Lusztig operator on $H_T^*(G/P)$ by the following formula (see \cite[Section 3.2]{mihalcea2020left}):
\[\calT_i^L:=\frac{\alpha_i+1}{\alpha_i}s_i^L-\frac{1}{\alpha_i}.\]
Then it is proved in \textit{loc. cit.} that these operators satisfy the braid relations and $(\calT_i^L)^2=\id$. Moreover, it is immediate to check the following lemma. 
\begin{lemma}\label{lem:heckeaction}
	There is an action $\Psi$ of the degenerate affine Hecke algebra $\calH$ on $H_T^*(G/P)$, sending $T_i$ to $\calT_i^L$ and $x_\lambda$ to $\lambda\in H_T^*(\pt)$. 
\end{lemma}

\subsection{Definition of the CSM/SM classes}\label{CSM:def}
Next we recall the basic definitions and properties of CSM classes for 
the Schubert cells in $G/P$, we will be brief. We refer the reader e.g. to \cite{ohmoto:eqcsm,AMSS:shadows}
for details, including a construction of these classes in the equivariant setting and for general varieties.

The (additive) group of constructible functions 
$\calF(X)$ consists of functions $\varphi = \sum_Z c_Z \one_Z$, 
where the sum is over a finite set of constructible subsets $W \subset X$,
 $c_Z \in \bbZ$ are integers, and $\one_Z$ is the characteristic function of $Z$.
 For a proper morphism $f:Y\rightarrow X$, there is a linear map 
 $f_*:\calF(Y)\rightarrow \calF(X)$, such that for any constructible subset 
 $Z\subset Y$, $f_*(\one_Z)(x)=\chi_{\textit{top}}(f^{-1}(x)\cap Z)$, where 
 $x\in X$ and $\chi_{\textit{top}}$ denotes the topological Euler characteristic.
A conjecture attributed to Deligne and Grothendieck states that there is a unique natural 
transformation $c_*: \calF \to H_*$ from the functor of constructible functions 
on a complex algebraic variety $X$ to the homology functor, where all morphisms are proper, 
such that if $X$ is smooth then $c_*(\one_X)=c(T_X)\cap [X]$. 
This conjecture was proved by MacPherson \cite{macpherson:chern}; 
the class $c_*(\one_X)$ for possibly singular $X$ was shown to coincide with a class defined earlier by M.-H.~Schwartz \cite{schwartz:1, schwartz:2, BS81}. 

There is an equivariant version of MacPherson's transformation defined 
by Ohmoto \cite{ohmoto:eqcsm}. In this case one starts with a variety $X$ 
with a $T$-action, and the equivariant version $\calF^T(X)$ of 
the group of constructible functions $\calF(X)$ contains the characteristic functions $\one_Z$ for $Z$ stable under the $T$-action. If $f:X \to Y$ is a proper $T$-equivariant morphism of algebraic varieties the induced homomorphism and $Z\subset X$ is constructible and $T$-stable then one defines $f_*^T: \calF^T(X) \to \calF^T(Y)$ with the property that $f_*^T(\one_Z) = f_*(\one_Z)$. Ohmoto proves \cite[Theorem 1.1]{ohmoto:eqcsm} that there is an equivariant version of MacPherson transformation $c_*^T: \calF^T(X) \to H_*^T(X)$ that satisfies $c_*^T(\one_X) = c^T(T_X) \cap [X]_T$ if $X$ is a non-singular $T$-variety, which and is functorial with respect to proper push-forwards. 
 The last statement means that for all proper $T$-equivariant morphisms $Y\to X$ the following diagram
commutes:
$$\xymatrix{ 
\calF^T(Y) \ar[r]^{c_*^T} \ar[d]_{f_*^T} & H_*^T(Y) \ar[d]^{f_*^T} \\ 
\calF^T(X) \ar[r]^{c_*^T} & H_*^T(X).}$$ 

\begin{defin}
Let $Z$ be a $T$-invariant constructible subvariety of $X$. 
\begin{enumerate}
\item
We denote by $\csm(Z):=c_*^T(\one_{Z}) {~\in H_*^T(X)}$ the {\em equivariant Chern--Schwartz--MacPherson (CSM) class\/} of $Z$. 
\item 
If $X$ is smooth, we denote by $\ssm(Z \subset X):=\frac{c_*^T(\one_{Z})}{c(T_X)} {~\in H_*^T(X)_{\loc}}$ the {\em equivariant Segre--MacPherson (SM) class\/} of $Z$, where $H_*^T(X)_{\loc}:=H_*^T(X)\otimes_{H_*^T(pt)} \Frac H_*^T(pt)$ denotes the localization of $H_*^T(X)$, and $\Frac H_*^T(pt)$ is the fraction field of $H_*^T(pt)$.
\end{enumerate}
\end{defin}

\subsection{The Chevalley formula in cohomology}\label{CSM:Schubert} We now specialize to $X=G/P$ 
with the usual $T$-action. For simplicity we will denote by 
$\ssm(Z \subset G/P)$ simply by $\ssm(Z)$. 
We will identify the equivariant (Borel-Moore) homology group $H_*^T(G/P)$ with 
the equivariant cohomology $H^*_T(G/P)$, using the Poincar{\'e} duality. 
The sets of CSM classes of Schubert cells $\{\csm(X(wW_P)^\circ)\mid w\in W^P\}$ and of the SM classes $\{\ssm(X(wW_P)^\circ)\mid w\in W^P\}$ form bases for
$H_T^*(G/P)_{\loc}:=H_T^*(G/P)\otimes_{H_T^*(pt)}\Frac H_T^*(pt)$. Moreover, if one takes the opposite 
Schubert cells in any of these sets, then the two bases are dual under the usual intersection pairing,
see \cite[Theorem 9.4]{AMSS:shadows}:
\begin{equation}\label{equ:dualP}
	\langle \csm(X(wW_P)^\circ),\ssm(Y(uW_P)^\circ)\rangle_{G/P}=\delta_{w,u} \textit{ for any } w,u\in W^P.
\end{equation}
The left Demazure--Lusztig operator acts on the CSM classes by the following formula (see \cite[Theorem 4.3]{mihalcea2020left})
\begin{equation*}
	\calT_i^L(\csm(X(wW_P)^\circ))=\csm(X(s_iwW_P)^\circ).
\end{equation*}
Hence, for any $w\in W$,
\begin{equation}\label{equ:leftcsm}
	\csm(X(wW_P)^\circ)=\calT_w^L([X(\id)]).
\end{equation}
Recall for any $\lambda\in X^*(T)_P$, $\calL_\lambda:=G\times^P\bbC_\lambda\in\Pic_T(G/P)$. 
The following is our main result in this Appendix, and it has also been proved in 
\cite[Thm. 9.10]{AMSS:shadows} using the Chevalley formula for the cohomological stable envelopes from \cite[Thm. 3.7]{su:quantum}. Here we give a direct proof based on the action of the degenerate affine Hecke algebra.
\begin{theorem}\label{thm:che4} 
For any $w\in W^P$ and $\lambda\in X^*(T)_P$, the following holds in $H_T^*(G/P)$:
\[c_1^T(\calL_\lambda)\cup \csm (X(wW_P)^\circ)=w(\lambda) \csm (X(wW_P)^\circ)-\sum_{\alpha>0,ws_\alpha <w}\langle\lambda,\alpha^\vee\rangle\csm (X(ws_\alpha W_P)^\circ),\]
and
\[c_1^T(\calL_\lambda)\cup \ssm (Y(wW_P)^\circ)=w(\lambda) \ssm (Y(wW_P)^\circ) -\sum_{\alpha>0,ws_\alpha>w}\langle\lambda,\alpha^\vee\rangle \ssm (Y(ws_\alpha W_P)^\circ) \/.\]
\end{theorem}
\begin{proof}
	Applying the Hecke action $\Psi$ in \Cref{lem:heckeaction} to the equation in \Cref{lem:comm}, and acting on the point class $[X(\id)]$, we get
	\begin{align*}
     c_1^T(\calL_\lambda)\cup \csm (X(wW_P)^\circ)=&c_1^T(\calL_\lambda)\cup \calT^L_w([X(\id)])\\
     =&\calT^L_w(c_1^T(\calL_\lambda)\cup [X(\id)])\\
     =&\calT^L_w(\lambda \cdot [X(\id)])\\
     =&\Psi(T_wx_\lambda)([X(\id)])\\
     =&\Psi(x_{w\lambda}T_w -\sum_{\alpha>0,ws_\alpha <w}\langle\lambda,\alpha^\vee\rangle T_{ws_\alpha })([X(\id)])\\
     =&w(\lambda) \csm (X(wW_P)^\circ)-\sum_{\alpha>0,ws_\alpha <w}\langle\lambda,\alpha^\vee\rangle \csm (X(ws_\alpha W_P)^\circ).
	\end{align*}
The second equality follows from the fact that the left operator $\calT_w^L$ commutes with $c_1^T(\calL_\lambda)$ because the latter is Weyl-group invariant, as $\calL_\lambda$ is 
a $G$-equivariant line bundle; see \cite{mihalcea2020left}.
Finally, the Chevalley formula for the SM classes follows from the one on CSM via the duality in \Cref{equ:dualP}, similar to the proof of \Cref{lem:mcsmc} above.
\end{proof}

\section{An example of the $\lambda$-chain formula}\label{ex:A_2-1}
We consider Lie type $A_2$, with the Weyl group $W=S_3$,
with
$\lambda=2\varpi_1+\varpi_2$, and $w=s_2 s_1$. 
We can find an alcove walk $p_{-\lambda}$ from $\fA$ to $\fA-\lambda$ as indicated by the red path in diagram (b).
This gives the corresponding reduced expression  
$v_{-\lambda}=s_2 s_1 s_2 s_0 s_1 s_2$ and   the corresponding alcove path 
 $$\fA=
 A_0  \stackrel{-\beta_1}{\longrightarrow}  
 A_1  \stackrel{-\beta_2}{\longrightarrow} 
 A_2  \stackrel{-\beta_3}{\longrightarrow} 
 A_3  \stackrel{-\beta_4}{\longrightarrow} 
 A_4  \stackrel{-\beta_5}{\longrightarrow} 
 A_5  \stackrel{-\beta_6}{\longrightarrow} 
 A_6=\fA-\lambda\;
 (A_{i}=r_i A_{i-1}, 1\leq i\leq 6),
 $$
with $\lambda$-chain
$\beta_1=\alpha_2,\;\beta_2=\alpha_1+\alpha_2,\;
\beta_3=\alpha_1,\;\beta_4=\alpha_1+\alpha_2,\;
\beta_5=\alpha_1,\;\beta_6=\alpha_1+\alpha_2$.

The associated sequence of hyperplanes is
$$h_1=H_{\alpha_2,0},
h_2=H_{\alpha_1+\alpha_2,0},
h_3=H_{\alpha_1,0},
h_4=H_{\alpha_1+\alpha_2,-1},
h_5=H_{\alpha_1,-1},
h_6=H_{\alpha_1+\alpha_2,-2}.
$$

\begin{figure}[hbtp]
\begin{minipage}[b]{6cm}
\newcommand*\rows{10}
\begin{center}
\begin{tikzpicture}[scale=0.95]
    \clip(0.6,0.5) rectangle (7.2, {4.3*sqrt(3)});
    \foreach \row in {0, 1, ...,\rows} {
        \draw[gray] ($\row*(0, {0.5*sqrt(3)})$) -- ($(\rows,0)+\row*(0, {0.5*sqrt(3)})$);};
        
    \foreach \row in {0,1,...,4,5}{
        \draw[gray] ($\row*(1, 0)$) -- ($(5, {\rows*0.5*sqrt(3)})+\row*(1.0,{0.0*sqrt(3)})$);};        
   \foreach \row in {5,...,\rows}{
        \draw[gray] ($\row*(1, 0)$) -- ($(\row-5,{\rows*0.5*sqrt(3)})$);};
 
     \foreach \row in {1,2,3,4}{
        \draw[gray] ($\row*(1, 0)$) -- ($(0 ,{\row*sqrt(3)})$);};       

     \foreach \row in {1,2,3,4}{
        \draw[gray] ($(\row, {5*sqrt(3)})$) -- ($(0 ,{(5-\row)*sqrt(3)})$);};  

          \foreach \row in {1,2,3,4}{
        \draw[gray] ($({\rows-\row}, 0)$) -- ($(\rows ,{\row*sqrt(3)})$);};  

           \foreach \row in {1,2,3,4}{
        \draw[gray] ($({\rows-\row},  {5*sqrt(3)})$) -- ($(\rows ,{(5-\row)*sqrt(3)})$);};

\newcommand{\trig}[2]{
\draw [lightgray,fill] ( #1,#2)--({#1+0.5},{#2+sqrt(3)/2-0.05})--({#1-0.5},{#2+sqrt(3)/2-0.05})--cycle;}
\trig{4}{3.5};
\trig{3.5}{6.1};

\foreach \x/\y in {0/1,0/-1, 0.75/0.5,-0.75/-0.5,-0.75/0.5, 0.75/-0.5,-0.75/1.5}
{\node[gray] at ({\x+4},{(\y+4)*sqrt(3)/2}) {\tiny 0};
\node[gray] at ({\x+4},{(\y+6)*sqrt(3)/2}) {\tiny 0};
};

\foreach \x/\y in {0/1,0/-1, 0.75/0.5,-0.75/-0.5,-0.75/0.5, 0.75/-0.5}
{\node[gray] at ({\x+5},{(\y+4)*sqrt(3)/2}) {\tiny 1};
\node[gray] at ({\x+5},{(\y+6)*sqrt(3)/2}) {\tiny 1};
};
\foreach \x/\y in {0/0,0/2,1.5/1,1.5/-1}
{\node[gray] at ({\x+4.5},{(\y+4)*sqrt(3)/2}) {\tiny 2};
\node[gray] at ({\x-0.25+4},{(\y+4)*sqrt(3)/2+sqrt(3)/4}) {\tiny 2};
\node[gray] at ({\x-0.25+4},{(\y+4)*sqrt(3)/2-sqrt(3)/4}) {\tiny 2};
};

\foreach \x/\y in {-1.5/0,-1.5/-2,1.5/0,1.5/2}
{\node[gray] at ({\x+4},{(\y+4)*sqrt(3)/2}) {\tiny 0};};
\foreach \x/\y in {-2.75/0.5,-2.75/-0.5,-2.75/-2.5,-2.75/-1.5,-1.25/-1.5}
{\node[gray] at ({\x+4.5},{(\y+4)*sqrt(3)/2}) {\tiny 0};};

\foreach \x/\y in {-0.9/2,-0.9/0}
{\node[gray] at ({\x+4.5},{(\y+4)*sqrt(3)/2}) {\tiny 1};};

\foreach \x/\y in {-1.75/0.5,-1.75/-0.5,-1.75/1.5,-1.75/-1.5}
{\node[gray] at ({\x+4.5},{(\y+4)*sqrt(3)/2}) {\tiny 1};};

\foreach \x/\y in {-1.5/1,-1.5/3,-1.5/-1}
{\node[gray] at ({\x+4.5},{(\y+4)*sqrt(3)/2}) {\tiny 2};};
\foreach \x/\y in {-2.25/0.5,-2.25/-0.5,-2.25/-2.5,-2.25/-1.5,-0.75/-1.5}
{\node[gray] at ({\x+4.5},{(\y+4)*sqrt(3)/2}) {\tiny 2};};

\draw [brown,thick] ($(2,0)$)--($(6.5, {(4.5)*sqrt(3)})$);
\draw [brown,thick] ($(6,0)$)--($(1.5, {(4.5)*sqrt(3)})$);
\draw [brown,thick] ($(0,{2*sqrt(3)})$)--($(8, {(2*sqrt(3)})$);

\node at (6.3,7.2) {$H_{\alpha_1,0}$};  
\node at (1.9,7.2) {$H_{\alpha_2,0}$}; 
\node at (6.4,3.47) {$H_{\alpha_1+\alpha_2,0}$}; 

\coordinate (a) at  (4,{2*sqrt(3)});
\draw[name path=circleA,thick] (a) circle (0.07);

\draw [->,thick,gray] (4,{2*sqrt(3)})--(3.5,{2.5*sqrt(3)});
\draw [->,thick,gray] (4,{2*sqrt(3)})--(4.5,{2.5*sqrt(3)});
\node at (3.3,4.5) {$\varpi_1$};
\node at (4.3,4.5) {$\varpi_2$};

\node at (4.05,3.9) {\tiny$\fA$};      
\node at (3.5,6.4) {\tiny$\fA+\lambda$};  
\node at (1.5,1.2) {\tiny$w(\fA+\lambda)$};  

\node at (1.5,1.6) {\tiny$(0)$};
\node at (1.5,3.5) {\tiny$(4)$};
\node at (2,4.3) {\tiny$(8)$};
\node at (1.3,2) {\tiny$(5)$};
\node at (2,2.6) {\tiny$(10)$};
\node at (2.7,1.8) {\tiny$(3)$};
\node at (4.2,2.8) {\tiny$(2)$};
\node at (5.55,3.9) {\tiny$(1)$};
\node at (3.3,5.0) {\tiny$(6)$};
\node at (5, 4.2) {\tiny$(7)$};
\node at (3.5,2.8) {\tiny$(9)$};

\draw [->,thick,red]  ($(4 , {(2+1/3)*sqrt(3)})$)--($(4 , {(2+2/3)*sqrt(3)})$)
                    --($(3.5 , {(3-1/6)*sqrt(3)})$)--($(3.5 , {(3+1/6)*sqrt(3)})$)
                   --($(3 , {(3+1/3)*sqrt(3)})$)--($(3 , {(3+2/3)*sqrt(3)})$)
                   --($(3.5 , {(4-1/6)*sqrt(3)})$); 
\draw [->,thick,green]  ($(4 , {(2+1/3)*sqrt(3)})$)--($(3.55 , {(2+1/6)*sqrt(3)})$)
                                         --($(3.55 , {(2-1/6)*sqrt(3)})$);
                        
     \draw [->,thick,blue]  ($(3.5 , {(2-1/6)*sqrt(3)})$)--($(3.0 , {(2-1/3)*sqrt(3)})$)
                    --($(3.0 , {(2-2/3)*sqrt(3)})$) --($(2.5 , {(1+1/6)*sqrt(3)})$)
                    --($(2.5 , {(1-1/6)*sqrt(3)})$)--($(2.0 , {(2/3)*sqrt(3)})$)--($(1.5 , {(1-1/6)*sqrt(3)})$); 
       \draw [->,thick] ($(2.5 , {sqrt(3)})$)--($(2.4 , {sqrt(3)})$)--($(2.4 , {(1+1/6)*sqrt(3)})$)
                               --($(2 , {(2-2/3)*sqrt(3)})$)--($(1.5 , {(1+1/6)*sqrt(3)})$);
        \draw [->,thick]  ($(3.0 , {(1.5)*sqrt(3)})$)--($(2.9 , {(1.5)*sqrt(3)})$)--($(2.9 , {(2-1/3)*sqrt(3)})$)
                     --($(2.5 , {(2-1/6)*sqrt(3)})$)--($(2.5 , {(2+1/6)*sqrt(3)})$)--($(2 , {(2+1/3)*sqrt(3)})$)
                     --($(1.5 , {(2+1/6)*sqrt(3)})$);
        \draw [->,thick]  ($(1.75 , {(2+1/4)*sqrt(3)})$)--($(1.7 , {(2+1/4+0.05)*sqrt(3)})$)
                               --($(2 , {(2+1/3+0.06)*sqrt(3)})$);
        \draw [->,thick]  ($(1.75 , {(1+1/4)*sqrt(3)})$)--($(1.7 , {(1+1/4+0.05)*sqrt(3)})$)
                               --($(2 , {(1+1/3+0.06)*sqrt(3)})$);
        \draw [->,thick]  ($(2.75 , {(1+1/4)*sqrt(3)})$) --($(2.8 , {(1+1/5)*sqrt(3)})$)
                       --($(3 , {(1+1/3-0.06)*sqrt(3)})$)--($(3.5 , {(1+1/6-0.02)*sqrt(3)})$)
                       --($(4 , {(1+1/3)*sqrt(3)})$) --($(4 , {(1+2/3-0.05)*sqrt(3)})$);

        \draw [->,thick]  ($(3.22 , {(1+1/5+0.02)*sqrt(3)})$)--($(3.26 , {(1+1/5+0.06)*sqrt(3)})$)
                       --($(3.1 , {(1+1/3-0.03)*sqrt(3)})$) --($(3.1 , {(1+2/3-0.03)*sqrt(3)})$) 
                       --($(3.5 , {(2-1/6-0.06)*sqrt(3)})$);
        \draw [->,thick]  ($(2.25 , {(3/4)*sqrt(3)})$)--($(2.3 , {(2/4+1/5)*sqrt(3)})$)
        --($(2.6 , {(1-0.21)*sqrt(3)})$)--($(2.6 , {(1+1/7)*sqrt(3)})$);

\draw [->,thick]   ($(3.25 , {(1.75)*sqrt(3)})$)--($(3.21 , {(1.8)*sqrt(3)})$)
                       --($(3.5 , {(2-1/6+0.07)*sqrt(3)})$)--($(4 , {(2-1/3+0.02)*sqrt(3)})$)
                       --($(4.5 , {(2-1/6)*sqrt(3)})$)--($(5 , {(2-1/3)*sqrt(3)})$)
                       --($(5.5 , {(2-1/6)*sqrt(3)})$)--($(5.5 , {(2+1/6)*sqrt(3)})$);
 \draw [->,thick]  ($(3.78 , {(2-1/4+0.02)*sqrt(3)})$)--($(3.82 , {(2-1/4+0.065)*sqrt(3)})$)
                       --($(3.5 , {(2-1/6+0.12)*sqrt(3)})$)--($(3.5 , {(2+1/6)*sqrt(3)})$)
                       --($(3 , {(2+1/3)*sqrt(3)})$)--($(3 , {(2+2/3)*sqrt(3)})$)
                       --($(3.4 , {(2+2/3+1/6-0.02)*sqrt(3)})$);
                       
\draw [->,thick]   ($(4.75 , {(2-1/4)*sqrt(3)})$)--($(4.785 , {(2-1/4+0.035)*sqrt(3)})$)
                       --($(4.55 , {(2-1/6+0.03)*sqrt(3)})$)--($(4.55 , {(2+1/6)*sqrt(3)})$)
                       --($(5 , {(2+1/3)*sqrt(3)})$);
                        
\newcommand{\hexagon}[2]{
\node[red] at (#1,{#2*sqrt(3)/2}) {$\circ$};
\draw [thick, yshift=0.0cm,xshift=0cm] 
({-0.5+#1},{-sqrt(3)/2 +#2*sqrt(3)/2})--
({0.5+#1},{-sqrt(3)/2 +#2*sqrt(3)/2})
--({1+#1},{0 +#2*sqrt(3)/2})--({0.5+#1},{sqrt(3)/2 +#2*sqrt(3)/2})--
({-0.5+#1},{sqrt(3)/2 +#2*sqrt(3)/2})--
({-1+#1},{0 +#2*sqrt(3)/2})
--({-0.5+#1},{-sqrt(3)/2 +#2*sqrt(3)/2});
}

\hexagon{2}{2};
\hexagon{2}{4};
\hexagon{3.5}{3};
\hexagon{3.5}{5};
\hexagon{3.5}{7};
\hexagon{5}{4};

\end{tikzpicture}
\end{center}
 \subcaption{Alcove walk $\color{red} p_{\lambda}$ from $\fA$ to $\fA+\lambda$\\
 $\color{blue}p_w=c^{-}_2 c^{-}_1$
,
$\color{red} p_\lambda=c^{+}_{0} c^{+}_{2} c^{+}_{1} c^{+}_{0} c^{+}_{2} c^{+}_{0}$}
\end{minipage}\hspace{1cm}
\begin{minipage}[b]{6cm}
\newcommand*\rows{10}
\begin{center}
\begin{tikzpicture}[scale=0.95]
    \clip(0.6,0.5) rectangle (7.2, {4.3*sqrt(3)});
    \foreach \row in {0, 1, ...,\rows} {
        \draw[gray] ($\row*(0, {0.5*sqrt(3)})$) -- ($(\rows,0)+\row*(0, {0.5*sqrt(3)})$);};
    \foreach \row in {0,1,...,4,5}{
        \draw[gray] ($\row*(1, 0)$) -- ($(5, {\rows*0.5*sqrt(3)})+\row*(1.0,{0.0*sqrt(3)})$);};        
   \foreach \row in {5,...,\rows}{
        \draw[gray] ($\row*(1, 0)$) -- ($(\row-5,{\rows*0.5*sqrt(3)})$);};
 
     \foreach \row in {1,2,3,4}{
        \draw[gray] ($\row*(1, 0)$) -- ($(0 ,{\row*sqrt(3)})$);};       

     \foreach \row in {1,2,3,4}{
        \draw[gray] ($(\row, {5*sqrt(3)})$) -- ($(0 ,{(5-\row)*sqrt(3)})$);};  

          \foreach \row in {1,2,3,4}{
        \draw[gray] ($({\rows-\row}, 0)$) -- ($(\rows ,{\row*sqrt(3)})$);};  

           \foreach \row in {1,2,3,4}{
        \draw[gray] ($({\rows-\row},  {5*sqrt(3)})$) -- ($(\rows ,{(5-\row)*sqrt(3)})$);};  

\newcommand{\trig}[2]{
\draw [lightgray,fill] ( #1,#2)--({#1+0.5},{#2+sqrt(3)/2-0.05})--({#1-0.5},{#2+sqrt(3)/2-0.05})--cycle;}
\trig{4}{3.5};
\trig{4.5}{0.9};

\foreach \x/\y in {0/1,0/-1, 0.75/0.5,-0.75/-0.5,-0.75/0.5, 0.75/-0.5,-0.75/1.5}
{\node[gray] at ({\x+4},{(\y+4)*sqrt(3)/2}) {\tiny 0};
\node[gray] at ({\x+4},{(\y+6)*sqrt(3)/2}) {\tiny 0};
};

\foreach \x/\y in {0/1,0/-1, 0.75/0.5,-0.75/-0.5,-0.75/0.5, 0.75/-0.5}
{\node[gray] at ({\x+5},{(\y+4)*sqrt(3)/2}) {\tiny 1};
\node[gray] at ({\x+5},{(\y+6)*sqrt(3)/2}) {\tiny 1};
};
\foreach \x/\y in {0/0,0/2,1.5/1,1.5/-1}
{\node[gray] at ({\x+4.5},{(\y+4)*sqrt(3)/2}) {\tiny 2};
\node[gray] at ({\x-0.25+4},{(\y+4)*sqrt(3)/2+sqrt(3)/4}) {\tiny 2};
\node[gray] at ({\x-0.25+4},{(\y+4)*sqrt(3)/2-sqrt(3)/4}) {\tiny 2};
};

\foreach \x/\y in {-1.5/0,-1.5/-2,1.5/0,1.5/2}
{\node[gray] at ({\x+4},{(\y+4)*sqrt(3)/2}) {\tiny 0};};
\foreach \x/\y in {-2.75/0.5,-2.75/-0.5,-2.75/-2.5,-2.75/-1.5,-1.25/-1.5}
{\node[gray] at ({\x+4.5},{(\y+4)*sqrt(3)/2}) {\tiny 0};};

\foreach \x/\y in {-0.9/2,-0.9/0}
{\node[gray] at ({\x+4.5},{(\y+4)*sqrt(3)/2}) {\tiny 1};};

\foreach \x/\y in {-1.75/0.5,-1.75/-0.5,-1.75/1.5,-1.75/-1.5,-0.25/-1.5}
{\node[gray] at ({\x+4.5},{(\y+4)*sqrt(3)/2}) {\tiny 1};};

\foreach \x/\y in {-1.5/1,-1.5/3,-1.5/-1}
{\node[gray] at ({\x+4.5},{(\y+4)*sqrt(3)/2}) {\tiny 2};};
\foreach \x/\y in {-2.25/0.5,-2.25/-0.5,-2.25/-2.5,-2.25/-1.5,-0.75/-1.5,-0.75/3.5}
{\node[gray] at ({\x+4.5},{(\y+4)*sqrt(3)/2}) {\tiny 2};};

\node[gray] at ({4.5},{(2)*sqrt(3)/2}) {\tiny 2};

\draw [brown,thick] ($(2,0)$)--($(6.5, {(4.5)*sqrt(3)})$);
\draw [brown,thick] ($(6,0)$)--($(1.5, {(4.5)*sqrt(3)})$);
\draw [brown,thick] ($(0,{2*sqrt(3)})$)--($(8, {(2*sqrt(3)})$);

\node at (6.3,7.2) {$H_{\alpha_1,0}$};  
\node at (1.9,7.2) {$H_{\alpha_2,0}$}; 
\node at (6.4,3.47) {$H_{\alpha_1+\alpha_2,0}$}; 

\coordinate (a) at  (4,{2*sqrt(3)});
\draw[name path=circleA,thick] (a) circle (0.07);

\draw [->,thick,gray] (4,{2*sqrt(3)})--(3.5,{2.5*sqrt(3)});
\draw [->,thick,gray] (4,{2*sqrt(3)})--(4.5,{2.5*sqrt(3)});
\node at (3.3,4.5) {$\varpi_1$};
\node at (4.3,4.5) {$\varpi_2$};

\node at (4.05,3.9) {\tiny$\fA$};      
\node at (4.5,1.2) {\tiny $\fA-\lambda$};  

\node at (5.5,5.0) {\tiny$(0)$};
\node at (4.8,4.5) {\tiny$(1)$};
\node at (3.4,3.7) {\tiny$(2)$};
\node at (1.9,2.8) {\tiny$(3)$};
\node at (4.5,3.0) {\tiny$(8)$};
\node at (5.5,3.95) {\tiny$(4)$};
\node at (5.5,2.3) {\tiny$(5)$};
\node at (4.15,1.5) {\tiny$(6)$};
\node at (3,2.1) {\tiny$(7)$};

\draw [->,thick,red]  ($(4 , {(2+1/3+0.02)*sqrt(3)})$)--($(3.5 , {(2+1/6)*sqrt(3)})$)
                   --($(3.5 , {(2-1/6)*sqrt(3)})$)--($(4 , {(2-1/3)*sqrt(3)})$)
                  --($(4 , {(1+1/3)*sqrt(3)})$)--($(4.5 , {(1+1/6)*sqrt(3)})$)
                  --($(4.5 , {(1-1/6)*sqrt(3)})$); 

\draw [->,thick,green]  ($(4 , {(2+1/3)*sqrt(3)})$)--($(3.55 , {(2+1/6)*sqrt(3)})$)
                                         --($(3.55 , {(2-1/6)*sqrt(3)})$);
\draw [->,thick,blue]  ($(3.55 , {(2-1/6)*sqrt(3)})$)--($(4.0 , {(2-1/3+0.02)*sqrt(3)})$)
                    --($(4.5 , {(2-1/6)*sqrt(3)})$) --($(4.5 , {(2+1/6)*sqrt(3)})$)
                    --($(5 , {(2+1/3)*sqrt(3)})$)--($(5 , {(2+2/3)*sqrt(3)})$)
                    --($(5.5 , {(3-1/6)*sqrt(3)})$); 
   \draw [->,thick] ($(5.25 , {(2.5+1/4)*sqrt(3)})$)--($(5.2 , {(2.5+1/4+0.08)*sqrt(3)})$)
   --($(4.95 , {(2+2/3+0.08)*sqrt(3)})$);
   \draw [->,thick]  ($(4.75 , {(2+1/4)*sqrt(3)})$)--($(4.7 , {(2+1/4+0.05)*sqrt(3)})$)
   --($(4.5 , {(2+1/6+0.06)*sqrt(3)})$)--($(4 , {(2+1/3+0.08)*sqrt(3)})$)
   --($(3.5 , {(2+1/6+0.08)*sqrt(3)})$);
    \draw [->,thick]  ($(4.3 , {(2+1/4+0.05)*sqrt(3)})$)--($(4.25 , {(2+1/4)*sqrt(3)})$)
    --($(4.43 , {(2+1/6)*sqrt(3)})$)--($(4.41 , {(2-1/6)*sqrt(3)})$);
   \draw [->,thick]  ($(4.25 , {(2-1/4)*sqrt(3)})$)--($(4.3 , {(2-1/4-0.05)*sqrt(3)})$)
                         --($(4.0 , {(2-1/3-0.05)*sqrt(3)})$)--($(3.5 , {(2-1/6-0.06)*sqrt(3)})$)
                         --($(3 , {(2-1/3)*sqrt(3)})$)--($(2.5 , {(2-1/6)*sqrt(3)})$)
                         --($(2 , {(2-1/3)*sqrt(3)})$);
    \draw [->,thick] ($(2.75 , {(2-1/4)*sqrt(3)})$)--($(2.71 , {(2-1/4-0.035)*sqrt(3)})$)
    --($(3 , {(2-1/3-0.05)*sqrt(3)})$)--($(3 , {(1+1/3-0.05)*sqrt(3)})$);
                         
  \draw [->,thick]  ($(3.7 , {(2-1/4-0.05)*sqrt(3)})$)--($(3.66 , {(2-1/4-0.08)*sqrt(3)})$)
  --($(4-0.05 , {(2-1/3-0.08)*sqrt(3)})$)--($(4-0.05 , {(2-2/3-0.05)*sqrt(3)})$)
  --($(4.45-0.03 , {(1+1/6-0.03)*sqrt(3)})$)--($(4.45-0.03 , {(1-1/6)*sqrt(3)})$);
                         
 \draw [->,thick] ($(5 , {(2+1/2)*sqrt(3)})$)--($(5.1 , {(2+1/2)*sqrt(3)})$)
 --($(5.1 , {(2+1/3-0.05)*sqrt(3)})$)--($(5.5 , {(2+1/6)*sqrt(3)})$);
                                
  \draw [->,thick]  ($(4.5 , {(2)*sqrt(3)})$) --($(4.6 , {(2)*sqrt(3)})$)
  --($(4.6 , {(2-1/6-0.05)*sqrt(3)})$)--($(5 , {(2-1/3)*sqrt(3)})$)
  --($(5 , {(2-2/3)*sqrt(3)})$)--($(5.5 , {(1+1/6)*sqrt(3)})$);

\newcommand{\hexagon}[2]{
\node[red] at (#1,{#2*sqrt(3)/2}) {$\circ$};
\draw [thick, yshift=0.0cm,xshift=0cm] 
({-0.5+#1},{-sqrt(3)/2 +#2*sqrt(3)/2})--
({0.5+#1},{-sqrt(3)/2 +#2*sqrt(3)/2})
--({1+#1},{0 +#2*sqrt(3)/2})--({0.5+#1},{sqrt(3)/2 +#2*sqrt(3)/2})--
({-0.5+#1},{sqrt(3)/2 +#2*sqrt(3)/2})--
({-1+#1},{0 +#2*sqrt(3)/2})
--({-0.5+#1},{-sqrt(3)/2 +#2*sqrt(3)/2});
}

\hexagon{3}{2};
\hexagon{3}{4};
\hexagon{1.5}{3};
\hexagon{4.5}{1};
\hexagon{4.5}{3};
\hexagon{4.5}{5};
\hexagon{6}{2};
\hexagon{6}{4};
\hexagon{6}{6};

\end{tikzpicture}
\end{center}

 \subcaption{Alcove walk $\color{red} p_{-\lambda}$ from $\fA$ to $\fA-\lambda$\\
 $\color{blue}p_w=c^{-}_2 c^{-}_1$,
 $\color{red} p_{-\lambda}=c^{-}_{2} c^{-}_{1} c^{-}_{2} c^{-}_{0} c^{-}_{1} c^{-}_{2}$}

\end{minipage}

\end{figure}

We first calculate $c_{u,\mu}^{w,\lambda}$ according to the formula 
\Cref{thm:lambda-chain2}.
In the proof we introduced $\lambda$-shifted (reversed) hyperplane sequence
 which can be seen in diagram (a).
$$h'_1=H_{\alpha_1+\alpha_2,1},h'_2=H_{\alpha_1,1},
h'_3=H_{\alpha_1+\alpha_2,2},h'_4=H_{\alpha_1,2},
,h'_5=H_{\alpha_1+\alpha_2,3},h'_6=H_{\alpha_2,1}.$$

According to the formula, we need to choose $J\subset \{ 1,2,\ldots, l \}$ such that 
$u\stackrel{J_{<}}{\longrightarrow} w$.

If $u=s_1$, $J=\{6\},\{4\},\{2\}$ as $s_{\alpha_1+\alpha_2}=s_1 s_2 s_1$.
For the weight $\mu$ , when $J=\{6\}$, we need to calculate $\mu=w \tilde{r}_{h_6}(\lambda)$.
But as is explained in the proof, $\tilde{r}_{h_6}=\hat{r}_{h'_1}$, so,
$\mu=w \hat{r}_{h'_1}(\lambda)=w (-\varpi_2)=-\varpi_1+\varpi_2$.
Likewise when $J=\{4\}$, $\mu=w \tilde{r}_{h_4}(\lambda)=w \hat{r}_{h'_3}(\lambda)=-\varpi_2$,
and when $J=\{2\}$, $\mu=w \tilde{r}_{h_2}(\lambda)=w \hat{r}_{h'_5}(\lambda)=\varpi_1-3\varpi_2$.
If $u=id$,  there are two possible Bruhat chains $u=id< s_1<s_2 s_1=w$ and 
$u=id< s_2<s_2 s_1=w$. For the first case, $J=\{5,6\}, \{3,6\},\{3,4\}$, and for the second case,
$J=\{1,5\},\{1,3\}$. 
When $J=\{5,6\}$, 
$\mu=w \tilde{r}_{h_6} \tilde{r}_{h_5} (\lambda)
=w \hat{r}_{h'_1} \hat{r}_{h'_2} (\lambda)=w \hat{r}_{h'_1}(2\varpi_2)=w(-\varpi_1+\varpi_2)=\varpi_1$.
All the possibilities and the corresponding (folded) alcove walks $\tilde{p}_{\lambda}$ are listed in the  table below.
(We can also see the bijection of \Cref{lemma:paths-bij}, cf. diagram (a).)
$$
\begin{array}{|c|c|c|c|c|c|}
\hline
p_w \tilde{p}_{\lambda}
&\tilde{p}_{\lambda}&
\mathcal{M}
&J
& \varphi(p)=u
&{\rm wt}(p)=w \tilde{r}^\lambda_{J_{>}}(\lambda)\\
\hline
\hline
(0)& {\color{blue}c^{-}_{0}c^{-}_{2}c^{-}_{1}c^{-}_{0}c^{-}_{2}c^{+}_{0}} & \{ \}&  \{ \}&s_2 s_1& \varpi_1-3\varpi_2\\
\hline
(1)& f^{+}_{0}c^{-}_{2}c^{+}_{1}c^{-}_{0}c^{+}_{2}c^{+}_{0} &\{ h_1\}& \{ 6\}&s_1 & -\varpi_1+\varpi_2\\
(2)& {\color{blue}c^{-}_{0}c^{-}_{2}}f^{+}_{1}c^{-}_{0}c^{+}_{2}c^{+}_{0} &\{ h_3\}&  \{ 4\}&s_1& -\varpi_2\\
(3)& {\color{blue}c^{-}_{0}c^{-}_{2}c^{-}_{1}c^{-}_{0}}f^{+}_{2}c^{+}_{0} &\{ h_5\}& \{ 2\}& s_1& \varpi_1-3\varpi_2\\ 
\hline
(4)& {\color{blue}c^{-}_{0}}f^{+}_{2}c^{+}_{1}c^{+}_{0}c^{+}_{2}c^{-}_{0} &\{ h_2\}& \{ 5\}& s_2&  2\varpi_1-2\varpi_2\\
(5)& {\color{blue}c^{-}_{0}c^{-}_{2}c^{-}_{1}}f^{+}_{0}c^{+}_{2}c^{-}_{0} &\{ h_4\}& \{ 3\}& s_2& \varpi_1-3\varpi_2\\
\hline
(6)& f^{+}_{0}f^{+}_{2}c^{+}_{1}c^{+}_{0}c^{+}_{2}c^{+}_{0} &\{ h_1, h_2\}& \{5, 6\}& id&\varpi_1 \\
(7)& f^{+}_{0}c^{-}_{2}c^{+}_{1}f^{+}_{0}c^{+}_{2}c^{+}_{0} &\{h_1, h_4 \}& \{ 3,6\}& id&-\varpi_1+\varpi_2\\
(8)& {\color{blue}c^{-}_{0}}f^{+}_{2}c^{+}_{1}c^{+}_{0}c^{+}_{2}f^{+}_{0} &\{ h_2, h_6\}& \{ 1,5\}& id& 2\varpi_1-2\varpi_2\\
(9)& {\color{blue}c^{-}_{0}c^{-}_{2}}f^{+}_{1}f^{+}_{0}c^{+}_{2}c^{+}_{0} &\{h_3, h_4 \}& \{ 3,4\}& id&-\varpi_2\\
(10)& {\color{blue}c^{-}_{0}c^{-}_{2}c^{-}_{1}}f^{+}_{0}c^{+}_{2}f^{+}_{0} &\{ h_4, h_6\}& \{1,3\}& id& \varpi_1-3\varpi_2\\
\hline
\end{array}
$$

By this table,  we get  all the coefficients $c^{w,\lambda}_{u,\mu}$ as follows.
$$
\begin{minipage}{12cm}
$c^{w,\lambda}_{s_2 s_1,\mu}=1 \text{ for } \mu= \varpi_1-3\varpi_2$,

$c^{w,\lambda}_{s_1,\mu}= (q-1) q^{-1} \text{ for } \mu=-\varpi_1+\varpi_2, -\varpi_2, \varpi_1-3\varpi_2$,

$c^{w,\lambda}_{s_2,\mu}= (q-1) q^{-1} \text{ for } \mu=2\varpi_1-2 \varpi_2,\varpi_1 -3\varpi_2$,

$c^{w,\lambda}_{id,\mu}=(q-1)^2 q^{-2}  \text{ for } \mu=\varpi_1,-\varpi_1+\varpi_2,2\varpi_1-2\varpi_2,
 -\varpi_2, \varpi_1-3\varpi_2$.
\end{minipage}
$$

Next we calculate $c^{w,-\lambda}_{u,\mu}$ according to \Cref{thm:lambda-chain1}.
For this case we need to chose
$J\subset \{ 1,2,\ldots, l \}$ such that 
$u\stackrel{J_{>}}{\longrightarrow} w$.

For example $u=id$, then there are three possible $J$ 
corresponding to the Bruhat chains
$u<u s_{\beta_3} <u  s_{\beta_3}  s_{\beta_2}=w$,
 $u<u s_{\beta_5} <u  s_{\beta_5}  s_{\beta_2}=w$,
$u<u s_{\beta_5} <u  s_{\beta_5}  s_{\beta_4}=w$.
For the first case the weight $\mu$ can be calculated as (using diagram (b)),
$$\mu=w \hat{r}_{J_{<}}(-\lambda)=s_2 s_1  \hat{r}_{h_2}  \hat{r}_{h_3}(-\lambda)=
s_2 s_1(3\varpi_1-2\varpi_2)=-2\varpi_1-\varpi_2.$$

All the possible $J$ and corresponding alcove walks $\tilde{p}_{-\lambda}$ are listed in the table below.
$$
\begin{array}{|c|c|c|c|c|c|}
\hline
p_w \tilde{p}_{-\lambda}
&\tilde{p}_{-\lambda}& 
\mathcal{M}&J
& \varphi(p)=u
&wt(p)=w \hat{r}_{J_{<}}(-\lambda)\\
\hline
\hline
(0)& {\color{blue}c^{-}_{2}c^{+}_{1}c^{+}_{2}c^{+}_{0}c^{+}_{1}c^{+}_{2}} & \{ \}& \{ \}& s_2 s_1& -\varpi_1+3\varpi_2\\
\hline
(1)& {\color{blue}c^{-}_{2}c^{+}_{1}c^{+}_{2}c^{+}_{0}c^{+}_{1}}f^{-}_{2} &\{ h_6\}& \{ 6\}& s_1 & \varpi_2\\
(2)& {\color{blue}c^{-}_{2}c^{+}_{1}c^{+}_{2}}f^{-}_{0}c^{+}_{1}c^{-}_{2} &\{ h_4\}&  \{4 \}&s_1&\varpi_1 -\varpi_2\\
(3)& {\color{blue}c^{-}_{2}}f^{-}_{1}c^{+}_{2}c^{-}_{0}c^{+}_{1}c^{-}_{2} &\{ h_2\}& \{ 2\}& s_1& 2\varpi_1-3\varpi_2\\ 
\hline
(4)& {\color{blue}c^{-}_{2}c^{+}_{1}c^{+}_{2}c^{+}_{0}}f^{-}_{1}c^{-}_{2} &\{ h_5\}&  \{ 5\}&s_2&  -2\varpi_1+2\varpi_2\\
(5)& {\color{blue}c^{-}_{2}c^{+}_{1}}f^{-}_{2}c^{-}_{0}c^{-}_{1}c^{-}_{2} &\{ h_3\}& \{3 \}& s_2& -3\varpi_1+\varpi_2\\
\hline
(6)& {\color{blue}c^{-}_{2}}f^{-}_{1}f^{-}_{2}c^{-}_{0}c^{-}_{1}c^{-}_{2} &\{ h_2, h_3\}& \{ 2,3\}& id&-2\varpi_1-\varpi_2 \\
(7)& {\color{blue}c^{-}_{2}}f^{-}_{1}c^{+}_{2}c^{-}_{0}f^{-}_{1}c^{-}_{2} &\{h_2, h_5 \}&  \{2,5 \}&id&-2\varpi_2\\
(8)& {\color{blue}c^{-}_{2}c^{+}_{1}c^{+}_{2}}f^{-}_{0}f^{-}_{1}c^{-}_{2} &\{ h_4, h_5\}& \{4,5 \}& id& -\varpi_1\\
\hline
\end{array}
$$

By this table, we get all the coefficients $c^{w,-\lambda}_{u,\mu}$ as  follows.
$$
\begin{minipage}{12cm}
$c^{w,-\lambda}_{s_2 s_1,\mu}=1 \text{ for } \mu= -\varpi_1+3\varpi_2$,

$c^{w,-\lambda}_{s_1,\mu}= (1-q) q^{-1} \text{ for } \mu=\varpi_2, \varpi_1-\varpi_2, 2\varpi_1-3\varpi_2$,

$c^{w,-\lambda}_{s_2,\mu}= (1-q) q^{-1} \text{ for } \mu=-2\varpi_1+2 \varpi_2,-3\varpi_1 +\varpi_2$,

$c^{w,-\lambda}_{id,\mu}=(1-q)^2 q^{-2}  \text{ for } \mu=-2\varpi_1-\varpi_2, -2\varpi_2, -\varpi_1$.
\end{minipage}
$$
\bibliographystyle{halpha}
\bibliography{che.bib}

\end{document}